\documentclass[a4paper,11pt,reqno]{amsart}

\usepackage{color}
\usepackage{graphicx}
\usepackage{caption}
\usepackage[english]{babel}
\usepackage{lmodern}
\usepackage[T1]{fontenc}
\usepackage[utf8]{inputenc}
\usepackage{csquotes}
\usepackage{verbatim}
\usepackage{todonotes}
\usepackage{amsmath}
\usepackage{amsthm}
\usepackage{amsfonts, amssymb, fancybox, multicol, array, tabularx, amscd}
\usepackage{enumerate}
\usepackage{upgreek}
\usepackage[shortlabels]{enumitem}
\usepackage{psfrag}
\newtheorem{theorem}{Theorem}[section]
\newtheorem{lemma}[theorem]{Lemma}
\newtheorem{proposition}[theorem]{Proposition}
\newtheorem{corollary}[theorem]{Corollary}

\newtheorem{remark}[theorem]{Remark}
\newtheorem{definition}[theorem]{Definition}

\numberwithin{equation}{section}

\numberwithin{equation}{section}
\numberwithin{figure}{section}

\newcommand{\esssup}{\mathrm{supess}}

\newcommand{\p}{\textbf{p}}
\newcommand{\q}{\textbf{q}}

\newcommand{\dist}{\mathrm{dist}}

\newcommand{\R}{\mathbb{R}}
\newcommand{\diver}{\mathrm{div}}
\newcommand{\E}{\partial^{*}E}
\newcommand{\h}{\mathcal{H}}

\newcommand{\Upp}{^{\vee}}
\newcommand{\Low}{^{\wedge}}
\newcommand{\aplim}{\mathrm{aplim}}
\newcommand{\mres}{\mathbin{\vrule height 1.6ex depth 0pt width
0.13ex\vrule height 0.13ex depth 0pt width 1.3ex}}

\def\XXint#1#2#3{{\setbox0=\hbox{$#1{#2#3}{\int}$ }
\vcenter{\hbox{$#2#3$ }}\kern-.6\wd0}}

 \title[]{Rigidity of Steiner's inequality for the anisotropic perimeter}
 
 \author[]
 {Matteo Perugini}
 \address[Matteo Perugini]{Universit\"at M\"unster, Applied Mathematics, Eisteinstr. 62, D-48149 M\"unster}
 \email{matteo.perugini@uni-muenster.de}

\begin{document}
\begin{abstract}
The aim of this work is to study the rigidity problem for Steiner's inequality for the anisotropic perimeter, that is, the situation in which the only extremals of the inequality are vertical translations of the Steiner symmetral that we are considering. Our main contribution consists in giving conditions under which rigidity in the anisotropic setting is equivalent to rigidity in the Euclidean setting. Such conditions are given in term of a restriction to the possible values of the normal vectors to the boundary of the Steiner symmetral (see Corollary \ref{crystal clear 1}, and Corollary \ref{crystal clear 2}). 
\end{abstract}

\maketitle
\tableofcontents

\section{Introduction}

\subsection{Overview}
The characterization of the geometric properties of minimizers of variational problems can be in general a delicate thing to achieve. The study of perimeter inequalities under symmetrisation, and in particular the study of rigidity for such inequalities, is a good way to possibly provide tools in order to show symmetries of the minimizers of the problem under consideration. Steiner's symmetrisation is a classical and powerful example of symmetrisation, that has been often used in the analysis of geometric variational problems. For instance, De Giorgi in his proof of the very celebrated Euclidean isoperimetric theorem \cite{DeGiorgi}, used Steiner's inequality (see (\ref{eq: Classic Steiner inequality})) to show that the minimum for the Isoperimetric Problem is a convex set. After De Giorgi, in the seminal paper \cite{CCF}, Chleb\'ik, Cianchi and Fusco discussed Steiner's inequality in the natural framework of sets of finite perimeter and provided sufficient condition for the rigidity of equality cases. In our context, by \emph{rigidity of equality cases} we mean the situation in which equality cases are solely obtained in correspondence of translations of the Steiner's symmetral.
Then, the characterization of the rigidity of equality cases was resumed by Cagnetti, Colombo, De Philippis and Maggi in their work presented in \cite{CCPMSteiner}. There, they managed to fully  characterize the equality cases for the Steiner's inequality and obtain further important results for the \emph{rigidity problem}. 

\noindent
Concerning \emph{rigidity} of equality cases, let us mention two results that were obtained in different settings from the one just described. First, in the framework of Gaussian perimeter, again Cagnetti, Colombo, De Philippis and Maggi managed to prove a complete characterization result of rigidity of equality cases for
Ehrhard’s symmetrisation inequality (see \cite{CCPMGauss}). The other result, is the characterization of rigidity for the Euclidean perimeter inequality under the spherical symmetrisation (see \cite{cagneti_perugini_stoeger}). For a recent survey on rigidity results for perimeter inequality under symmetrisation see also \cite{portugalia}.\\
The main goal of this work is to characterize rigidity of the equality cases of the Steiner's inequality for the anisotropic perimeter. Our main contribution is presented in Proposition \ref{final sum up}, that provides sufficient conditions to get that rigidity of equality cases of the Steiner's inequality in the Euclidean setting coincides with rigidity of equality cases of the Steiner's inequality in the anisotropic setting (see Corollary \ref{crystal clear 1}, and Corollary \ref{crystal clear 2}). In the remaining part of this introduction, we will introduce some notation and state our main results (see in particular Section \ref{subsec: 1.7}). 

The remaining part of the paper is organized as follows: 
In Section \ref{preliminaries} we recall some basic notions of geometric measure theory and we introduce some useful notation. In Section \ref{section 3} we focus our attention on the properties of the surface tension $\phi_K$ (see (\ref{def surface tension})). In particular, we characterize the cases of additivity for the function $\phi_K$ (see Proposition \ref{prop:linearityphi}), and we prove other intermediate results that will be used in the proof of our main results about rigidity. In Section \ref{section 4} we prove a characterization result for the anisotropic total variation (see Definition \ref{def: anisotropic total variation}). Such result (see Theorem \ref{thm:Charac.anis.total.var.}) will play an important role in Section \ref{section 5}. In Section \ref{section 5} we prove a formula to compute the anisotropic perimeter for some classes of sets $E\subset\R^n$ having finite perimeter, and whose vertical
sections are segments (see Corollary \ref{cor: formula perimetro con b v}, and Corollary \ref{cor: formula per F[v]}). With these results at hands, in Section \ref{section 6} we prove the first of our main results, namely the characterization of equality cases for the anisotropic perimeter inequality under Steiner's symmetrisation (see Theorem \ref{thm:2.2 pag 117}). Lastly, in Section \ref{section Rigidity} we prove the other main results about rigidity, namely Theorem \ref{thm: rigidity}, Proposition \ref{prop: R1,R2}, and Corollary \ref{cor: smooth K^s}.

\subsection{Basic notions on sets of finite perimeter}\label{subsec: 1.2}
For every $r > 0$ and $x \in \R^n$, we denote by $B (x, r)$ the open ball of $\R^n$
with radius $r$ centred at $x$. In the special case $x = 0$, we set $B(r):= B(0, r)$.
%Let $n,k\in\mathbb{N}$, and $\delta>0$. The $k$-dimensional Hausdorff measure of step $\delta$ of a set $E\subset \R^n$ is defined as
%\begin{align*}
%\h^k_{\delta}(E):= \inf_{\mathcal{F}}\sum_{F\in\mathcal{F}}\omega_k\frac{\diam(F) }{2}^k,
%\end{align*}  
%where $\mathcal{F}$ is a countable covering of $E$ by sets $F\subset \R^n$ such that $\diam(F)<\delta$, and $\omega_k=\mathcal{L}^k(B(1))$ (where $B(1)$ is the unitary open ball in $\R^k$). The {\it $k$-dimensional Hausdorff measure} of $E$ is then 
%$$
%\h^k(E):=\sup_{\delta>0}\h^k_{\delta}(E)=\lim_{\delta\to 0^+}\h^k_{\delta}(E).
%$$
Let $E \subset \R^n$ be a measurable set, and let $t \in [0, 1]$.
We denote by $E^{(t)}$ the set of points of density $t$ of $E$, given by 
$$
E^{(t)} : = \left\{ x \in \R^n : \lim_{\rho \to 0^+ } 
\frac{\mathcal{H}^n (E \cap B(x,\rho))}{\omega_n \rho^n} = t \right\},
$$
where here and in the following $\h^k$, $k\in \mathbb{N}$ with $0\leq k \leq n$, stands for $k$-dimensional Hausdorff measure, and $\omega_n=\mathcal{H}^n(B(1))$.
We then define the essential boundary of $E$ as 
$$
\partial^{\textnormal{e}} E:= E \setminus (E^{(1)} \cup E^{(0)}).
$$
Let $G \subset \R^n$ be any Borel set. We define the perimeter of $E$
relative to $G$ as the extended real number given by
$$
P (E; G) := \mathcal{H}^{n-1} (\partial^{\textnormal{e}} E \cap G) \in [0, \infty], 
$$
and the perimeter of $E$ as $P(E):= P (E; \R^n)$.
When $E$ is a set with smooth boundary, it turns out that 
$\partial^{\textnormal{e}} E = \partial E$, and the perimeter of $E$
agrees with the usual notion of $(n-1)$-dimensional measure of $\partial E$. If $P(E;C) < \infty$ for every compact set $C\subset \R^n$, $E$ is called a set of \emph{locally finite perimeter} and we can define the reduced boundary $\partial^*E$ of $E$. This has the property that $\partial^*E \subset \partial^{\textnormal{e}} E$, 
$\mathcal{H}^{n-1} (\partial^{\textnormal{e}} E \setminus \partial^*E) = 0$, 
and is such that for every $x \in \partial^* E$ there exists the measure theoretic outer unit normal $\nu^E (x)$ to $E$ at $x$ (see Section~\ref{preliminaries}).
\subsection{Steiner's inequality}\label{subsec: 1.3} We decompose $\R^n$, $n\geq 2$, as the Cartesian product $\R^{n-1}\times \R$. Then, for every $x=(x_1,\dots,x_n)\in\R^n$ we will write $x=(\p x,\q x)$, where $\p x= (x_1,\dots, x_{n-1})$, and $\q x= x_n$ are the "horizontal" and "vertical" projections, respectively. Given a Lebesgue measurable function $v: \R^{n-1}\rightarrow [0,\infty ]$, we say that a Lebesgue measurable set $E\subset \R^n$ is $v$\emph{-distributed} if, denoting by $E_z$ its vertical section with respect to $z\in \R^{n-1}$, that is 
\begin{align}\label{la striscia verticale}
E_z:=\left\{t\in \R:\, (z,t)\in E  \right\},\quad z\in \R^{n-1},
\end{align} 
we have that 
\begin{align*}
v(z)=\h^1(E_z),\quad \textit{\emph{for $\h^{n-1}$-a.e. $z\in\R^{n-1}$}}.
\end{align*}
\noindent
Among all $v$-distributed sets, we denote by $F[v]$ the only one (up to $\mathcal{H}^n$ negligible modifications) that is symmetric by reflection with respect to $\{\q x=0  \}$, and whose vertical sections are segments, that is 
\begin{align}\label{eq:F[v]}
F[v]:= \left\{x\in \R^n:\, |\q x| < \frac{v(\p x)}{2}  \right\}.
\end{align}
If $E$ is a $v$-distributed set, we define the Steiner symmetral $E^s$ of $E$ as $E^s:=F[v]$. Note that, if $v$ if Lebesgue measurable, then $F[v]$ is a Lebesgue measurable set. Furthermore, by Fubini's Theorem, Steiner symmetrisation preserves the volume, that is, if $E$ is a $v$-distributed set such that $\mathcal{H}^n(E)<\infty$, then $\h^n(E)=\h^n(F[v])$. A very important fact is that Steiner symmetrisation acts monotonically on the perimeter. More precisely, \emph{Steiner's inequality} holds true (see for instance \cite[Theorem 14.4]{MaggiBOOK}): if $E$ is a $v$-distributed set then
\begin{align}\label{eq: Classic Steiner inequality}
P(E;G\times \R)\geq P(F[v];G\times \R) \quad \textit{\emph{for every Borel set $G\subset \R^{n-1}$}}.
\end{align}

The next two results give the minimal regularity assumptions needed to study inequality (\ref{eq: Classic Steiner inequality}) (see \cite[Lemma 3.1]{CCF} and  \cite[Proposition 3.2]{CCPMSteiner}, respectively).
\begin{lemma}\emph{(Chleb\'ik, Cianchi and Fusco)}
Let $E$ be a $v$-distributed set of finite perimeter in $\R^n$, for some measurable function $v:\R^{n-1}\rightarrow [0,\infty]$. Then, one and only one of the following two possibilities is satisfied:
\begin{itemize}
\item[i)]$v(x')=\infty$ for $\h^{n-1}$-a.e. $x'\in\R^{n-1}$ and $F[v]$ is $\h^n$-equivalent to $\R^n$;
\medskip
\item[ii)]$v(x')<\infty$ for $\h^{n-1}$-a.e. $x'\in\R^{n-1}$, $\h^n(F[v])<\infty$, and $v\in BV(\R^{n-1})$,
\end{itemize}
where $BV(\R^{n-1})$ denotes the space of functions of bounded variation in $\R^{n-1}$ (see Section~2).
\end{lemma}

\begin{lemma}
Let $v:\R^{n-1}\rightarrow [0,\infty)$ be measurable. Then, we have $0<\h^n(F[v])<\infty$ and $P(F[v])<\infty$ if and only if 
\begin{align}\label{due tilde}
v\in BV(\R^{n-1}),\quad \textit{and}\quad
0<\h^{n-1}\left(\{v>0 \} \right)<\infty.
\end{align}

\end{lemma}

\subsection{Rigidity for Steiner's inequality}\label{subsec: 1.4}
Given $v$ as in (\ref{due tilde}) we set:
\begin{align}
\mathcal{M}(v):=\left\{E\subset \R^n:\, \textit{\emph{$E$ is $v$-distributed and $P(E)=P(F[v])$}}\right\}.
\end{align}
We say that \emph{rigidity} holds true for Steiner's inequality if the only elements of $\mathcal{M}(v)$ are ($\h^n$-equivalent to) vertical translations of $F[v]$, namely:
\begin{align}\tag{RS}\label{rigidity steiner}
E\in \mathcal{M}(v) \quad \Longleftrightarrow \quad \h^n(E\Delta (F[v]+te_n))=0\quad \textnormal{ for some $t\in\R$},
\end{align}
where $\Delta$ stands for the symmetric difference between sets, and  $e_1,\dots,e_n$ are the elements of the canonical basis of $\R^n$. 

A natural step in order to understand when (\ref{rigidity steiner}) holds true, is to study the set $\mathcal{M}(v)$. The characterization of equality cases in (\ref{eq: Classic Steiner inequality}) was first addressed by Ennio De Giorgi in \cite{DeGiorgi}, where he showed that any set $E\in\mathcal{M}(v)$ is such that 
\begin{align}\label{eq: Ez segment}
E_z \textit{\emph{ is $\h^1$-equivalent to a segment}},\quad \textit{\emph{for $\h^{n-1}$-a.e. $z\in \R^{n-1}$}},
\end{align}
(see also \cite[Theorem 14.4]{MaggiBOOK}). After that, further information about $\mathcal{M}(v)$ was given
by Chleb\'ik, Cianchi and Fusco (see \cite[Theorem 1.1]{CCF}). The study of equality cases in Steiner's inequality was then resumed by Cagnetti, Colombo, De Philippis and Maggi in \cite{CCPMSteiner}, where the authors give a complete characterization of elements of $\mathcal{M}(v)$ (see Theorem \ref{thm:1.2 Filippo} below). In order to explain their result, let us observe that any $v$-distributed set $E$ satisfying (\ref{eq: Ez segment}) is uniquely determined by the barycenter function 
%
%
%In order to get this result, the authors first observed  that the necessary condition for the equality in (\ref{eq: Classic Steiner inequality}) for a set $E$ satisfying (\ref{eq: Ez segment}), implies that we are concerned with sets $E$ of the form $E= \Sigma_{u_1}\cap \Sigma^{u_2}$ corresponding to Lebesgue measurable functions $u_1$ and $u_2$ such that $u_1\leq u_2$ on $\R^{n-1}$. Here, $\Sigma_{u_1}=\{ x\in \R^n:\, \q x>u_1(\p x) \}$ and $\Sigma^{u_2}=\{ x\in \R^n:\, \q x<u_2(\p x) \}$ are the epigraph of $u_1$ and the subgraph of $u_2$ respectively. Then, they focused on the fine properties of the barycenter function of a given measurable set $E\subset \R^n$. Given a $v$-distributed set $E$, we define the barycenter function of $E$,
$b_E:\R^{n-1}\rightarrow \R$, defined as: 
\begin{align}\label{def:barycenter of E}
b_E(z)=\begin{cases} \frac{1}{v(z)}\int_{E_z}t\,d\h^1(t) & \mbox{if }0<v(z)<\infty,\\ 0 & \mbox{otherwise. }
\end{cases}
\end{align}
Note that, if E satisfies (\ref{eq: Ez segment}), for every $z\in\{0<v<\infty\}$, $b_E$ represents the midpoint of $E_z$. In general, $b_E$ may fail to be a $BV$, or even an $L^1_{loc}$ function, even if $E$ is a set of finite perimeter (see \cite[Remark 3.5]{CCPMSteiner}). The optimal regularity for $b_E$, when $E$ satisfies (\ref{eq: Ez segment}), is given by the following result (see \cite[Theorem 1.7]{CCPMSteiner}). 
\begin{theorem}\label{thm: baricentro Filippo}
Let $v$ be as in (\ref{due tilde}), and let $E$ be a $v$-distributed set of finite perimeter satisfying (\ref{eq: Ez segment}). Then,
\begin{align*}
b_{\delta}=1_{\{v>\delta\}}\,b_E\in GBV(\R^{n-1}),
\end{align*}
for every $\delta>0$ such that $\{ v>\delta \}$ is a set of finite perimeter, where $1_{\{v>\delta\}}$ stands for the characteristic function of the set $\{v>\delta\}$. Moreover, $b_E$ is approximately differentiable $\h^{n-1}$-a.e. on $\R^{n-1}$, namely the approximate gradient $\nabla b_E(x)$ (see Section \ref{preliminaries}) exists for $\h^{n-1}$-a.e. $x\in\R^{n-1}$. Finally, for every Borel set $G\subset \{ v\Upp >0 \}$ the following coarea formula holds:
\begin{align}\label{eq: 1.12 Filippo pag 7}
\int_{\R}\h^{n-2}(G\cap \partial^e\{b_E>t\}) dt= \int_G|\nabla b_E|d\h^{n-1}+\int_{G\cap S_{b_E}}[b_E]d\h^{n-2}+|D^cb_E|^+(G),
\end{align} 
where $|D^cb_E|^+$ is the Borel measure on $\R^{n-1}$ defined by
\begin{align*}
|D^cb_E|^+(G):=\lim_{\delta\rightarrow 0^+}|D^cb_{\delta}|(G)=\sup_{\delta>0}|D^cb_{\delta}|(G),\qquad \forall\,G\subset \R^{n-1}.
\end{align*}
\end{theorem}
\noindent
Here $GBV$ is the space of functions of generalized bounded variation, $v\Upp$ and $v\Low$ are the approximate limsup and approximate liminf of $v$ respectively, $[b_E]:=b_E\Upp-b_E\Low$ is the jump of $b_E$, $S_{b_E}$ is the jump set of $b_E$, and $D^cb_{\delta}$ is the Cantor part of the distributional derivative $Db_{\delta}$ of $b_{\delta}$ (for more details see Section \ref{preliminaries}). Starting from  this result it is possible to establish a formula for the perimeter of $E$ in terms of $v$ and $b_E$ (see \cite[Corollary 3.3]{CCPMSteiner}). With such formula at hands, as shown in the next result (see \cite[Theorem 1.9]{CCPMSteiner}), a full characterization of $\mathcal{M}(v)$ can be given. 
Below, we set $\tau_M(s):=\max \{-M, \min\{ M,s \} \}$ for every $s\in\R$, and $M\geq 0$, that is 
\begin{align*}
\tau_M(s):=
\begin{cases}
-M  &\mbox{if}\quad s\leq -M,\\
s  &\mbox{if}\quad -M < s < M,\\
M   &\mbox{if}\quad  s \geq M.
\end{cases}
\end{align*}

\begin{theorem}\label{thm:1.2 Filippo}
Let $v$ be as in (\ref{due tilde}), and let $E$ be a $v$-distributed set of finite perimeter. Then, $E\in \mathcal{M}(v)$ if and only if 
\begin{align}
&i)\;E_z \textit{ is $\h^1$-equivalent to a segment;}\quad
 \textit{for $\h^{n-1}$-a.e. $z\in\R^{n-1}$},\label{eq: 1.15 Filippo}\\
&ii)\;\nabla b_E(z)=0,\quad \textit{for $\h^{n-1}$-a.e. $z\in\R^{n-1}$};\label{eq: 1.16 Filippo}\\
&iii)\;2[b_E]\leq [v], \quad \textit{ $\h^{n-2}$-a.e. on $\{v\Low >0  \}$};\label{eq: 1.18 Filippo}\\
&iv)\;\textit{there exists a Borel function $f:\R^{n-1}\rightarrow [-1/2,1/2]$ such that } \nonumber\\
 & \quad D^c\left( \tau_M(b_{\delta}) \right)(G)= \int_{G\cap \{ v>\delta \}^{(1)}\cap \{ |b_E|<M \}^{(1)}}fd(D^cv),\label{eq: 1.18Dc Filippo}
\end{align}
\textit{for every bounded Borel set $G\subset\R^{n-1}$} and $M>0$, and for $\h^1$-a.e. $\delta>0$. In particular, if $E\in\mathcal{M}(v)$ then
\begin{align}
2|D^cb_E|^+(G)\leq |D^cv|(G),\quad \textit{for every Borel set $G\subset \R^{n-1}$},\label{eq: 1.19 Filippo}
\end{align}
and, if $K$ is a concentration set for $D^cv$ and $G$ is a Borel subset of $\{ v\Low >0 \}$, then
\begin{align}
\int_{\R}\h^{n-2}(G\cap \partial^e\{b_E>t  \})dt=\int_{G\cap S_{b_E}\cap S_v}[b_E]d\h^{n-2}+|D^cb_E|^+(G\cap K).\label{eq: 1.20 Filippo}
\end{align}
\end{theorem}
\medskip
\noindent
Theorem \ref{thm: baricentro Filippo} and Theorem \ref{thm:1.2 Filippo} play a key role in the study of rigidity. Indeed, (\ref{rigidity steiner}) holds true if and only if the following condition is satisfied:
\begin{align}
E\in\mathcal{M}(v) \quad \Longleftrightarrow \quad \textit{$b_E$ is $\h^{n-1}$-a.e. constant on $\{v>0 \}$}. 
\end{align}

%To better understand when rigidity fails, i.e. when $b_E$ is non constant (up to a set of $\h^{n-1}$ measure zero) on $\{ v>0 \}$, let us introduce the following definition (see \cite[Definition 1.4]{CCPMSteiner}). Let $A$ and $G$ be Borel sets in $\R^m$. One says that $A$ \emph{essentially disconnects} $G$ if there exists a non trivial Borel partition $\{G_+,G_- \}$ of $G$ modulo $\mathcal{H}^{m}$ such that
%\begin{equation}
%\mathcal{H}^{m-1}\left(\left(G^{(1)}\cap \partial^eG_+ \cap \partial^eG_-   \right)\setminus A  \right)=0;
%\end{equation}
%conversely, one says that $A$ \emph{does not essentially disconnect} $G$ if, for every non trivial Borel partition $\{G_+,G_-  \}$ of $G$ modulo $\mathcal{H}^{m}$,
%\begin{equation}
%\mathcal{H}^{m-1}\left(\left(G^{(1)}\cap \partial^eG_+ \cap \partial^eG_-   \right)\setminus A  \right)>0.
%\end{equation}
%Finally, $G$ is \emph{essentially connected} if $\emptyset$ does not disconnect $G$. We recall that a \emph{non trivial partition} of $\{G_+,G_-  \}$ of $G$ modulo $\mathcal{H}^{m}$ we mean that 
%$$ \mathcal{H}^{m}\left(G_+ \cap G_-  \right)=0,\quad \mathcal{H}^{m}\left(G\Delta (G_+ \cup G_-)  \right)=0, \quad \mathcal{H}^{m}\left(G_+\right)\mathcal{H}^{m}\left( G_-  \right)>0. $$
%Thus, rigidity fails if and only if there exists $I\subset \R$ with $\h^1(I)>0$ such that $\forall\, t\in I$, $\partial^e\{ b_E>t \}$ essentially disconnects $\{ v>0\}$ (see the introduction of\cite{CCPMSteiner} for more details) 
Based on the previous results, several rigidity results are given in \cite{CCPMSteiner}, depending of the regularity assumptions on $v$ (see \cite[Theorems 1.11-1.30]{CCPMSteiner}). In particular, a complete characterization of rigidity is given when $v$ is a special function of bounded variation with locally finite jump set (see \cite[Theorem 1.29]{CCPMSteiner}).

\subsection{Anisotropic perimeter}\label{subsec: 1.5} Let us start by recalling some basic notions. A function $\phi:\R^n \rightarrow [0,\infty)$ is said to be \emph{1-homogeneous} if
\begin{align}\label{eq: 1-homgeneity}
\phi(x)= |x|\phi\left( \frac{x}{|x|} \right)\qquad \forall\,x\in \R^n\setminus\{0\}.
\end{align}
If $\phi$ is 1-homogeneous, then we say that it is \emph{coercive} if there exists $c>0$ such that
\begin{align}
\phi(x)\geq c|x| \qquad\forall\,x\in \R^n.
\end{align}
In the following, we will assume that
\begin{align}\label{HP per K}
\textit{$K\subset \R^n$ is open, bounded, convex and contains the origin.}
\end{align}
Given $K$ as in (\ref{HP per K}), one can define a one-homogeneous, convex and coercive function $\phi_K:\R^n\rightarrow [0,\infty)$ in this way:
\begin{align}\label{def surface tension}
\phi_K(x):=\sup\left\{x\cdot y:\, y\in K  \right\},
\end{align} 
see Figure \ref{fig:phi definitionZ}.
\begin{figure}[!htb]
\centering
\def\svgwidth{7cm}
%% Creator: Inkscape inkscape 0.92.3, www.inkscape.org
%% PDF/EPS/PS + LaTeX output extension by Johan Engelen, 2010
%% Accompanies image file 'DefinizionefiZ.pdf' (pdf, eps, ps)
%%
%% To include the image in your LaTeX document, write
%%   \input{<filename>.pdf_tex}
%%  instead of
%%   \includegraphics{<filename>.pdf}
%% To scale the image, write
%%   \def\svgwidth{<desired width>}
%%   \input{<filename>.pdf_tex}
%%  instead of
%%   \includegraphics[width=<desired width>]{<filename>.pdf}
%%
%% Images with a different path to the parent latex file can
%% be accessed with the `import' package (which may need to be
%% installed) using
%%   \usepackage{import}
%% in the preamble, and then including the image with
%%   \import{<path to file>}{<filename>.pdf_tex}
%% Alternatively, one can specify
%%   \graphicspath{{<path to file>/}}
%% 
%% For more information, please see info/svg-inkscape on CTAN:
%%   http://tug.ctan.org/tex-archive/info/svg-inkscape
%%
\begingroup%
  \makeatletter%
  \providecommand\color[2][]{%
    \errmessage{(Inkscape) Color is used for the text in Inkscape, but the package 'color.sty' is not loaded}%
    \renewcommand\color[2][]{}%
  }%
  \providecommand\transparent[1]{%
    \errmessage{(Inkscape) Transparency is used (non-zero) for the text in Inkscape, but the package 'transparent.sty' is not loaded}%
    \renewcommand\transparent[1]{}%
  }%
  \providecommand\rotatebox[2]{#2}%
  \newcommand*\fsize{\dimexpr\f@size pt\relax}%
  \newcommand*\lineheight[1]{\fontsize{\fsize}{#1\fsize}\selectfont}%
  \ifx\svgwidth\undefined%
    \setlength{\unitlength}{810.08390024bp}%
    \ifx\svgscale\undefined%
      \relax%
    \else%
      \setlength{\unitlength}{\unitlength * \real{\svgscale}}%
    \fi%
  \else%
    \setlength{\unitlength}{\svgwidth}%
  \fi%
  \global\let\svgwidth\undefined%
  \global\let\svgscale\undefined%
  \makeatother%
  \begin{picture}(1,0.74976066)%
    \lineheight{1}%
    \setlength\tabcolsep{0pt}%
    \put(0,0){\includegraphics[width=\unitlength,page=1]{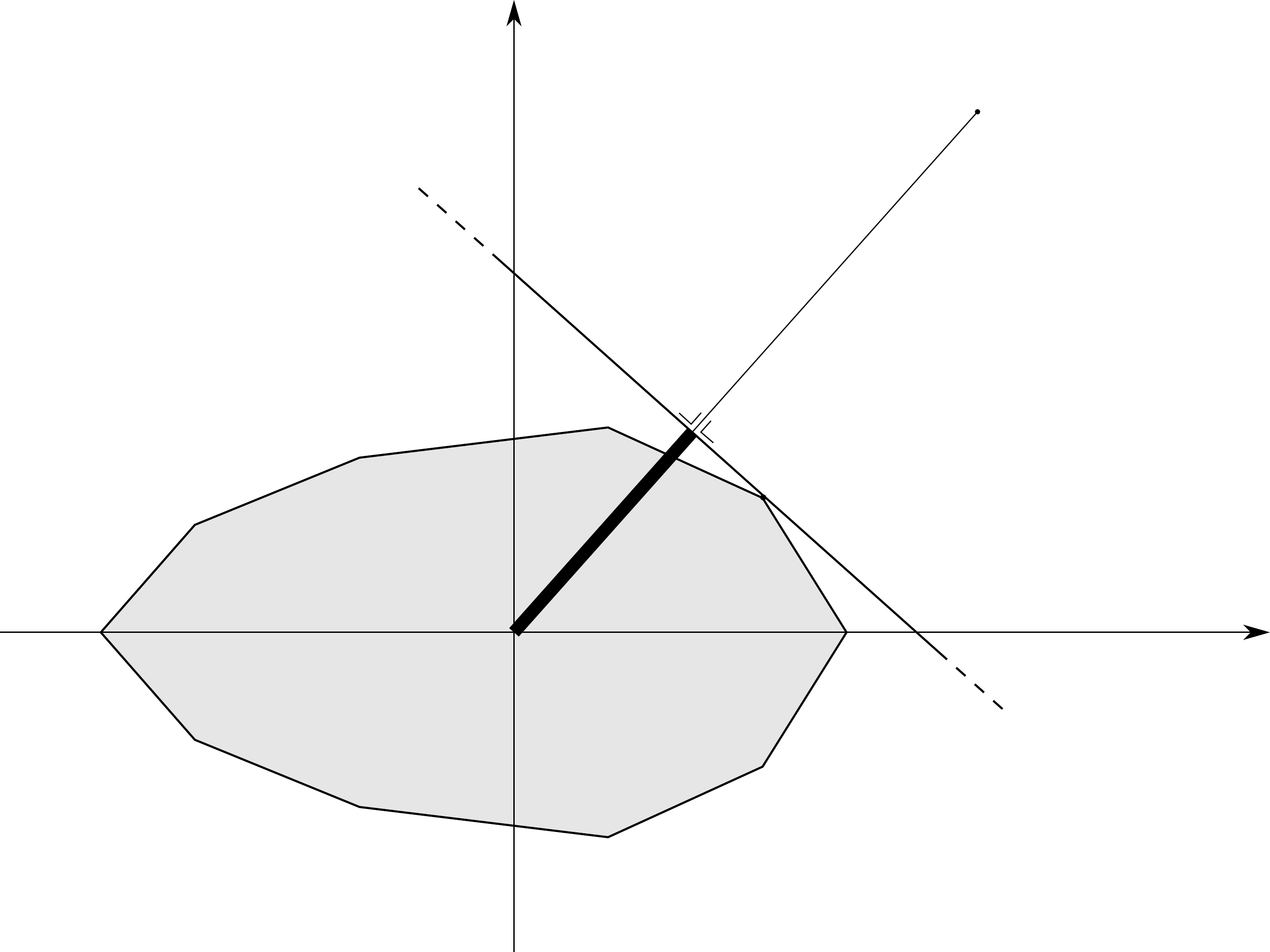}}%
    \put(0.77111894,0.63900264){\color[rgb]{0,0,0}\makebox(0,0)[lt]{\lineheight{0.82591271}\smash{\begin{tabular}[t]{l}\tiny{$x$}\end{tabular}}}}%
    \put(0.42475939,0.2142728){\color[rgb]{0,0,0}\makebox(0,0)[lt]{\lineheight{0.82591271}\smash{\begin{tabular}[t]{l}\tiny{$O$}\end{tabular}}}}%
    \put(0.61875377,0.35225982){\color[rgb]{0,0,0}\makebox(0,0)[lt]{\lineheight{0.82591271}\smash{\begin{tabular}[t]{l}\tiny{$\bar{y}$}\end{tabular}}}}%
    \put(0.18255559,0.27554821){\color[rgb]{0,0,0}\makebox(0,0)[lt]{\lineheight{0.82591271}\smash{\begin{tabular}[t]{l}\tiny{$K$}\end{tabular}}}}%
  \end{picture}%
\endgroup%

\caption{A pictorial description of link between the set $K$ and the function $\phi_K$ in case $n=2$. The length of the segment in bolt equals $\phi_K\left( \frac{x}{|x|} \right)$. Note that $\bar{y}$ is the point such that we have $\phi_{K}(x)=x\cdot \bar{y}$. Therefore, the line passing through $\bar{y}$ orthogonal to the vector $x$ represents the hyperplane $\left\{y\in\R^2:\, y\cdot \frac{x}{|x|}= \phi_{K}\left(\frac{x}{|x|}  \right)  \right\}$.}
\label{fig:phi definitionZ}
\end{figure}
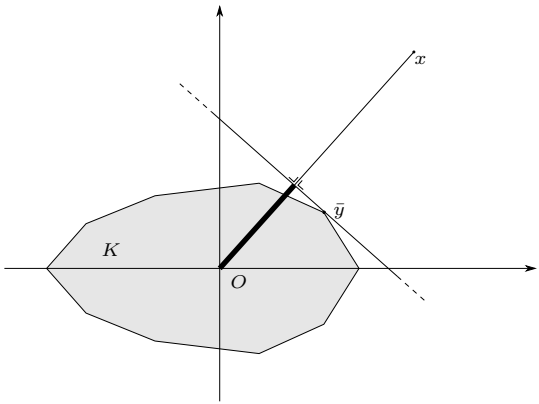
\noindent
By homogeneity, convexity of $\phi_K$ is equivalent to \emph{subadditivity} (see for instance \cite[Remark~20.2]{MaggiBOOK}), namely
\begin{align}\label{eq: phi subadditiva}
\phi_K(x_1+x_2)\leq \phi_K(x_1)+\phi_K(x_2),\qquad \forall\,x_1,x_2\in \R^n.
\end{align}
Let us notice that there is a one to one correspondence between open, bounded and convex sets $K$ containing the origin and one-homogeneous, convex and coercive functions $\phi:\R^n\rightarrow [0,\infty)$. Indeed, given a one-homogeneous, convex and coercive function $\phi:\R^{n}\rightarrow [0,\infty)$, then the set
\begin{align}\label{def: Wulff shape K}
K=\bigcap_{\omega \in \mathbb{S}^{n-1}}\left\{x\in \R^n:\, x\cdot\omega < \phi(\omega)   \right\},
\end{align}
satisfies (\ref{HP per K}), where $\mathbb{S}^{n-1}=\{x\in\R^n:\,|x|=1\}$, and 
$$\phi(x)=\sup\left\{x\cdot y:\, y\in K  \right\}=\phi_K(x),$$
\noindent
where $\phi_K$ is given by (\ref{def surface tension}).
Let $E\subset \R^n$ be a set of finite perimeter and let $G\subset \R^{n}$ be a Borel set. Given $K\subset \R^n$ as in (\ref{HP per K}), we define the \emph{anisotropic perimeter}, with respect to $K$, of $E$ relative to $G$, as 
\begin{align*}
P_K(E;G)= \int_{\E\cap G}\phi_K(\nu^{E}(x)) d\mathcal{H}^{n-1}(x),
\end{align*}
and the anisotropic perimeter $P_K(E)$ of $E$ as $P_K(E;\R^n)$. Observe that in the special case $\phi_K(x)=|x|$, this notion of perimeter agrees with the one of Euclidean perimeter which corresponds to $K=B(1)$. Note that, in general, $\phi_K$ is not a norm, unless $\phi_K(x)=\phi_K(-x)$ for every $x\in \R^n$. \\
In the applications, the anisotropic perimeter can be used to describe the surface tension in the study of equilibrium configurations of solid crystals with sufficiently small grains
\cite{H, Ta, W}, and represents the basic model for surface energies in phase transitions \cite{G}. These applications motivate the study of the the \emph{Wulff problem} (or \emph{anisotropic isoperimetric problem}): 
\begin{align}\label{Wulff Problem}
\inf\left\{\int_{\E}\phi_K(\nu^E(x))d\mathcal{H}^{n-1}(x):\, E\subset \R^n,\, \h^n(E)=\h^n(K) \right\}.
\end{align}
This name comes from the Russian  crystallographer Wulff, who was the first one to study (\ref{Wulff Problem}) and who first conjectured that $K$ is the unique (modulo translations and scalings) minimizer of (\ref{Wulff Problem}) (see \cite{W}). Indeed the anisotropic perimeter inequality holds true (see for instance \cite[Chapter 20]{MaggiBOOK}): 
\begin{align}\label{wulff inequality}
P_K(K)\leq P_K(E) \quad \textit{for every $E\subset\R^n$ with $\h^n(E)=\h^n(K)$},
\end{align}
with equality if and only if $\h^n(K\Delta (E+x))=0$ for some $x\in\R^n$. The proof of the uniqueness was then given by Taylor (see \cite{Ta}) and later, with a different method, by Fonseca and M\"uller (see \cite{FM}). We usually refer to $K$ as the \emph{Wulff shape} for the surface tension $\phi_K$.\\

\subsection{Steiner's inequality for the anisotropic perimeter}\label{subsec: 1.6}
Note that the analogous of inequality (\ref{eq: Classic Steiner inequality}) for the anisotropic perimeter in general fails. Indeed, choose $K$ as in (\ref{HP per K}) such that 
\begin{align*}
\inf_{x\in\R^n}\h^n(K\Delta (K^s+x))> 0,
\end{align*}
where $K^s$ denotes the Steiner symmetral of $K$. Then, by uniqueness of the solution for (\ref{Wulff Problem}), we have that 
$$
P_K(K)<P_K(K^s).
$$

%Let us give a simple example of the above inequality in dimension 2. Let $K$ and $K^s$ be as in Figure \ref{fig:esempiocorrezioni3Z}. Then, one can see that 
%$$
%P_K(K)=8<10=P_K(K^s).
%$$

The above considerations show that, for an inequality as in (\ref{eq: Classic Steiner inequality}) to hold true in the anisotropic setting, one should at least consider the perimeter $P_{K^s}$ with respect to the Steiner symmetral $K^s$ of $K$.

\begin{remark}Let us observe that since $K$ is a convex set, by properties of Steiner symmetrisation, $K^s$ is convex too. This general property of Steiner symmetrisation can be summed up in the following statement. Let $v$ be as in (\ref{due tilde}), such that $F[v]$ is not a convex set, then every $v$-distributed set $E$ satisfying (\ref{eq: Ez segment}) cannot be convex. To prove this, let us consider two points $x,y\in F[v]$ such that the segment joining $x$ and $y$, namely $\overline{xy}$ is not fully contained in $F[v]$. It is not restrictive to make the following assumptions on $x,y$, namely $\p x\neq \p y$, $\q x> 0$, and also $\q y>0$. Let us call by $x^-$, $y^-$ the two points of $F[v]$ obtained by reflecting $x$, and $y$ with respect to $\{x_n=0 \}$, namely $x^-=(\p x, -\q x)$, $y^-=(\p y,-\q y)$. Let us consider the quadrilateral $Q$ with vertexes in $x,y,y^-,x^-$. By symmetry of $F[v]$, and since we assumed that $\overline{xy}$ is not fully contained in $F[v]$, there exists $\bar{z}\in \overline{\p x\p y}\setminus \{\p x,\p y\}$ such that 
\begin{align}\label{eq: dajelo}
\h^1(Q_{\bar{z}})> v(\bar{z}),
\end{align}
where we recall $Q_z$ is defined in (\ref{la striscia verticale}). Let us now consider any $v$-distributed set $E$ satisfying condition (\ref{eq: Ez segment}), and let $x_1,y_1,x_1^-,y_1^-$ be the four points obtained from $x,y,x^-,y^-$ in the following way: $x_1=(\p x,\q x+b_E(\p x))$, $y_1=(\p y,\q y+b_E(\p y))$, $x_1^-=(\p x,-\q x+b_E(\p x))$, and $y_1^-=(\p y,-\q y+b_E(\p y))$. Observe that by construction $x_1,y_1,x_1^-,y_1^-\in E$. Let us call $Q^1$ the quadrilateral with vertexes in those four points. By construction of $Q,Q^1$, and since Steiner symmetrisation preserves vertical distances, we have that $\h^1(Q_z)=\h^1((Q^1)_z)$ for every $z\in \overline{\p x\p y}$. In particular, recalling (\ref{eq: dajelo}) we get
\begin{align*}
\h^1((Q^1)_{\bar{z}})> v(\bar{z}).
\end{align*}
As a direct consequence of the above inequality, we get that the quadrilateral $Q^1$, with vertexes in $E$, is not fully contained in $E$, and thus $E$ is not convex. By generality of $E$, we conclude. 
\end{remark}

\begin{remark}\label{rem: K simmetrica e quindi pure phi_K}
Let us observe that since $K^s$ is symmetric with respect to $\{x_n=0\}$, then  $\forall x\in \R^n$ we have that $\phi_{K^s}(\p x,\q x)=\phi_{K^s}(\p x, -\q x )$. 
\end{remark}
What actually can be proved is the following result (for its proof see \cite[Theorem 2.8]{cianchifusco2}).
\begin{theorem}\label{thm:anisotropic perimeter inequality}
Let $K \subset \R^n$ be as (\ref{HP per K}), let $K^s$ be its Steiner symmetral, and let $v$ as in (\ref{due tilde}). Then, for every $E\subset \R^n$ $v$-distributed we have
\begin{align}\tag{AS}\label{eq:anisotropic Steiner inequality}
P_{K^s}(E;G\times \R)\geq P_{K^s}(F[v];G\times \R)\quad \textit{for every Borel set $G\subset \R^{n-1}$}.
\end{align}
\end{theorem}\noindent

%\begin{figure}[!htb]
%\centering
%\def\svgwidth{8cm}
%\input{esempioCorrezioni3Z.pdf_tex}
%\caption{An example in which $P_K(K)<P_K(K^s)$. The coordinates of the vertices are $A=(-1,0)$, $B=(0,1)$, $C=(1,0)$, $D=(0,-3)$, $E=(-1,0)$, $F=(0,2)$, $G=(1,0)$, $H=(0,-2)$.}
%\label{fig:esempiocorrezioni3Z}
%\end{figure}

\subsection{Statement of the main results} \label{subsec: 1.7} We are now ready to state our main results. Given $v$ as in (\ref{due tilde}), and $K\subset\R^n$ satisfying (\ref{HP per K}) we denote by
\begin{align}
\mathcal{M}_{K^s}(v):=\left\{E\subset \R^n:\, \textit{\emph{$E$ is $v$-distributed and $P_{K^s}(E)=P_{K^s}(F[v])$}}\right\},
\end{align}
the family of sets achieving equality in (\ref{eq:anisotropic Steiner inequality}). In this context, we say that \emph{rigidity} holds true for (\ref{eq:anisotropic Steiner inequality}) if the only elements of $\mathcal{M}_{K^s}(v)$ are vertical translations of $F[v]$, namely
\begin{align}\tag{RAS}\label{rigidity anisotropic steiner}
E\in \mathcal{M}_{K^s}(v) \quad \Longleftrightarrow \quad \h^n(E\Delta (F[v]+te_n))=0 \textit{ for some $t\in\R$}.
\end{align}

\begin{figure}[!htb]
\centering
\def\svgwidth{15cm}
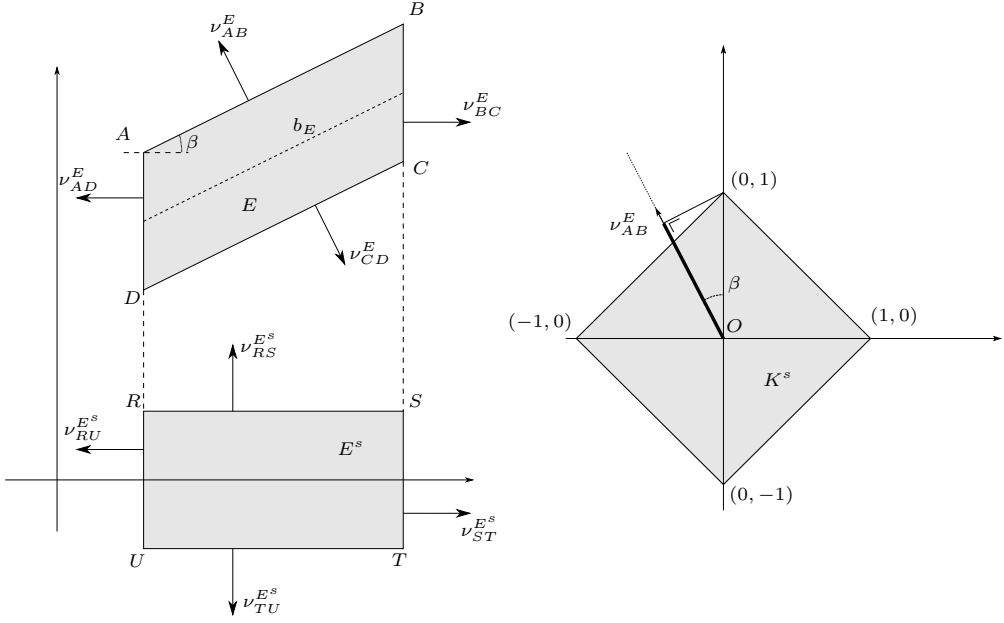
\caption{Suppose that $0< \beta\leq \pi/4$. By definition of $\phi_{K^s}$, one can check that the length of the segment in bolt equals $\phi_{K^s}(\nu^E_{AB})=\phi_{K^s}(\nu^E_{CD})= \cos(\beta)$. As a consequence, we have $P_{K^s}(E)=P_{K^s}(E^s)$, even if $b'_E=\tan\beta\neq0$.}
\label{fig:wulff shapesZ}
\end{figure}

As done for the study of (\ref{rigidity steiner}), we start by characterizing the set $\mathcal{M}_{K^s}(v)$. Note that, in the anisotropic setting, the conditions given in Theorem \ref{thm:1.2 Filippo} do not give a characterization of equality cases of (AS). In particular, let us show with an example in dimension $2$, that condition (\ref{eq: 1.16 Filippo}) fails to be necessary.  Let $K^s, E$, and  $E^s$ be as in Figure \ref{fig:wulff shapesZ}. Observe that, although $\nabla b_E=b'_E=\tan(\beta)\neq 0$ we have $P_{K^s}(E)=P_{K^s}(E^s)$, if $0< \beta\leq \pi/4$. Indeed, in this case, setting $h=\mathcal{H}^1(AD)=\mathcal{H}^1(BC)=\mathcal{H}^1(RU)=\mathcal{H}^1(ST)$, and $l=\mathcal{H}^1(RS)=\mathcal{H}^1(TU)$ we get 
\begin{align*}
P_{K^s}(E)&= \phi_{K^s}(\nu^E_{AB})\h^1(AB) + \phi_{K^s}(\nu^E_{CD})\h^1(CD)+ \phi_{K^s}(\nu^E_{AD})h + \phi_{K^s}(\nu^E_{BC})h \\
&=2\cos(\beta)\h^1(AB)+2h =2 \cos(\beta)\frac{l}{\cos(\beta)}+2h= 2l+2h = P_{K^s}(E^s).
\end{align*}
Interestingly, if $\pi/4 < \beta <\pi/2$ one can see that $P_{K^s}(E)> P_{K^s}(E^s)$.\\

\noindent
We will see that this simple example carries some important features of the general case. In order to characterize $\mathcal{M}_{K^s}(v)$ we start by proving a formula that allows to calculate $P_K(E)$ in terms of $b_E$ and $v$ whenever $E$ is a $v$-distributed set satisfying (\ref{eq: Ez segment}) (see Corollary \ref{cor: formula perimetro con b v}). After that, we need to carefully study under which conditions equality holds true in (\ref{eq: phi subadditiva}), see Proposition \ref{prop:linearityphi}. 

Before stating our results, let us give some definitions. 
%In order to study and characterize the set $\mathcal{M}_{K^s}(v)$, we first study in details the properties of the surface tension $\phi_K$, focusing in particular in its cases of additivity. To better explain what we did, let us recall few definitions. 
If $K\subset \R^n$ is as in (\ref{HP per K}), we define the gauge function $\phi^*_K:\R^n \rightarrow [0,\infty)$ as
\begin{align}\label{def: gauge function}
\phi^*_K(x):=\sup\{x\cdot y:\, \phi_K(y) <1  \}.
\end{align}
It turns out that $\phi_K^*$ is one-homogeneous, convex and coercive on $\R^n$ (see Proposition \ref{prop: Maggi 20.10}).
%Therefore, as explained in (\ref{def: Wulff shape K}) we associate to $\phi^*_{K}$ the set $K^*$ and we call it the \emph{dual} of $K$ (see Figure \ref{fig:C_K^*}).
%\begin{figure}[!htb]
%\centering
%\def\svgwidth{12.5cm}
%\input{wulffshapeconduale.pdf_tex}
%\caption{An example of Wulff shape $K$ and its dual $K^*$ in $\R^2$.}
%\label{fig:K e Kduale}
%\end{figure}
%\begin{definition}[Sub-differential of a convex function]\label{def: sub-differential}
%Let $\varphi:\R^n\rightarrow [0,\infty]$ be a convex function. Let us fix $x_0 \in \R^n$ and consider all vectors $y_0 \in \R^n$ such that
%\begin{align}
%\varphi(z)\geq \varphi(x_0) +y_0\cdot (z-x_0)\quad \forall z \in \R^n.
%\end{align}
%The set of all vectors $y_0$ satisfying the above property is called \emph{sub-differential} of $\varphi$ in $x_0$ and we indicate it by $\partial \varphi(x_0)$.
%\end{definition}
Let now $x_0\in \partial K$ and let $\partial\phi^*_K(x_0)$ denote the sub-differential of $\phi_K^*$ at $x_0$ (see Definition \ref{def: sub-differential}). We define the \emph{positive cone} generated by $\partial\phi^*_K(x_0)$, as
\begin{align}\label{def:weak sub-differential}
C^*_K(x_0):=\left\{\lambda y:\, y\in \partial\phi^*_K(x_0)\,, \lambda\geq 0  \right\},
\end{align} 
see Figure \ref{fig:C_K^*}. In the following, if $\mu$ is an $\R^n$-valued Radon measure in $\R^{n-1}$, we denote by $|\mu|_K$ the anisotropic total variation (with respect to $K$) of $\mu$, see Definition \ref{def: anisotropic total variation}.

\begin{figure}[!htb]
\centering
\def\svgwidth{13cm}
\input{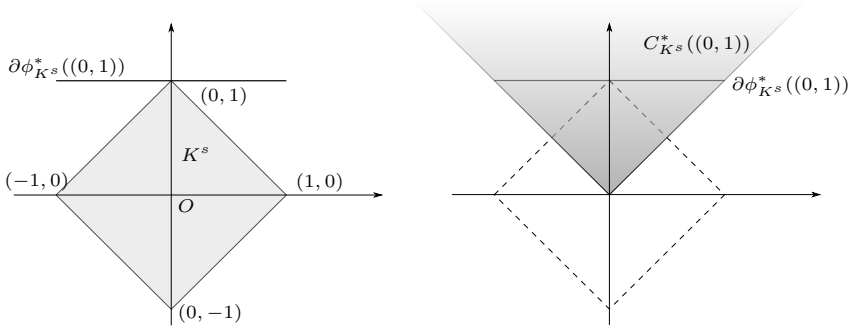}
\caption{On the left $K^s$ and a pictorial idea of the sub-differential $\partial\phi_{K^s}^*((0,1))$, whereas on the right a pictorial representation of  $C_{K^s}^*((0,1))$.}
\label{fig:C_K^*}
\end{figure}   

Our first result gives a complete characterization of $\mathcal{M}_{K^s}(v)$, and can be considered as the anisotropic version of Theorem \ref{thm:1.2 Filippo}. Note that, in particular, this extends \cite[Theorem 2.9]{cianchifusco2}, where necessary conditions for a set to belong to $\mathcal{M}_{K^s}(v)$ were given.

\begin{theorem}\label{thm:2.2 pag 117}
Let $v$ be as in (\ref{due tilde}), let $K\subset \R^n$ satisfy (\ref{HP per K}), and let $E$ be a $v$-distributed set of finite perimeter. Then, $E\in \mathcal{M}_{K^s}(v)$ if and only if
\begin{itemize}
\item[i)]$E_x$ is $\h^1$-equivalent to a segment, for $\h^{n-1}$-a.e. $x\in\R^{n-1}$;
\item[ii)]for $\mathcal{H}^{n-1}$-a.e. $x\in\{v>0\}$ there exists $z(x)\in\partial K^s$ s.t.
\begin{align}\label{eq:1.16 pag 7 Filippo}
\left\{\left(-\frac{1}{2}\nabla v(x)+t \nabla b_E(x),1  \right):\,t\in [-1,1]   \right\} \subset C^*_{K^s}(z(x));
\end{align}
\item[iii)]for $\mathcal{H}^{n-2}$-a.e. $x\in\{v\Low >0\}$ we have that
\begin{align}\label{eq:1.17 pag 7 Filippo}
2[b_E](x)\leq [v](x);
\end{align}
\item[iv)] There exists a Borel function $g:\R^{n-1}\rightarrow\R^{n-1}$ such that 
\begin{align*}
D^c (\tau_Mb_{\delta})(G)=\int_{G\cap \{ v>\delta \}^{(1)}\cap \{ |b_E|<M \}^{(1)}}g(x)d|(D^cv/2,0)|_{K^s}(x),
\end{align*}
for every bounded Borel set $G\subset\R^{n-1}$, every $M>0$, and $\mathcal{H}^1$-a.e. $\delta >0$. Moreover, $g$ satisfies the following property: for $|D^cv|$-a.e. $x\in\{v\Low>0  \}$ there exists $z(x)\in \partial K$ s.t.
\begin{align}\label{eq:1.18 pag 7 Filippo}
\left\{(h(x)+tg(x),0):\,t\in[-1,1]\right\}\subset C^*_{K^s}(z(x)),
\end{align}
where 
\begin{align}\label{h}
h(x):= \frac{-dD^c v/2 }{d\left|(D^c v/2,0) \right|_{K^s}}(x),
\end{align}
is the derivative of $-D^c v/2$ with respect to the anisotropic total variation $\left|(D^c v/2,0) \right|_{K^s}$ in the sense of Radon measures.
\end{itemize}
\end{theorem}
\noindent
\begin{remark}
In Figure \ref{fig:spiegazione condizione su nabla b_E}, we give a pictorial idea of condition (\ref{eq:1.16 pag 7 Filippo}) for the example of Figure \ref{fig:wulff shapesZ}.
\end{remark}

\begin{figure}[!htb]
\centering
\def\svgwidth{13cm}
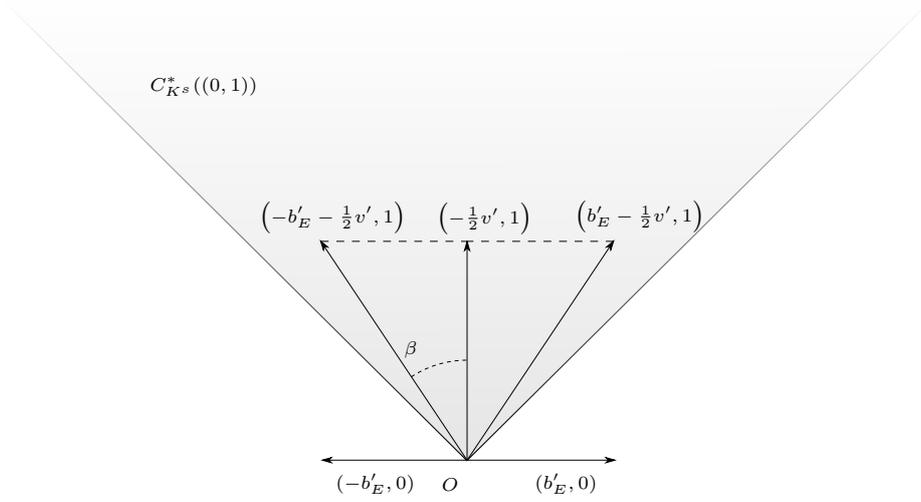
\caption{A pictorial idea of condition (\ref{eq:1.16 pag 7 Filippo}), for the example given in Figure~\ref{fig:wulff shapesZ}. As long as $0\leq \beta\leq \pi/4$, $\nabla b_E = b'_E$ is such that (\ref{eq:1.16 pag 7 Filippo}) is satisfied, and so in this simple example we get that $E\in\mathcal{M}_{K^s}(v)$. Note that $v'=0$, since $v$ is constant.}
\label{fig:spiegazione condizione su nabla b_E}
\end{figure}  
\noindent

If the first step we did, in order to study the rigidity problem in the anisotropic setting, was the characterization of the set $\mathcal{M}_{K^s}(v)$, the second step consists in the understanding of the relation between the sets $\mathcal{M}_{K^s}(v)$ and $\mathcal{M}(v)$. To get how important this is for our goal, let us observe the following fact. Let $v$ be as in (\ref{due tilde}), let $K\subset \R^n$ satisfy (\ref{HP per K}), and let us assume that (\ref{rigidity anisotropic steiner}) holds true. Then, $\mathcal{M}_{K^s}(v)\subset \mathcal{M}(v)$. Indeed, consider $E\in\mathcal{M}_{K^s}(v)$, i.e. $E$ is a $v$-distributed set of finite perimeter such that $P_{K^s}(E)=P_{K^s}(F[v])$. By (\ref{rigidity anisotropic steiner}), we know that $\h^n(E\Delta (F[v]+te_n))=0$ for some $t\in\R$. As a direct consequence we get that $P(E)=P(F[v])$, and so $E\in \mathcal{M}(v)$. So, we have just proved that a necessary condition in order to have rigidity in the anisotropic setting is to require that $\mathcal{M}_{K^s}(v)\subset\mathcal{M}(v)$ holds true. Let us remark that the opposite inclusion, namely $\mathcal{M}(v)\subset\mathcal{M}_{K^s}(v)$ is always verified. This is true because in general, the conditions given in Theorem \ref{thm:1.2 Filippo} are more stringent than those appearing in Theorem \ref{thm:2.2 pag 117}. So, to sum up all the previous observations we obtained, a necessary condition to require in order to get rigidity of equality cases in the anisotropic setting is that $\mathcal{M}_{K^s}(v)=\mathcal{M}(v)$. 

Therefore, to study the rigidity problem in the anisotropic setting, it is crucial to understand when the non trivial inclusion $\mathcal{M}_{K^s}(v)\subset\mathcal{M}(v)$ holds true. To this aim, given $K\subset\R^n$ as in (\ref{HP per K}) and $y\in\R^n$, we set 
\begin{align}\label{def: Z_K(y)}
\mathcal{Z}_K\left( y \right):=\left\{z\in\partial K:\, y\in C^*_{K}(z)  \right\}.
\end{align}
Note that $\emptyset\neq \mathcal{Z}_K\left( y \right)=\mathcal{Z}_K\left( \lambda y \right)$ for ever $y\in\R^n$ and for every $\lambda>0$ (see for instance relation (\ref{eq: step 2 sub-dif bord wulff}) in Lemma \ref{lem: sub-diff e bordo wulff}). The following two conditions will play an important role in the understanding of rigidity.
\vspace*{0.3cm}\\
\textbf{R1}: $\forall\,y\in \R^n$, for $\h^{n-1}$-a.e. $x\in \{v>0\}$, and $\forall\,z\in\mathcal{Z}_{K^s}\left( \left(-\frac{1}{2}\nabla v(x),1  \right) \right)$,
\begin{align*}
\left(-\frac{1}{2}\nabla v(x),1  \right)+ y,\left(-\frac{1}{2}\nabla v(x),1  \right)- y \in C^*_{K^s}(z)\quad \Longrightarrow\quad y=\lambda \left(-\frac{1}{2}\nabla v(x),1  \right),
\end{align*}
for some $\lambda\in [-1,1]$.
\vspace*{0.3cm}\\
\textbf{R2}: $\forall\,y\in \R^n$, for $|D^cv|$-a.e. $x\in\{v\Low>0\}$, and $\forall\,z\in\mathcal{Z}_{K^s}\left((h(x),0) \right)$, 
\begin{align*}
(h(x),0)+ y,\,(h(x),0)-y \in C^*_{K^s}(z)\quad \Longrightarrow\quad y=\lambda (h(x),0),\quad \quad \textit{for some $\lambda\in [-1,1]$},
\end{align*}
where $h$ has been defined in (\ref{h}). Next result shows the importance of conditions \textbf{R1} and \textbf{R2}. We anticipate that although conditions \textbf{R1} and \textbf{R2} may seem quite complicated, they can be characterized in a simple way in terms of the possible value of the normal vectors to $\partial^*F[v]$ (see Proposition \ref{prop: R1,R2}, and Remark \ref{rem: post Prop 1.12 2}).
\begin{theorem}\label{thm: rigidity}
Let $v$ be as in (\ref{due tilde}) and let $K\subset\R^n$ be as in (\ref{HP per K}). In addition, let us assume that \textbf{R1} and \textbf{R2} hold true. Then, $\mathcal{M}_{K^s}(v)\subset\mathcal{M}(v)$. As an immediate consequence, (\ref{rigidity steiner}) and (\ref{rigidity anisotropic steiner}) are equivalent.
\end{theorem}
\begin{remark}
The above result can be seen as a generalization of \cite[Theorem 2.10]{cianchifusco2}.
\end{remark}
\noindent
Thanks to Theorem \ref{thm: rigidity}, all the characterization results for (RS) proved in \cite{CCPMSteiner}, also hold true in the anisotropic setting, provided conditions \textbf{R1} and \textbf{R2} are satisfied. In particular, as a direct consequence of Theorem \ref{thm: rigidity}, we have the following result.

\begin{theorem}\label{thm: eredità della rigidità di Steiner}
Let $v$ be as in (\ref{due tilde}) and let $K\subset\R^n$ be as in (\ref{HP per K}) such that \textbf{R1} and \textbf{R2} are satisfied. Then, the following results from \cite{CCPMSteiner} hold true, provided \emph{rigidity} is substituted with (RAS) and $\mathcal{M}(v)$ is substituted with $\mathcal{M}_{K^s}(v)$: \cite[Theorem 1.11]{CCPMSteiner}, \cite[Theorem 1.13]{CCPMSteiner}, \cite[Theorem~1.16]{CCPMSteiner}, \cite[Theorem~1.20]{CCPMSteiner}, \cite[Theorem 1.29]{CCPMSteiner}, and \cite[Theorem 1.30]{CCPMSteiner}.   
\end{theorem}
\noindent
To check whether conditions \textbf{R1}, \textbf{R2} hold true might be difficult in general. Thus, using well known concepts of convex analysis such as the definition of \emph{extreme point} and of \emph{exposed point} (see Definition \ref{def: exposed point} and Definition \ref{def: extreme point} respectively), we give simple necessary and sufficient conditions for \textbf{R1} and \textbf{R2} to hold true (see Proposition \ref{prop: R1,R2}, Figure \ref{fig:Rigiditacristallino}, and Figure \ref{fig:Rigiditanoncristallino}  below). Indeed, next proposition shows that \textbf{R1} and \textbf{R2} can be expressed in a clear geometric way, by comparing the set of normal vectors to $\partial^* F [v]$ to the set of normal vectors to $\partial^* K^s$. Roughly speaking, conditions \textbf{R1} and \textbf{R2} are both satisfied if and only if the first of these two sets is contained in the closure of the second one (see also Corollary \ref{crystal clear 1}). For the proofs of this and other results about rigidity, we refer to Section \ref{section Rigidity}.  

\noindent
To state our next result, we need another definition. If $K\subset\R^n$ is as in (\ref{HP per K}) we define the following set:
\begin{align}\label{V_{K^s}}
\mathbb{V}_{K^s}:=\left\{ \nu^{K^s}(x):\, x\in\partial^*K^s \right\}.
\end{align}
We indicate with $\overline{\mathbb{V}_{K^s}}$ the topological closure of $\mathbb{V}_{K^s}$.

%%%%%%%%%%%%%%%%%%%%%%%%%%%%%%%%%%%%%%%

%\begin{proposition}
%Let $v$ be as in (\ref{due tilde}) and let $K\subset\R^n$ be as in (\ref{HP per K}). Then, the following statements are equivalent:
%\begin{itemize}
%\item[i)]conditions \textbf{R1}, \textbf{R2} hold true;
%\medskip
%\item[ii)]there exists $S\subset \{v\Low >0 \}$ with the property that $\mathcal{H}^{n-1}(S)= |D^cv|(S)=0$, such that for every $z\in \{v\Low>0 \}\setminus (S\cup J_v)$ the following condition holds true 
%\begin{align}
%\nu^{F[v]}\left(z,\frac{1}{2}v(z)\right) \in \overline{\mathbb{V}_{K^s}}.
%\end{align}
%\end{itemize}
%\end{proposition}

\begin{proposition}\label{prop: R1,R2}
Let $v$ be as in (\ref{due tilde}) and let $K\subset\R^n$ be as in (\ref{HP per K}). Then, the following statements are equivalent:
\begin{itemize}
\item[i)]conditions \textbf{R1}, \textbf{R2} hold true;
\medskip
\item[ii)]$\exists\, S\subset\{ v\Low>0 \}$ such that $\mathcal{H}^{n-1}(S)=|D^cv|(S)=0$, and 
\begin{align}\label{eq lem R1,R2}
\nu^{F[v]}\left(z,\frac{1}{2}v(z)\right) \in \overline{\mathbb{V}_{K^s}}\quad \forall\,z\in \{v\Low>0\}\setminus S.
\end{align}
\end{itemize}
\end{proposition}
See Figure \ref{fig:Rigiditacristallino} and Figure \ref{fig:Rigiditanoncristallino} for a pictorial idea of condition $ii)$ in Proposition \ref{prop: R1,R2}.

%\begin{remark}\label{rem. explanation ii) Prop importante}
%An equivalent formulation for condition $ii)$ in the proposition above is the following: the set
%\begin{align*}
%S:=\left\{z\in \{v\Low >0 \}:\, \nu^{F[v]}(z,t)\notin \overline{\mathbb{V}_{K^s}},\text{ for $t>0$ such that }(z,t)\in\partial^*F[v]   \right\}
%\end{align*}
%is such that $\mathcal{H}^{n-1}(S)=|D^cv|(S)=0$.
%\end{remark}
%
%\begin{remark}\label{rem. explanation2 ii) Prop importante}
%An equivalent formulation for condition $ii)$ in the proposition above is the following: there exist $S_1, S_2\subset \{v\Low >0 \}$ such that $\mathcal{H}^{n-1}(S_1)=|D^cv|(S_2)=0$ and with the property that
%\begin{align*}
%\nu^{F[v]}(z,t)\in \overline{\mathbb{V}_{K^s}}\quad \forall\,z \in \{v\Low >0  \}\setminus (S_1\cap S_2),
%\end{align*}
%where $t>0$ is such that $(z,t)\in \partial^*F[v]$.
%\RRR posso lasciare $\frac{1}{2}v(z)$ al posto di $t$?\EEE
%\end{remark}

%%%%%%%%%%%%%%%%%%%%%%%%%%%%%%%%%%%%%%

\begin{remark}\label{rem: post Prop 1.12 2}
If $K^s$ is crystalline (i.e. $K^s$ is polyhedral) (see Figure \ref{fig:Rigiditacristallino}), or in case $K^s$ has $C^1$ boundary, then $\mathbb{V}_{K^s}$ is closed and so in (\ref{eq lem R1,R2}) we can substitute $\overline{\mathbb{V}_{K^s}}$ with $\mathbb{V}_{K^s}$.
\end{remark}

\begin{figure}[!htb]
\centering
\def\svgwidth{13cm}
%% Creator: Inkscape inkscape 0.92.3, www.inkscape.org
%% PDF/EPS/PS + LaTeX output extension by Johan Engelen, 2010
%% Accompanies image file 'Rigiditacristallino.pdf' (pdf, eps, ps)
%%
%% To include the image in your LaTeX document, write
%%   \input{<filename>.pdf_tex}
%%  instead of
%%   \includegraphics{<filename>.pdf}
%% To scale the image, write
%%   \def\svgwidth{<desired width>}
%%   \input{<filename>.pdf_tex}
%%  instead of
%%   \includegraphics[width=<desired width>]{<filename>.pdf}
%%
%% Images with a different path to the parent latex file can
%% be accessed with the `import' package (which may need to be
%% installed) using
%%   \usepackage{import}
%% in the preamble, and then including the image with
%%   \import{<path to file>}{<filename>.pdf_tex}
%% Alternatively, one can specify
%%   \graphicspath{{<path to file>/}}
%% 
%% For more information, please see info/svg-inkscape on CTAN:
%%   http://tug.ctan.org/tex-archive/info/svg-inkscape
%%
\begingroup%
  \makeatletter%
  \providecommand\color[2][]{%
    \errmessage{(Inkscape) Color is used for the text in Inkscape, but the package 'color.sty' is not loaded}%
    \renewcommand\color[2][]{}%
  }%
  \providecommand\transparent[1]{%
    \errmessage{(Inkscape) Transparency is used (non-zero) for the text in Inkscape, but the package 'transparent.sty' is not loaded}%
    \renewcommand\transparent[1]{}%
  }%
  \providecommand\rotatebox[2]{#2}%
  \newcommand*\fsize{\dimexpr\f@size pt\relax}%
  \newcommand*\lineheight[1]{\fontsize{\fsize}{#1\fsize}\selectfont}%
  \ifx\svgwidth\undefined%
    \setlength{\unitlength}{1257.94612553bp}%
    \ifx\svgscale\undefined%
      \relax%
    \else%
      \setlength{\unitlength}{\unitlength * \real{\svgscale}}%
    \fi%
  \else%
    \setlength{\unitlength}{\svgwidth}%
  \fi%
  \global\let\svgwidth\undefined%
  \global\let\svgscale\undefined%
  \makeatother%
  \begin{picture}(1,0.38753462)%
    \lineheight{1}%
    \setlength\tabcolsep{0pt}%
    \put(0,0){\includegraphics[width=\unitlength,page=1]{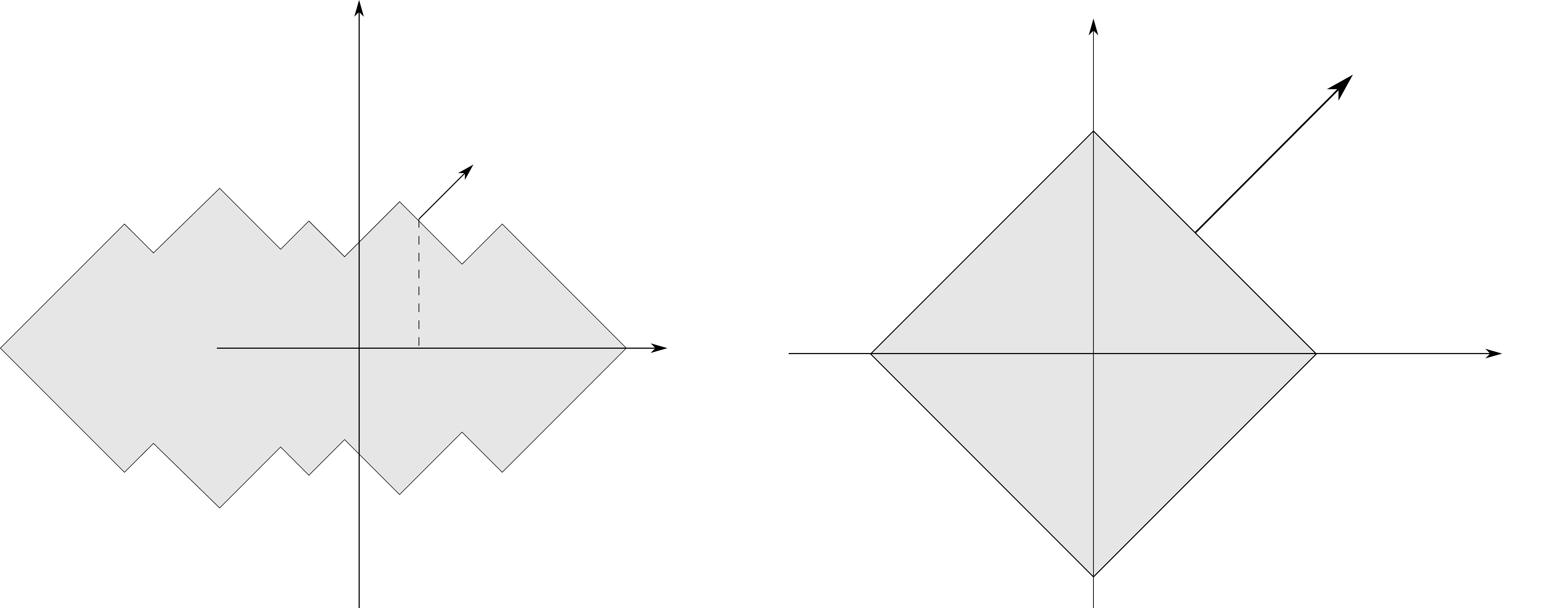}}%
    \put(0.20791974,0.17212721){\color[rgb]{0,0,0}\makebox(0,0)[lt]{\lineheight{0.82591271}\smash{\begin{tabular}[t]{l}\tiny{$O$}\end{tabular}}}}%
    \put(0.11613831,0.12396481){\color[rgb]{0,0,0}\makebox(0,0)[lt]{\lineheight{0.82591271}\smash{\begin{tabular}[t]{l}\tiny{$F[v]$}\end{tabular}}}}%
    \put(0.28267491,0.29558772){\color[rgb]{0,0,0}\makebox(0,0)[lt]{\lineheight{0.82591271}\smash{\begin{tabular}[t]{l}\small{$\nu^{F[v]}\left(z,\frac{1}{2}v(z)\right)$}\end{tabular}}}}%
    \put(0.61833945,0.11311022){\color[rgb]{0,0,0}\makebox(0,0)[lt]{\lineheight{0.82591271}\smash{\begin{tabular}[t]{l}\tiny{$K^s$}\end{tabular}}}}%
    \put(0.70521574,0.14121727){\color[rgb]{0,0,0}\makebox(0,0)[lt]{\lineheight{0.82591271}\smash{\begin{tabular}[t]{l}\tiny{$O$}\end{tabular}}}}%
    \put(0.80486468,0.35654593){\color[rgb]{0,0,0}\makebox(0,0)[lt]{\lineheight{0.82591271}\smash{\begin{tabular}[t]{l}\small{$\nu^{K^s}$}\end{tabular}}}}%
    \put(0.2606094,0.14672583){\color[rgb]{0,0,0}\makebox(0,0)[lt]{\lineheight{0.82591271}\smash{\begin{tabular}[t]{l}\tiny{$z$}\end{tabular}}}}%
  \end{picture}%
\endgroup%

\caption{In this case conditions \textbf{R1} and \textbf{R2} are satisfied because the set of possible normals to $\partial^* F[v]$ is a subset of $\mathbb{V}_{K^s}$ (in fact coincides with it) . See also Remark \ref{rem: post Prop 1.12 2}.}
\label{fig:Rigiditacristallino}
\end{figure}
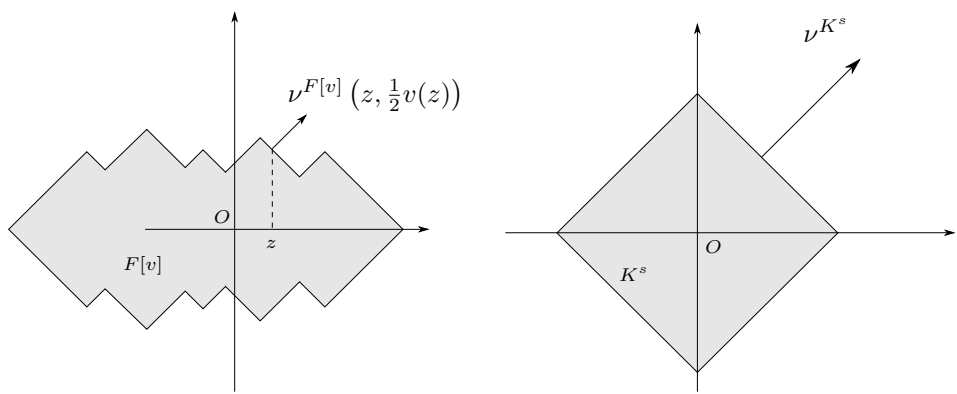

\begin{figure}[!htb]
\centering
\def\svgwidth{13cm}
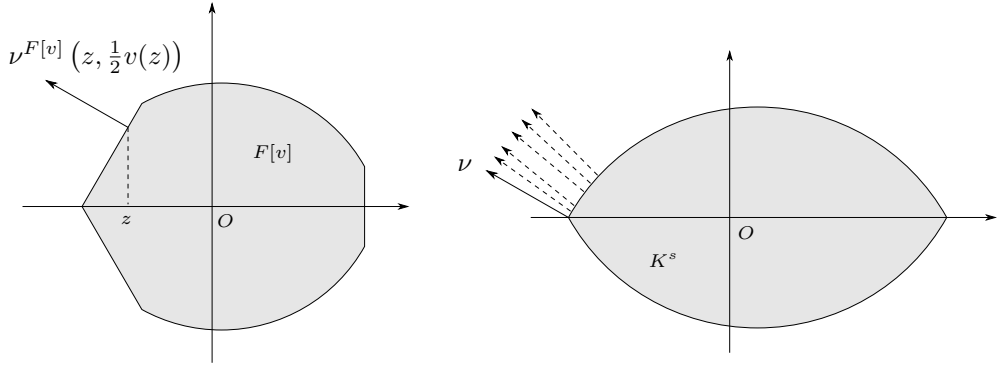
\caption{In this case, for the point $z$ in the left figure, we can find a sequence of vectors in $\mathbb{V}_{K^s}$ (see the dashed vectors in the right figure) whose limit $\nu$, coincides with $\nu^{F[v]}\left(z, \frac{1}{2}v(z)  \right)$. Therefore conditions \textbf{R1} and \textbf{R2} hold true. Note that $\nu\notin \mathbb{V}_{K^s}$.}
\label{fig:Rigiditanoncristallino}
\end{figure}

Let us stress that asking $K^s$ to be polyhedral does not automatically imply that \textbf{R1} and \textbf{R2} hold true. Indeed, by definition, the validity of such conditions depends on both $K^s$, and the function $v$. The importance of the relation between $K^s$ and the function $v$ for the validity of \textbf{R1} and \textbf{R2} is made even more explicit in Proposition \ref{prop: R1,R2}: there, it is given an operative characterization of conditions \textbf{R1} and \textbf{R2} in terms of the relation that have to occur between  the normal vectors to $\partial^*F[v]$, and the normal vectors to $\partial^*K^s$ (see indeed condition \emph{ii)} in Proposition \ref{prop: R1,R2}). An example of \textbf{R1} not being satisfied is indeed presented in Figure \ref{fig:wulff shapesZ}, where despite $K^s$ is a polyhedron, it is clear by construction that $\nu^{F[v]}_{RS},\nu^{F[v]}_{TU}\notin \mathbb{V}_{K^s}$, and so condition \emph{ii)} of Proposition \ref{prop: R1,R2} is not verified, implying that condition \textbf{R1} fails to be true. Nonetheless, in Figure \ref{fig:Rigiditacristallino}, choosing the same $K^s$ as in the previous example, but a different $v$, we show that condition \emph{ii)} of Proposition \ref{prop: R1,R2} holds true.  

\noindent
Different conclusions can be made if instead we have that $K^s$ has $C^1$ boundary. In that case, we have the following result.

\begin{corollary}\label{cor: smooth K^s}
Let $v$ be as in (\ref{due tilde}) and let $K\subset\R^n$ be as in (\ref{HP per K}). In addition, assume that $K^s$ has $C^1$ boundary. Then, conditions \textbf{R1}, \textbf{R2} hold true.
\end{corollary}

To conclude this section, we combine the results obtained in Proposition \ref{prop: R1,R2}, and Theorem \ref{thm: rigidity} to obtain the following proposition that can be considered the main contribution of the present work.

%As a direct consequence of Proposition \ref{prop: R1,R2}, Theorem \ref{thm: rigidity}, and also thanks to Corollary \ref{cor: smooth K^s} we have the following result, that can be regarded as the main result of the present work.

\begin{proposition}\label{final sum up}
Let $v$ be as in (\ref{due tilde}) and let $K\subset\R^n$ be as in (\ref{HP per K}). Let us assume in addition that there exists $S\subset\{ v\Low>0 \}$ such that $\mathcal{H}^{n-1}(S)=|D^cv|(S)=0$, and 
\begin{align*}
\nu^{F[v]}\left(z,\frac{1}{2}v(z)\right) \in \overline{\mathbb{V}_{K^s}}\quad \forall\,z\in \{v\Low>0\}\setminus S.
\end{align*}
Then, $\mathcal{M}_{K^s}(v)\subset\mathcal{M}(v)$. As a consequence, (\ref{rigidity steiner}) and (\ref{rigidity anisotropic steiner}) are equivalent.
\end{proposition}
  
\noindent
A simplified version of the above result is the following.

\begin{corollary}\label{crystal clear 1}
Let $v$ be as in (\ref{due tilde}) and let $K\subset\R^n$ be as in (\ref{HP per K}). Let us assume in addition that the set of outer unit normal vectors to $\partial^*F[v]$ is contained in the closure of the set of outer unit normal vectors to $\partial^*K^s$, namely that
\begin{align}\label{simplest condition}
\nu^{F[v]}(x)\in \overline{\mathbb{V}_{K^s}}\quad \forall\,x\in\partial^*F[v].
\end{align} 
Then, $\mathcal{M}_{K^s}(v)\subset\mathcal{M}(v)$. As a consequence, (\ref{rigidity steiner}) and (\ref{rigidity anisotropic steiner}) are equivalent.
\end{corollary}  

\noindent
Finally, we combine Theorem \ref{thm: rigidity} and Corollary \ref{cor: smooth K^s} to obtain the following result.

\begin{corollary}\label{crystal clear 2}
Let $v$ be as in (\ref{due tilde}) and let $K\subset\R^n$ be as in (\ref{HP per K}). Let us assume in addition that $K^s$ has $C^1$ boundary. Then, $\mathcal{M}_{K^s}(v)\subset\mathcal{M}(v)$. As a consequence, (\ref{rigidity steiner}) and (\ref{rigidity anisotropic steiner}) are equivalent.
\end{corollary}

It would be actually interesting checking whether conditions \textbf{R1} and \textbf{R2} are also necessary in order to get $\mathcal{M}_{K^s}(v)\subset \mathcal{M}(v)$. This seems quite a delicate problem, which we think is worth further investigation.
\section{Basic notions of Geometric Measure Theory}
\label{preliminaries}
The aim of this section is to introduce some tools from Geometric Measure Theory that will be largely used in the article. 
For more details the reader can have a look in the monographs 
\cite{AFP, GiaMoSou, MaggiBOOK, Simon}.
%\subsection{General notation in $\R^n$}\label{section general notation}
Note that even if part of the notations we will use, has been already presented across the Introduction, we briefly restate it in the next lines, in such a way that the reader can easily access to them. For $n \in \mathbb{N}$, we denote with $\mathbb{S}^{n-1}$
the unit sphere of $\mathbb{R}^n$, i.e. 
$$
\mathbb{S}^{n-1} = \{ x \in \R^n :  |x| = 1  \},
$$
and we set $\R^n_0:= \R^n \setminus \{ 0 \}$. For every $x=(x_1,\dots,x_n)\in\R^n$ we define $\p x= (x_1,\dots, x_{n-1})$, and $\q x= x_n$ are the "horizontal" and "vertical" projections respectively, so that $x=(\p x,\q x)$.
We denote by $e_1, \ldots, e_n$ the canonical basis in $\R^n$, and 
for every $x, y \in \R^n$, $x \cdot y$ stands for the standard scalar product in $\R^n$ between $x$ and $y$.
For every $r > 0$ and $x \in \R^n$, we denote by $B (x, r)$ the open ball of $\R^n$
with radius $r$ centred at $x$.
In the special case $x = 0$, we set $B(r):= B(0, r)$.
For every $x, y \in \R^n$, $x \cdot y$ stands for the standard scalar product in $\R^n$ between $x$ and $y$. We denote the $(n-1)$-dimensional ball in $\R^{n-1}$ of center $z\in\R^{n-1}$ and radius $r>0$ as
\begin{align}\label{eq: disc}
D_{z,r}&=\left\{\eta\in\R^{n-1}:\, |\eta-z|<r   \right\}.
\end{align}
For $x\in\R^n$ and $\nu\in \mathbb{S}^{n-1}$, we will denote by $H_{x,\nu}^+$ and $H_{x,\nu}^-$ the closed half-spaces
whose boundaries are orthogonal to $\nu$:
\begin{eqnarray}\label{Hxnu+}
  H_{x,\nu}^+:=\Big\{y\in\R^n:(y-x)\cdot\nu\ge 0\Big\},\quad  H_{x,\nu}^-:=\Big\{y\in\R^n:(y-x)\cdot\nu\le 0\Big\}\,.
\end{eqnarray}
If $1 \leq k \leq n$, we denote by $\mathcal{H}^k$ the $k$-dimensional Hausdorff measure in $\R^n$.
%Finally, we decompose $\R^n$ as the product $\R^{n-1}\times\R$, and denote by $\p:\R^n\to\R^{n-1}$ and $\q:\R^n\to\R$ the corresponding horizontal and vertical projections, so that $x=(\p x,\q x)=(x',x_n)$ and $x'=(x_1,\dots,x_{n-1})$ for every $x\in\R^n$. We set
%\begin{eqnarray*}
%\C_{x,r}&=&\Big\{y\in\R^n:|\p x- \p y|<r\,,|\q x- \q y|<r\Big\}\,,
%\\
%\D_{z,r}&=&\Big\{w\in\R^{n-1}:|w-z|<r\Big\}\,,
%\end{eqnarray*}
%for the vertical cylinder of center $x\in\R^n$ and radius $r>0$, and for the $(n-1)$-dimensional ball in $\R^{n-1}$ of center $z\in\R^{n-1}$ and radius $r>0$, respectively. In this way, $\C_{x,r}=\D_{\p x,r}\times(\q x-r,\q x+r)$. 
If $\{E_h\}_{h\in\mathbb{N}}$ is a sequence of Lebesgue measurable sets in $\R^n$
with finite volume, and $E \subset \R^n$ is also measurable with finite volume, 
we say that $\{E_h\}_{h\in\mathbb{N}}$ converges to $E$ as $h\to\infty$, and write $E_h\to E$, 
if $\mathcal{H}^n(E_h\Delta E)\to 0$ as $h\to\infty$. In the following, we will denote by $\chi_E$
the characteristic function of a measurable set $E \subset \R^n$.

\subsection{Density points} 
Let $E \subset \R^n$ be a Lebesgue measurable set and let $x\in\R^n$. 
The upper and lower $n$-dimensional densities of $E$ at $x$ are defined as
\begin{eqnarray*}
  \theta^*(E,x) :=\limsup_{r\to 0^+}\frac{\mathcal{H}^n(E\cap B(x,r))}{\omega_n\,r^n}\,,
  \qquad
  \theta_*(E,x) :=\liminf_{r\to 0^+}\frac{\mathcal{H}^n(E\cap B(x,r))}{\omega_n\,r^n}\,,
\end{eqnarray*}
respectively. 
It turns out that $x \mapsto  \theta^*(E,x)$ and $x \mapsto  \theta_*(E,x)$
are Borel functions that agree $\mathcal{H}^n$-a.e. on $\R^n$. 
Therefore, the $n$-dimensional density of $E$ at $x$
\[
\theta(E,x) := \lim_{r\to 0^+}\frac{\mathcal{H}^n(E\cap B(x,r))}{\omega_n\,r^n}\,,
\]
is defined for $\mathcal{H}^n$-a.e. $x\in\R^n$, and $x \mapsto  \theta (E,x)$
is a Borel function on $\R^n$.
Given $t \in [0,1]$, we set
$$
E^{(t)} :=\{x\in\R^n:\theta(E,x)=t\}.
$$
By the Lebesgue differentiation theorem, the pair $\{E^{(0)},E^{(1)}\}$ 
is a partition of $\R^n$, up to a $\mathcal{H}^n$-negligible set. 
%It is useful to keep in mind that
%\begin{eqnarray*}
%  &&x\in E^{(1)}\qquad\mbox{if and only if}\qquad E_{x,r}\toloc\R^n\quad\mbox{as $r\to 0^+$}\,,
%  \\
%  &&x\in E^{(0)}\qquad\mbox{if and only if}\qquad E_{x,r}\toloc\emptyset\quad\mbox{as $r\to 0^+$}\,,
%\end{eqnarray*}
%where $E_{x,r}$ denotes the blow-up of $E$ at $x$ at scale $r$, defined as
%\begin{eqnarray*}
%  E_{x,r}=\frac{E-x}{r}=\Big\{\frac{y-x}r:y\in E\Big\}\,,\qquad x\in\R^n\,,r>0\,.
%\end{eqnarray*}
The set $\partial^{\mathrm{e}} E :=\R^n\setminus(E^{(0)}\cup E^{(1)})$ is called the \textit{essential boundary} of $E$. 
%Note that, if $E$ is a measurable set, we only have $\mathcal{H}^n(\pae E)=0$, and in general 
%$\pae E$ may not be ``$(n-1)$-dimensional''.

\medskip

\subsection{Rectifiable sets and sets of finite perimeter}\label{section sofp} 
Let $1\le k\le n$, $k\in\mathbb{N}$. 
If $A, B \subset \R^n$ are Borel sets we say that 
$A \subset_{\mathcal{H}^k} B$ if $\mathcal{H}^k (B \setminus A) = 0$, and $A =_{\mathcal{H}^k} B$ if $\mathcal{H}^k (A \Delta B) = 0$,
where $\Delta$ denotes the symmetric difference of sets.
Let $M\subset\R^n$ be a Borel set.
We say that $M$ is {\it countably $\mathcal{H}^k$-rectifiable} if there exist 
Lipschitz functions $f_h:\R^k\to\R^n$ ($h\in\mathbb{N}$) such that $M\subset_{\mathcal{H}^k}\bigcup_{h\in\mathbb{N}}f_h(\R^k)$.
Moreover, we say that $M$ is {\it locally $\mathcal{H}^k$-rectifiable} if is {\it countably $\mathcal{H}^k$-rectifiable} and $\mathcal{H}^k(M\cap K)<\infty$ for every compact set $K\subset\R^n$, or, equivalently, if $\mathcal{H}^k\llcorner M$ is a Radon measure on $\R^n$. 
%Let $M$ be a locally $\mathcal{H}^k$-rectifiable set in $\R^n$, let $x \in \R^n$, and let $L$
%be a $k$-dimensional subspace of $\R^n$. 
%We say that $L$ is the {\it approximate tangent plane}
%of $M$ at $x$ if $\mathcal{H}^k\llcorner (M-x)/r\weak \mathcal{H}^k\llcorner L$ as $r\to 0^+$ weakly-star in the sense of Radon measures.
%If this is the case, we set $T_xM:= L$. It turns out that $T_xM$ exists and is uniquely defined at $\mathcal{H}^k$-a.e. $x\in M$. 
%Moreover, given two locally $\mathcal{H}^k$-rectifiable sets $M_1$ and $M_2$ in $\R^n$, 
%we have $T_xM_1=T_xM_2$ for $\mathcal{H}^k$-a.e. $x\in M_1\cap M_2$.
 Given a $\R^m$-valued Radon measure $\mu$ on $\R^n$, we define its \emph{total variation} $|\mu|$ as 
\begin{align}\label{def: tot.var. for a measure}
|\mu|(\Omega)=\sup\left\{\int_{\R^n}\varphi(x)\cdot d\mu(x):\,\varphi\in C^{\infty}_c(\Omega;\R^m),\,|\varphi|\leq 1  \right\},\quad \forall\,\Omega\subset\R^n\textit{ open.}
\end{align}
If we consider a generic Borel set $B\subset \R^n$ then
\begin{align*}
|\mu|(B)=\inf \left\{|\mu|(\Omega):\, B\subset \Omega,\, \Omega\subset \R^n\textit{\emph{ open set}}  \right\}.
\end{align*}
Let $\mu$ be a Radon measure on $\R^n$, let $1\leq p< \infty$ and $m\geq 1 $ with $m\in\mathbb{N}$. The vector space $L^p(\R^n,\mu;\R^m)$ is defined as 
\begin{align*}
L^p(\R^n,\mu;\R^m)=\left\{f:\R^n\rightarrow \R^m:\,f\textit{ is $\mu$-measurable, }\int_{\R^n}|f|^p d\mu < \infty   \right\},
\end{align*}
equipped with the norm
$$
\|f\|_{L^p(\R^n,\mu;\R^m)}=\left(\int_{\R^n}|f|^p d\mu   \right)^{\frac{1}{p}}.
$$
If $p=\infty$ then $L^\infty(\R^n,\mu;\R^m)$ is defined as
\begin{align*}
L^\infty(\R^n,\mu;\R^m)=\left\{f:\R^n\rightarrow \R^m:\,f\textit{ is $\mu$-measurable, } \esssup_{\R^n}f < \infty   \right\},
\end{align*}
where 
$$
\esssup_{\R^n}f:= \inf \left\{ c>0:\, \mu\left( \{|f|>c \} \right)=0 \right\}.
$$
We equip this space with the norm
$$
\|f\|_{L^\infty(\R^n,\mu;\R^m)}=\esssup_{\R^n}f.
$$
We say that $f\in  L_{loc}^p(\R^n,\mu;\R^m)$, $1\leq p \leq \infty$ if $f\in  L^p(C,\mu;\R^m)$ for every compact set $C\subset \R^n$.
\begin{remark}\label{rem: 4.8 Maggi}
Let $\mu$ be a Radon measure on $\R^n$ and let $f\in L^1_{\textit{loc}}(\R^n,\mu;\R^m)$ with $m\geq 1$, $m\in\mathbb{N}$. Then, we define a $\R^m$-valued Radon measure on $\R^n$ by setting 
\begin{align*}
f\mu(B)=\int_{B}f(x)\,d\mu(x)\quad \forall\,\textit{Borel set }B\subset \R^n.
\end{align*}
Its total variation is then defined as 
\begin{align*}
|f\mu|(B)=\int_B|f(x)|d\mu(x)\quad \forall\,\textit{Borel set }B\subset \R^n.
\end{align*}
For more details see \cite[Example 4.6, Remark 4.8]{MaggiBOOK}.
\end{remark}

\noindent
A Lebesgue measurable set $E\subset\R^n$ is said of {\it locally finite perimeter} in $\R^n$ if there exists a $\R^n$-valued Radon measure $\mu_E$, called the {\it Gauss--Green measure} of $E$, such that
\[
\int_E\nabla\varphi(x)\,dx=\int_{\R^n}\varphi(x)\,d\mu_E(x)\,,\qquad\forall \varphi\in C^1_c(\R^n)\,,
\]
where $C^1_c (\R^n)$ denotes the class of $C^1$ functions in $\R^n$
with compact support.
The relative perimeter of $E$ in $A\subset\R^n$ is then defined by setting $P(E;A):=|\mu_E|(A)$
for any Borel set $A \subset \R^n$.
The perimeter of $E$ is then defined as $P(E):=P(E;\R^n)$.
If $P(E) < \infty$, we say that $E$ is a set of {\it finite perimeter} in $\R^n$.
The {\it reduced boundary} of $E$ is the set $\partial^*E$ of those $x\in\R^n$ such that
\begin{align*} 
\nu^E(x)=\frac{d\mu_E}{d|\mu_E|}(x)=\lim_{r\rightarrow 0^+}\frac{\mu_E(B(x,r))}{|\mu_E|(B(x,r))}\quad \textit{\emph{exists and belongs to $\mathbb{S}^{n-1}$}},
\end{align*}
where $\frac{d\mu_E}{d|\mu_E|}$ indicates the derivative of $\mu_E$ with respect its total variation $|\mu_E|$ in the sense of Radon measure. The Borel function $\nu^E:\partial^*E\to \mathbb{S}^{n-1}$ 
is called the {\it measure-theoretic outer unit normal} to $E$. 
%When $x \in \partial^*E$, we will use the decomposition 
%$\nu^E (x) = (\nu^E_{x'} (x), \nu^E_{y} (x))$, with $\nu^E_{x'} (x) = (\nu^E_1 (x), \ldots, \nu^E_{n-1} (x)) \in \R^{n-1}$,  
%and $\nu^E_y (x) \in \R$.
If $E$ is a set of locally finite perimeter, it is possible to show
that $\partial^*E$ is a locally $\mathcal{H}^{n-1}$-rectifiable set in $\R^n$ 
\cite[Corollary 16.1]{MaggiBOOK}, with $\mu_E=\nu^E\,\mathcal{H}^{n-1}\mres\partial^*E$, and
\[
\int_E\nabla\varphi(x)\,dx=\int_{\partial^*E}\varphi(x)\,\nu^E(x)\,d\mathcal{H}^{n-1}(x)\,,\qquad\forall \varphi\in C^1_c(\R^n).
\] 
Thus, $P(E;A)=\mathcal{H}^{n-1}(A\cap\partial^*E)$ for every Borel set $A\subset\R^n$. 
%We say that $x\in\R^n$ is a jump point of $E$, if  there exists $\nu\in \mathbb{S}^{n-1}$ such that
%\begin{equation}
%  \label{jump point of E}
%  E_{x,r}\toloc H_{0,\nu}^+\,,\qquad\mbox{as $r\to 0^+$}\,,
%\end{equation}
%and we denote by $\partial^JE$ the set of {\it jump points} of $E$. Notice that we always have $\partial^JE\subset E^{(1/2)}\subset\partiale E$. In fact, if $E$ is a set of locally finite perimeter and $x\in\partial^*E$, then \eqref{jump point of E} holds with $\nu=-\nu_E(x)$, so that $\partial^*E\subset\partial^JE$. 
If $E$ is a set of locally finite perimeter, it turns out that 
\begin{equation*}
  \label{inclusioni frontiere}
  \partial^*E  \subset E^{(1/2)} \subset \partial^{\mathrm{e}} E\,.
\end{equation*}
Moreover, {\it Federer's theorem} holds true (see \cite[Theorem 3.61]{AFP} and \cite[Theorem 16.2]{MaggiBOOK}):
\[
\mathcal{H}^{n-1}(\partial^{\mathrm{e}} E\setminus\partial^*E)=0\,,
\]
thus implying that the essential boundary $\partial^{\mathrm{e}} E$ of $E$ is locally $\mathcal{H}^{n-1}$-rectifiable in $\R^n$. 

\subsection{General facts about measurable functions}
Let $f:\R^n\to\R$ be a Lebesgue measurable function.
We define the {\it approximate upper limit} $f^\vee(x)$ and the {\it approximate lower limit} 
$f^\wedge(x)$ of $f$ at $x\in\R^n$ as 
\begin{eqnarray}
  \label{def fvee}
  f^\vee(x)=\inf\Big\{t\in\R: x\in \{f>t\}^{(0)}\Big\}\,,
  \\
  \label{def fwedge}
  f^\wedge(x)=\sup\Big\{t\in\R: x\in \{f<t\}^{(0)}\Big\}\,.
\end{eqnarray}
We observe that $f^\vee$ and $f^\wedge$ are Borel functions that are defined at 
{\it every} point of $\R^n$, with values in $\R\cup\{\pm\infty\}$.
Moreover, if $f_1: \R^n \to \R$ and $f_2: \R^n \to \R$
are measurable functions satisfying $f_1=f_2$ $\mathcal{H}^n$-a.e. on $\R^n$, 
then $f_1^\vee=f_2^\vee$ and $f_1^\wedge=f_2^\wedge$ {\it everywhere} on $\R^n$. 
We define the {\it approximate discontinuity} set $S_f$ of $f$ as
\begin{align}\label{eq: approx disc set}
S_f: =\{f^\wedge<f^\vee\}.
\end{align}
Note that, by the above considerations, it follows that $\mathcal{H}^n(S_f)=0$. 
Although $f^\wedge$ and $f^\vee$ may take infinite values on $S_f$, 
the difference $f^\vee(x)-f^\wedge(x)$ is well defined in $\R\cup\{\pm\infty\}$ for every $x\in S_f$.
Then, we can define the {\it approximate jump} $[f]$ of $f$ as the Borel function $[f]:\R^n\to[0,\infty]$ given by
  \begin{eqnarray*}
    [f](x):=\left\{\begin{array}{l l}
      f^\vee(x)-f^\wedge(x)\,,&\mbox{if $x\in S_f$}\,,
      \vspace{.2cm} \\
      0\,,&\mbox{if $x\in \R^n\setminus S_f$}\,.
    \end{array}
    \right .
  \end{eqnarray*}

The \emph{approximate average} of $f$ is the Borel function 
\begin{align*}
\tilde{f}(x)=
\begin{cases}
\frac{f\Upp(x)+f\Low(x)}{2} , &\mbox{if } x\in \R^n\setminus\{ f\Low=-\infty, f\Upp=+\infty \},\\
0,  &\mbox{if } x\in \{ f\Low=-\infty, f\Upp=+\infty \}.
\end{cases}
\end{align*}
It also holds the following limit relation
\begin{align}\label{eq: limit relation}
\tilde{f}(x) = \lim_{M\rightarrow \infty}\widetilde{\tau_Mf}(x)=\lim_{M\rightarrow \infty} \frac{\tau_M(f\Upp)+\tau_M(f\Low)}{2},\quad \forall x\in \R^n,
\end{align}
that we want to be true for every Lebesgue measurable function $f:\R^n \rightarrow \R$, where, here and in the rest of the work,  
\begin{align}\label{eq: come definita tau_M}
\tau_M(s)=\max \{-M, \min\{ M,s \} \}, \quad s\in\R \cup \{\pm \infty \}.
\end{align}
By definition, $\tau_M$ is equivalently defined as
\begin{align*}
\tau_M(s)=\begin{cases}
M  \quad &s> M \\ s  &-M\leq s\leq M \\ -M  &s<-M
\end{cases}
\end{align*}
and the following properties can be easily proved
\begin{align}
\tau_M(s_2) \geq \tau_M&(s_1)\qquad \forall\,s_2\geq s_1,\, \textit{provided }M>0.\label{eq: alphas}\\
\tau_{M_2}(s) \geq \tau_{M_1}&(s)\qquad \forall\,M_2\geq M_1\geq 0,\, \textit{provided }s\geq0.\label{eq: betas1}\\
\tau_{M_2}(s) \leq \tau_{M_1}&(s)\qquad \forall\,M_2\geq M_1\geq 0,\, \textit{provided }s\leq 0.\label{eq: betas2}\\
(\tau_{M_2}-\tau_{M_1})(s_2) \geq (\tau_{M_2}-\tau_{M_1})&(s_1)\qquad \forall\,s_2\geq s_1,\, \textit{provided } M_2\geq M_1\geq 0.\label{eq: gammas}\\
\tau_{M_2}(s_2)-\tau_{M_2}(s_1) \geq \tau_{M_1}(s_2)-\tau_{M_1}&(s_1)\qquad \forall\,M_2\geq M_1\geq 0,\, \textit{provided } s_2\geq s_1.\label{eq: deltas}
\end{align}
The validity of the limit relation (\ref{eq: limit relation}) can be easily checked noticing that 
\begin{align*}
\tau_M(f)\Low= \tau_M(f\Low),\quad \tau_M(f)\Upp= \tau_M(f\Upp),\quad \widetilde{\tau_M(f)}(x)= \frac{\tau_M(f\Upp)+\tau_M(f\Low)}{2}, \quad \forall x \in \R^n.  
\end{align*}
Using these above definitions, the validity of the following properties can be easily deduced. For every Lebesgue measurable function $f:\R^n \rightarrow \R$ and for every $t\in \R$ we have that
\begin{align}
\{|f|\Upp < t  \} &= \{-t<f\Low  \} \cap \{f\Upp < t   \},\label{eq: 2.7 Filippo page 15}\\
\{ f\Upp < t \}&\subset \{f < t  \}^{(1)} \subset \{ f\Upp \leq t  \},\label{eq: 2.8 Filippo page 15}\\
\{f\Low > t  \} &\subset \{ f> t  \}^{(1)} \subset \{f \Low \geq t  \}.\label{eq: 2.9 Filippo page 15}
\end{align}
Furthermore, if $f,g:\R^n \rightarrow \R$ are Lebesgue measurable functions and $f=g$ $\mathcal{H}^n$-a.e. on a Borel set $E$, then
\begin{align}\label{eq: 2.10 Filippo pag 16}
f\Upp(x)= g\Upp(x), \quad f\Low(x)=g\Low(x),\quad [f](x)=[g](x), \quad \forall x\in E^{(1)}.
\end{align} 
Let $A\subset\R^n$ be a Lebesgue measurable set. 
We say that $t\in\R\cup\{\pm\infty\}$ is the approximate limit of $f$ at $x$ with respect to $A$, 
and write $t=\aplim (f,A,x)$, if
\begin{eqnarray}
  &&\theta\Big(\{|f-t|>\varepsilon\}\cap A;x\Big)=0\,,\qquad\forall\varepsilon>0\,,\hspace{0.3cm}\qquad (t\in\R)\,,
  \\
  &&\theta\Big(\{f<M\}\cap A;x\Big)=0\,,\qquad\hspace{0.6cm}\forall M>0\,,\qquad (t=+\infty)\,,
  \\
  &&\theta\Big(\{f>-M\}\cap A;x\Big)=0\,,\qquad\hspace{0.3cm}\forall M>0\,,\qquad (t=-\infty)\,.
\end{eqnarray}
We say that $x\in S_f$ is a \textit{jump point} of $f$ if there exists $\nu\in \mathbb{S}^{n-1}$ such that
\[
f^\vee(x)=\aplim(f,H_{x,\nu}^+,x)> f^\wedge(x)=\aplim(f,H_{x,\nu}^-,x)\,.
\]
If this is the case, we say that $\nu_f(x):= \nu$ is the approximate jump direction of $f$ at $x$.
If we denote by $J_f$ the set of approximate jump points of $f$, we have that $J_f\subset S_f$ 
and $\nu_f:J_f\to \mathbb{S}^{n-1}$ is a Borel function. 
%Finally, we say that $f$ is {\it approximately differentiable} at $x\in S_f^c = \R^m \setminus S_f$ provided $f^\wedge(x)=f^\vee(x)\in\R$ 
%and there exists $\xi\in\R^m$ such that
%\[
%\aplim(g_{\xi},\R^n,x)=0\,,
%\]
%where $g_{\xi}(y):=(f(y)-\widetilde{f}(x)-\xi\cdot(y-x))/|y-x|$ for $y\in\R^n\setminus\{x\}$, 
%and $\widetilde{f}(x):= f^\wedge(x)=f^\vee(x)$.
%If this is the case, then $\xi$ is uniquely determined, we set $\nabla f(x):= \xi$, and call $\nabla f(x)$ 
%the {\it approximate differential} of $f$ at $x$. 

%We shall need at several occasions to use the following very fine criterion for finite perimeter, known as {\it Federer's criterion} \cite[4.5.11]{FedererBOOK} (see also \cite[Theorem 1, section 5.11]{EvansGariepyBOOK}): if $E$ is a Lebesgue measurable set in $\R^n$ such that $\pae E$ is locally $\mathcal{H}^{n-1}$-finite, then $E$ is a set of locally finite perimeter.
Consider $f:\R^n\rightarrow \R$ Lebesgue measurable, then we say that $f$ is \emph{approximately differentiable} at $x\in S_f^c$ provided $f\Low(x)=f\Upp(x)\in \R$ if there exists $\xi \in \R^n$ such that
\begin{align*}
\aplim(g,\R^n,x)=0,
\end{align*} 
where $g(y)= (f(y)-\tilde{f}(x)-\xi\cdot(y-x))/|y-x|$ for $y \in \R^n\setminus\{x \}$.
If this is the case, then $\xi$ is uniquely determined, we set $\xi=\nabla f(x)$, and call $\nabla f(x)$ the \emph{approximate differential} of $f$ at $x$. The localization property (\ref{eq: 2.10 Filippo pag 16}) holds true also for the approximate differentials, namely if $g,f:\R^n\rightarrow \R$ are Lebesgue measurable functions, $f=g$ $\mathcal{H}^n$-a.e. on a Borel set $E$, and $f$ is approximately differentiable $\mathcal{H}^n$-a.e. on $E$, then so it is $g$ $\h^n$-a.e. on $E$ with
\begin{align}\label{eq: 2.12 FIlippo pag 16}
\nabla f(x)=\nabla g(x), \quad \textit{\emph{for $\h^n$-a.e. $x\in E$}}.
\end{align}

\subsection{Functions of bounded variation} %\label{section sofp} Let $1\le k\le n$, $k\in\N$. 
Let $f:\R^n\to\R$ be a Lebesgue measurable function, 
and let $\Omega\subset\R^n$ be open. 
We define the {\it total variation of $f$ in $\Omega$} as
\[
|Df|(\Omega)=\sup\Big\{\int_\Omega\,f(x)\,\diver\,T(x)\,dx:T\in C^1_c(\Omega;\R^n)\,,|T|\le 1\Big\}\,,
\]
where $C^1_c(\Omega;\R^n)$ is the set of $C^1$ functions from 
$\Omega$ to $\R^n$ with compact support.
We also denote by $C^0_c (\Omega ; \R^n)$ the class of all
continuous functions from $\Omega$ to $\R^n$.
Analogously, for any $k \in \mathbb{N}$, the class of $k$ times continuously differentiable functions from $\Omega$ to $\R^n$ is denoted by $C^k_c (\Omega ; \R^n)$. We say that $f$ belongs to the space of functions of bounded variations, $f \in BV(\Omega)$, if $|Df|(\Omega)<\infty$ and $f\in L^1(\Omega)$. Moreover, we say that $f\in BV_{\textnormal{loc}}(\Omega)$ if $f\in BV(\Omega')$ for every open set $\Omega'$ compactly contained in $\Omega$. 
Therefore, if $f\in BV_{\textnormal{loc}}(\R^n)$ the distributional derivative $Df$ of $f$ is an $\R^n$-valued Radon measure. 
In particular, $E$ is a set of locally finite perimeter if and only if $\chi_E\in BV_{\textnormal{loc}}(\R^{n})$. 
%Sets of finite perimeter and functions of bounded variation are related by the fact that, if $f\in BV_{loc}(\R^n)$, then, for a.e. $t\in\R$, $\{f>t\}$ is a set of finite perimeter, and the {\it coarea formula},
%\begin{equation}
%  \label{coarea bv}
%  \int_\R\,P(\{f>t\};G)\,dt=|Df|(G)\,,
%\end{equation}
%holds (as an identity in $[0,\infty]$) for every Borel set $G\subset\R^n$. 
If $f\in BV_{loc}(\R^n)$, one can write the  Radon--Nykodim decomposition of $Df$ with respect to $\mathcal{H}^n$
as $Df=D^af+D^sf$, where $D^sf$ and $\mathcal{H}^n$ are mutually singular, and where $D^af\ll\mathcal{H}^n$. 
We denote the density of $D^af$ with respect to $\mathcal{H}^n$ by $\nabla f$, 
so that $\nabla\,f\in L^1(\Omega;\R^n)$ with $D^af=\nabla f\,d\mathcal{H}^n$. Moreover, for a.e. $x\in\R^n$, $\nabla f(x)$ is the approximate differential of $f$ at $x$. 
If $f\in BV_{\textnormal{loc}}(\R^n)$, then $S_f$ is countably $\mathcal{H}^{n-1}$-rectifiable.
Moreover,  we have  $\mathcal{H}^{n-1}(S_f\setminus J_f)=0$, $[f]\in L^1_{loc}(\mathcal{H}^{n-1}\llcorner J_f)$, 
and the $\R^n$-valued Radon measure $D^jf$ defined as
\[
D^jf=[f]\,\nu_f\,d\mathcal{H}^{n-1}\llcorner J_f\,,
\]
is called the {\it jump part of $Df$}. 
If we set $D^cf=D^sf-D^jf$, we have that 
$D f = D^af+D^jf+D^cf$.
The $\R^n$-valued Radon measure $D^cf$ is called the {\it Cantorian part} of $Df$, 
and it is such that $|D^cf|(M)=0$ for every $M \subset \R^n$ 
which is $\sigma$-finite with respect to $\mathcal{H}^{n-1}$. Let us recall some useful properties we will need on the next sections (see \cite[Lemma 2.2, Lemma 2.3]{CCPMSteiner} for further details).
\begin{lemma}\label{lem: cantor total variation of v on v=0 is null}
If $v\in BV(\R^n)$, then $|D^cv|(\{v\Low=0  \})=0$. In particular, if $f=g$ $\mathcal{H}^n$-a.e. on a Borel set $E\subset \R^n$, then $D^cf\mres E^{(1)}= D^cg\mres E^{(1)}$.
\end{lemma}
\begin{lemma}\label{lem: 2.3 FIlippo page 18}
If $f,g \in BV(\R^n)$, $E$ is a set of finite perimeter and $f=1_Eg$, then
\begin{align}
\nabla f&=1_E \nabla g,\quad \quad \quad \mathcal{H}^n\textit{\emph{- a.e. on }}\R^n,\\
D^cf&= D^cg\mres E^{(1)},\\
S_f\cap E^{(1)} &= S_g \cap E^{(1)}. 
\end{align}
\end{lemma}
\noindent
A Lebesgue measurable function $f:\R^n\rightarrow \R$, it's called of \emph{generalized bounded variation} on $\R^n$, shortly $f\in GBV(\R^n)$ if and only if $\tau_M(u)\in BV_{loc}(\R^{n-1})$ for every $M>0$ (where $\tau_M(s)$ has been defined in the previous subsection). It is interesting to notice that the structure theory of BV-functions holds true for GBV-functions too. Indeed, given $f\in GBV(\R^n)$, then, (see \cite[Theorem 4.34]{AFP}) $\{ f>t \}$ is a set of finite perimeter too for $\mathcal{H}^1$-a.e. $t\in \R$, $f$ is approximately differentiable $\mathcal{H}^n$-a.e. on $\R^n$, $S_f$ is countably $\mathcal{H}^{n-1}$-rectifiable  and $\mathcal{H}^{n-1}$-equivalent to $J_f$ and the usual coarea formula takes the form 
\begin{align*}
\int_{\R}P(\{ f> t \};G)dt = \int_{G}|\nabla f|d\mathcal{H}^n + \int_{G\cap S_f}[f]d\mathcal{H}^{n-1} + |D^cf|(G),
\end{align*}
for every Borel set $G\subset \R^n$, where $|D^c f|$ denotes the Borel measure on $\R^n$ defined as
\begin{align}\label{def: Cantor part GBV}
|D^c f|(G)=\lim_{M\rightarrow +\infty}|D^c(\tau_M(f))|(G)=\sup_{M>0}|D^c(\tau_M)(f)|(G),
\end{align}
whenever $G$ is a Borel set in $\R^n$.

\section{Setting of the problems and preliminary results}\label{section 3}
We recall in here, few results that will be useful later on for the proof of (\ref{eq:anisotropic Steiner inequality}) (for more details see \cite[Section 2 and 3]{CCF}). Let us start with a version of a result by Vol'pert (see \cite[Theorem G]{CCF}).
\begin{theorem}\label{thm: Volpert}
Let $v\in BV(\R^{n-1})$ such that $\h^{n-1}(\{v>0  \})<\infty$. Let $E\subset \R^{n}$ be a $v$-distributed set of finite perimeter. Then, we have for $\mathcal{L}^{n-1}$-a.e. $z\in\R^{n-1}$,
\begin{equation}\label{eq: 2.13 CCF}
E_z\textit{ has finite perimeter in $\R$};
\end{equation}
\begin{equation}\label{eq: 2.14 CCF}
(\partial^e E)_z= (\partial^* E)_z=\partial^*( E_z)=\partial^e( E_z);
\end{equation}
\begin{equation}\label{eq: 2.15 CCF}
\q (\nu^E(z,t)) \neq 0 \textit{ for every $t$ such that }(z,t)\in \partial^*E;
\end{equation}
In particular, there exists a Borel set $G_E\subseteq \{ v>0 \} $ such that $\mathcal{L}^{n-1}(\{v > 0\}\setminus G_E)=0$ and (\ref{eq: 2.13 CCF})-(\ref{eq: 2.15 CCF}) are satisfied for every $z\in G_E$.
\end{theorem}
\noindent
The next result is a version of the \emph{Coarea formula} for rectifiable sets (see \cite[Theorem F]{CCF}).
\begin{theorem}
Let $E$ be a set of finite perimeter in $\R^n$ and let $g:\R^n\rightarrow [0,+\infty]$ be any Borel function. Then,
\begin{align}\label{eq: Coarea CCF}
\int_{\partial^*E}g(x)|\q (\nu^E(x))|d\h^{n-1}(x)=\int_{\R^{n-1}}dz \int_{(\partial^*E)_{z}}g(z,y)d\h^0(y).
\end{align}
\end{theorem}
\noindent
Lastly, next result is a version of \cite[Lemma 3.2]{CCF}.
\begin{lemma}\label{lem: 3.2 CCF}
Let $v\in BV(\R^{n-1})$ such that $\h^{n-1}(\{v>0  \})<\infty$. Let $E\subset \R^{n}$ be a $v$-distributed set of finite perimeter. Then, for $\mathcal{L}^{n-1}$-a.e. $z\in \{v>0  \}$
\begin{align*}
\frac{\partial v}{\partial x_i}(z)=-\int_{(\partial^*E)_z}\frac{\nu_i^E(z,y)}{|\q(\nu^E(z,y))|}\,d\h^0(y),\qquad i=1,\dots,{n-1},
\end{align*}
In particular by (\ref{eq: 2.14 CCF}) and the above relation, we get for $\mathcal{L}^{n-1}$-a.e. $z\in \{v>0  \}$
\begin{align*}
\frac{1}{2}\frac{\partial v}{\partial x_i}(z)=-\frac{\nu_i^{F[v]}(z,y)}{|\q(\nu^{F[x]}(z,y))|}\,d\h^0(y),\qquad i=1,\dots,{n-1},\, y\in(\partial^*F[z])_z.
\end{align*}
\end{lemma}

\subsection{Properties of the surface tension $\phi_K$}
Let us start recalling some basic facts about the surface tension $\phi_K$. First of all, let us sum up some known properties of the gauge function in the following result, that can be easily deduced from \cite[Proposition 20.10]{MaggiBOOK}.
\begin{proposition}\label{prop: Maggi 20.10}
Consider $K\subset \R^n$ as in (\ref{HP per K}). Consider $\phi_K\,,\phi_K^*:\R^n\rightarrow [0,\infty)$ the corresponding surface tension and gauge function defined in (\ref{def surface tension}), (\ref{def: gauge function}) respectively. Then the following properties hold true.
\begin{itemize}
\item[i)] The function $\phi_K^*$ is one-homogeneous, convex and coercive on $\R^n$ and there exist positive constants $c$ and $C$ such that 
\begin{align*}
c|x|\leq \phi_K(x)\leq C|x|,\quad &\forall x\in\R^n,\\
\frac{|x|}{C}\leq \phi_K^*(x)\leq \frac{|x|}{c},\quad &\forall x\in\R^n.
\end{align*}
\item[ii)]The so called \emph{Fenchel inequality} holds true i.e.
\begin{align}\label{eq: Fenchel inequality}
x\cdot y \leq \phi_K^*(x)\phi_K(y),\quad \forall x,y\in\R^n.
\end{align}
\item[iii)] The gauge function $\phi_K^*$ provides a new characterization for the Wulff shape $K$ i.e.
\begin{align*}
K=\left\{x\in \R^n:\, \phi_K^*(x)<1 \right\},
\end{align*}
from which we can immediately derive that 
\begin{align*}
\phi_K(x)&= \sup\left\{x\cdot y:\, \phi^*_K(x)< 1  \right\},\\
\phi_K(x)&=(\phi_K^*)^*(x).
\end{align*} 
\item[iv)] If $x\in \partial^*K$ and $y\in \mathbb{S}^{n-1}$, then equality holds in (\ref{eq: Fenchel inequality}) if and only if $y=\nu^{K}(x)$; in particular 
\begin{align}\label{eq: P_K(K)}
P_K(K)=n|K|.
\end{align}
\end{itemize}
\end{proposition}

\begin{remark}
By (i) of Proposition \ref{prop: Maggi 20.10} we have that $E$ is a set of locally finite perimeter if and only if $E$ is a set of locally finite anisotropic perimeter i.e. $P_K(E;C)<\infty$ for every $C\subset \R^n$ compact set.
\end{remark}

\begin{figure}[!htb]
\centering
\def\svgwidth{13cm}
\input{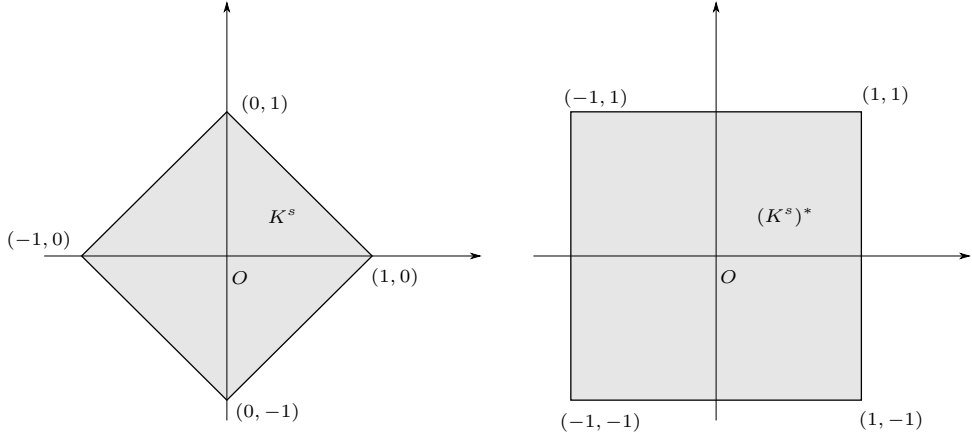}
\caption{A two dimensional example of $K^s$ and its dual $(K^s)^*$. }
\label{fig:K e K^*}
\end{figure}   
%%% parte del caption che ora non serve

%According with Lemma \ref{lem: sub-diff e bordo wulff} we see that $\partial\phi_{K^s}^*((0,1))$ is a convex subset of the boundary of $(K^s)^*$. The fact that $\partial\phi_{K^s}^*((0,1))$ actually contains the point $(0,1)$ is just a consequence of the specific Wulff shape considered in the example.

%%%%
\begin{remark}\label{rem: abouth gauge function}
Thanks to $iii)$ of the above
proposition we have
$$K^*=\{x\in \R^n:\, \phi_K(x)<1  \},$$
from which together with (\ref{def: gauge function}) gives 
\begin{align*}
\phi_K^*(x)=\sup\{x\cdot y:\, y\in K^*  \}\quad \forall\, x\in \R^n.
\end{align*}
For a pictorial idea of $K$ and $K^*$ see for instance Figure \ref{fig:K e K^*}. Furthermore, observe that
\begin{align}\label{eq: Phi=1, Phi* =1}
\phi_K(x)=1 \quad \forall\,x\in \partial K^*,\\
\phi^*_K(x)=1 \quad \forall\,x\in \partial K.
\end{align}
\end{remark}
\begin{remark}
Let us consider $K\subset \R^n$ as in  (\ref{HP per K}). According to Proposition \ref{prop: Maggi 20.10}, iii) another way to define the Wulff shape $K$ is
\begin{align*}
K:= \p\left( \Sigma_{\phi^{*}_{K}} \cap \{x_{n+1}=1\} \right),
\end{align*}
where $\Sigma_{\phi^{*}_{K}}$ is the epigraph of $\phi^{*}_{K}$ in $\R^{n+1}$ and $\textbf{p}:\R^{n+1}\rightarrow \R^{n}$ corresponds to the horizontal  projection. By the one-homogeneity of $\phi_{K}$ we get that
\begin{align}\label{linear along radial dir.}
\phi_{K}(tx)=t|x|\phi_{K}\left(\frac{tx}{t|x|}  \right)=t\phi_{K}(x)\quad \forall x\in \R^n\setminus \{0 \},\, \forall t>0.
\end{align}
By (\ref{linear along radial dir.}), we get for every constant $\lambda>0$ that
\begin{align*}
\lambda K:= \p\left( \Sigma_{\phi^{*}_{K}} \cap \{x_{n+1}=\lambda\} \right).
\end{align*}
Another thing we would like to 
observe is that given $x,y\in \R^n$ with $x\in \lambda K$ and $y\in (\lambda K)^{c}$, (for some $\lambda >0$) then $\phi^{*}_{K}(x) < \phi^{*}_{K}(y)$. Naturally, these considerations hold true for $K^*$ and $\phi_{K}$ too.
\end{remark}
\begin{definition}[Sub-differential]\label{def: sub-differential}
Let $\varphi:\R^n\rightarrow [0,\infty]$ be a convex function. Let us fix $x_0 \in \R^n$ and consider all vectors $y_0 \in \R^n$ such that
\begin{align}
\varphi(z)\geq  \phi(x_0) +y_0\cdot (z-x_0)\quad \forall z \in \R^n.
\end{align}
The set of all vectors $y_0$ satisfying the above property is called \emph{sub-differential} of $\varphi$ at $x_0$ and we indicate it by $\partial \varphi(x_0)$.
\end{definition}
\noindent 
Keeping in mind Definition \ref{def:weak sub-differential} we have the following remarks.
\begin{remark}\label{rem: convexity of C_K^*}
For every $x_0\in\R^n$, the sub-differential $\partial\phi(x_0)$ is a closed and convex set of $\R^n$ (see \cite{Rock} chapter 5). From this, it can be proved that, given $x\in \partial K$, also $C^*_K(x)$ is a convex set of $\R^n$, where $C^*_K(x)$ is defined as in (\ref{def:weak sub-differential}).
\end{remark}
\begin{remark}\label{rem: differentiability and sub differential}
Let $\phi:\R^n\rightarrow [0,\infty]$ be a convex function. It is a well known result about convex functions that, $\phi$ is differentiable in $x_0\in\R^n$ if and only if   $\partial \phi(x_0)$ consists of only one element. In that situation, we call $\nabla \phi(x_0)$ is the only element in the sub-differential $\partial\phi(x_0)$. 
\end{remark}

\begin{definition}\label{def: anisotropic total variation}
Fix an integer $m\geq 1$ and let $K\subset \R^n$ be as in (\ref{HP per K}). Given a $\R^n$-valued Radon measure $\mu$ on $\R^m$ and a generic Borel set $F \subset \R^m$, we define the $\phi_K$-anisotropic total variation of $\mu$ on $F$ as 
\begin{align*}
|\mu|_{K}(F)= \int_F \phi_K\left( \frac{d\mu}{d|\mu|}(x) \right) d|\mu|(x).
\end{align*}
\end{definition}

\begin{remark}\label{rem: |mu|_K << |mu|}
By condition $i)$ in Proposition \ref{prop: Maggi 20.10} we have that 
\begin{align*}
|\mu|_K(F)=\int_F \phi_K\left( \frac{d\mu}{d|\mu|}(x) \right) d|\mu|(x)\leq C\int_F d|\mu|(x)= C|\mu|(F).
\end{align*}
Analogously,
\begin{align*}
|\mu|(F)=\int_F d|\mu|(x)\leq \frac{1}{c}\int_F \phi_K\left( \frac{d\mu}{d|\mu|}(x) \right) d|\mu|(x)=\frac{1}{c}|\mu|_K(F).
\end{align*}
Thus, $|\mu|_{K}<\!\!<|\mu|$ and $|\mu|<\!\!<|\mu|_K$.
\end{remark}
\begin{remark}\label{rem: GBV anisotropic cantor part}
Given $f\in GBV(\R^{n-1})$, motivated by (\ref{def: Cantor part GBV}), for every Borel set $G\subset \R^{n-1}$ we define 
\begin{align}\label{def: Cantor part GBV anisotropica}
|(D^c f,0)|_K(G)=\lim_{M\rightarrow +\infty}|(D^c(\tau_M(f),0)|_K(G)=\sup_{M>0}|(D^c(\tau_M)(f),0)|(G).
\end{align}
\end{remark}

\noindent
The following Lemma is the anisotropic version of \cite[Definition 1.4 (b)]{AFP}.
\begin{lemma}\label{lem: AFP page 3}
Fix an integer $m\geq 1$ and let $K\subset \R^n$ be as in (\ref{HP per K}). Given a $\R^n$-valued Radon measure $\mu$ on $\R^m$ we have 
\begin{align}\label{eq: step 1 anisotropic Page 3 AFP}
|\mu|_{K}(G)= \sup\left\{\sum_{h\in\mathbb{N}}\phi_{K}(\mu(G_h)):\,(G_h)_{h\in\mathbb{N}}\textit{ pairwise disjoint,}\,\bigcup_{h}G_h =G  \right\},\quad \forall\,G\subset\R^m \textit{ Borel,}
\end{align}
where $G_h$ are bounded Borel sets.
\end{lemma}
\begin{proof}
Thanks to Jensen Inequality and 1-homogeneity of $\phi_{K}$ we get
\begin{align*}
\phi_{K}\left(\mu(G_h)\right)&=\phi_{K}\left(\int_{G_h}\frac{d\mu}{d|\mu|}(x)d|\mu|(x) \right)
\leq |\mu|_{K}(G_h),
\end{align*}
so using that $G_h\cap G_k=\emptyset$ $\forall h\neq k$
\begin{align*}
|\mu|_{K}(G)=|\mu|_{K}\left(\cup_{h}G_h\right)= \sum_{h\in\mathbb{N}}|\mu|_{K}(G_h)\geq \sum_{h\in\mathbb{N}}\phi_{K}(\mu(G_h)).
\end{align*}
Taking the $\sup$ on the right hand side we proved that $|\mu|_{K}(G)$ is greater or equal than the right hand side of relation (\ref{eq: step 1 anisotropic Page 3 AFP}). We are then left to prove that 
\begin{align*}
|\mu|_{K}(G)\leq \sup\left\{\sum_{h\in\mathbb{N}}\phi_{K}(\mu(G_h)):\,(G_h)_{h\in\mathbb{N}}\textit{ pairwise disjoint,}\,\bigcup_{h}G_h =G  \right\},
\end{align*}
Let $G\subset \R^n$ be a bounded Borel set. Let us consider the function 
$$  
f(x)= \frac{d\mu}{d|\mu|}(x)\in L^{\infty}(\R^{m},|\mu|;\R^n).
$$
For each $i\in\{1,\dots,n\}$ we also have
\begin{align*}
f_i(x)=\frac{d\mu_i}{d|\mu|}(x) \in L^1_{loc}(\R^{m},|\mu|),
\end{align*}
where $\mu=(\mu_1,\dots,\mu_n)$. Consider $\forall\,i \in $ a sequence of step functions $\{f_{i,h}  \}_{h\in\mathbb{N}}$ such that 
$$
\| f_{i,h}-f_i \|_{L^{\infty}(\R^{m},|\mu|)}\rightarrow 0 \quad \textit{ as }h\rightarrow \infty.
$$
As a consequence, if we set $f_h=(f_{1,h},\dots,f_{n,h})$ we have that $
\| f_{h}-f \|_{L^{\infty}(\R^{m},|\mu|;\R^n)}\rightarrow 0 $ as $h\rightarrow \infty$. Fix $\epsilon>0$, then there exists $h(\epsilon)>0$ such that  
$$ 
\| f_{h}-f \|_{L^{\infty}(\R^{m},|\mu|;\R^n)}< \epsilon\qquad\forall\, h>h(\epsilon). 
$$
Since for each $i\in\{1,\dots,n\}$ the function $f_{h,i}$ is simple, there exists $n(h)\in\mathbb{N}$ and a finite pairwise disjoint partition $\{ G^h_k\}_{k=1,\dots,n(h)}$ of $G$ such that $f_{h}$ is constant $|\mu|$-a.e. in $G^h_k$, $\forall\,k\in\{1,\dots,n(h) \}$, namely $\exists\,a_{h,k}\in\R^n$ s.t. $f_h(x)=a_{h,k}$ for $|\mu|$-a.e. $x\in G_k^h$, $\forall\,k\in\{1,\dots,n(h) \}$. Then thanks to the one-homogeneity and subadditivity of $\phi_K$ we get
\begin{align*}
\int_{G}\phi_K\left( f_h(x)  \right)\,d|\mu|(x)&=\sum_{k=1}^{n(h)}\int_{G_k^h}\phi_K\left( f_h(x)  \right)\,d|\mu|(x)=\sum_{k=1}^{n(h)}\phi_K\left( a_{h,k}  \right)|\mu|(G_k^h)\\&=\sum_{k=1}^{n(h)}\phi_K\left(a_{h,k} |\mu|(G_k^h) \right)=\sum_{k=1}^{n(h)}\phi_K\left(  \int_{G_k^h}f_h(x)\,d|\mu|(x) \right)\\
&=\sum_{k=1}^{n(h)}\phi_K\left( \int_{G_k^h}f(x)\,d|\mu|(x) +\int_{G_k^h}(f_h(x)-f(x))\,d|\mu|(x) \right)\\
&\leq \sum_{k=1}^{n(h)}\phi_K\left( \int_{G_k^h}f(x)\,d|\mu|(x)\right) + \sum_{k=1}^{n(h)}\phi_K\left( \int_{G_k^h}(f_h(x)-f(x))\,d|\mu|(x) \right)\\
&=\sum_{k=1}^{n(h)}\phi_K\left( \mu(G_k^h)\right) + \sum_{k=1}^{n(h)}\left| \int_{G_k^h}(f_h-f)\,d|\mu| \right|\phi_K\left(\frac{\int_{G_k^h}(f_h-f)\,d|\mu| }{\left| \int_{G_k^h}(f_h-f)\,d|\mu| \right|} \right)\\
&\leq \sum_{k=1}^{n(h)}\phi_K\left( \mu(G_k^h)\right)+ C\sum_{k=1}^{n(h)} \int_{G_k^h}\left|f_h(x)-f(x)\right|\,d|\mu|(x) \\
&\leq \sum_{k=1}^{n(h)}\phi_K\left( \mu(G_k^h)\right)+ \epsilon\, C\sum_{k=1}^{n(h)} |\mu|(G_k^h) \\
&= \sum_{k=1}^{n(h)}\phi_K\left( \mu(G_k^h)\right)+ \epsilon\, C |\mu|(G) \qquad \forall\,h>h(\epsilon),
\end{align*}
where $C:=\sup_{\omega\in\mathbb{S}^{n-1}}\phi_K(\omega)$. So we proved that $\forall\, \epsilon>0$ $\exists\, h(\epsilon)>0$ s.t. $\forall\,h>h(\epsilon)$ there are $n(h)\in\mathbb{N}$ and $\{G_k^h\}_{k=1,\dots,n(h)}$ such that the following holds
\begin{align*}
\int_{G}\phi_K\left( f_h(x)  \right)\,d|\mu|(x)&\leq \sum_{k=1}^{n(h)}\phi_K\left( \mu(G_k^h)\right)+ \epsilon\, C |\mu|(G)\\ &\leq \sup\left\{\sum_{h\in\mathbb{N}}\phi_K(\mu(G_h)):\,(G_h)_{h\in\mathbb{N}}\textit{ pairwise disjoint,}\,\bigcup_{h}G_h =G  \right\}\\&+  \epsilon\, C |\mu|(G).
\end{align*}
Taking the limit as $h\rightarrow +\infty$ in the left hand side, by Lebesgue dominated theorem we get
\begin{align*}
|\mu|_{K}(G)&\leq \sup\left\{\sum_{h\in\mathbb{N}}\phi_K(\mu(G_h)):\,(G_h)_{h\in\mathbb{N}}\textit{ pairwise disjoint,}\,\bigcup_{h}G_h =G  \right\}\\&+  \epsilon\, C |\mu|(G). 
\end{align*}  
By the arbitrariness of $\epsilon > 0$ we conclude for $G$ bounded. Thanks to standard considerations we can extend the result also for $G$ unbounded.
\end{proof}

\begin{definition}[Hausdorff distance]
Let $A,B \subset \R^n$. We define the \emph{Hausdorff distance} between A and B as
\begin{align*}
\dist_H(A,B):= \max\left\{\sup_{x\in A}d(x,B);\sup_{x\in B}d(x,A)  \right\},
\end{align*}
where $d(\cdot,A)$ denotes the Euclidean distance from $A$.
\end{definition}

\begin{figure}
\centering
\def\svgwidth{7cm}
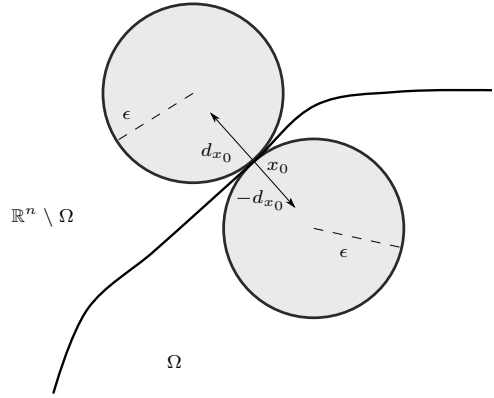
\caption{A pictorial idea of the $\epsilon$- ball property.}
\label{fig:epsilonproperty}
\end{figure}

\begin{definition}[$\epsilon$-ball property]
Let $\epsilon >0$. We say that an open bounded set $\Omega \subset \R^n$ satisfies the $\epsilon$-ball property if for any point $x\in \partial \Omega$ $\exists$ a unit vector $d_x\in \mathbb{S}^{n-1}$ s.t. 
\begin{align*}
& B(x-\epsilon d_x,\epsilon) \subset \Omega,\\
& B(x+\epsilon d_x,\epsilon) \subset \R^n\setminus \overline{\Omega}.
\end{align*}
\end{definition}
\noindent
Roughly speaking, a set satisfies the $\epsilon$-ball property if it is possible to roll two tangent balls, one in the interior and the other one in the exterior part of $\Omega$ (see for instance figure \ref{fig:epsilonproperty}).

\begin{definition}
Let $\mathcal{S}\subset \R^n$ be non-empty. We say that $\mathcal{S}$ is a $C^{1,1}$ hypersurface if for every point $x\in \mathcal{S}$, there exists an open neighbourhood $D$ of $x$, an open set $E$ of $\R^{n-1}$, and a continuously differentiable bijection $\varphi:E\to D\cap \mathcal{S}$ with $\varphi$ and its gradient $\nabla \varphi$ both Lipschitz continuous, and $\mathcal{J}\varphi>0$ on $E$, where $\mathcal{J}\varphi$ stands for the Jacobian of $\varphi$.
\end{definition}
%\begin{definition}
%Let $\mathcal{S}\subset \R^n$ be non-empty. We say that $\mathcal{S}$ is a $C^{1,1}$-hypersurface if there exists an open subset $\Omega$ of $\R^n$ with the following features: $\partial \Omega =S$, moreover, for any point $x\in\partial \Omega$, there exists a local orthogonal coordinate  centred in $x$ such that, in this local system of coordinates, there exists a map $\varphi : D_{x,r}\rightarrow (-a,a)$ continuously differentiable, $a>0$, $D_{x,r}$ the $(n-1)$-dimensional ball centred in $x$ in the hyperplane such that $\varphi$ and its gradient $\nabla \varphi$ are both Lipschitz continuous satisfying $\varphi(x)=0$, $\nabla\varphi(x)=x$, and also
%\begin{align*}
%\partial \Omega\cap (D_{x,r}\times(-a,a)) &= \left\{(z,\varphi(z)):\,z\in D_{x,r} \right\}\\
%\Omega\cap (D_{x,r}\times(-a,a)) &= \left\{(z,t):\,z\in D_{x,r}\textit{ and }-a<t<\varphi(z) \right\}.
%\end{align*} 
%\end{definition}
\medskip
Given $K\subset \R^n$ as in (\ref{HP per K}), we will now prove few more properties about the surface tension $\phi_K$. In particular,  the main result we present is Proposition \ref{prop:linearityphi} that gives a characterization of the cases of additivity for the function $\phi_K$. 
\begin{lemma}\label{prop:L sse B}
Let $K\subset \R^n$ be as in (\ref{HP per K}), and let $y_1,y_2 \in \R^n$. Then, the following are equivalent:
\begin{itemize}
\item[(i)]$\phi_K(y_1)+\phi_K(y_2)=\phi_K(y_1+y_2)$; 
\vspace{.1cm}

\item[(ii)] $\exists \, \bar{z}\in \partial K $ s.t. $\phi_K(y_1)=y_1\cdot \bar{z}$ and $\phi_K(y_2)=y_2\cdot \bar{z}$.
\end{itemize}
\end{lemma}

\begin{proof}
Assume \textit{(ii)} is satisfied. Then, 
\begin{align*}
\phi_K(y_1+y_2) &= \max_{z\in \partial K} [(y_1+y_2) \cdot z] 
\geq \bar{z} \cdot (y_1+y_2)= \phi_K(y_1)+\phi_K(y_2), 
\end{align*}
which gives \textit{(i)}.
Let now \textit{(i)} be satisfied and suppose, by contradiction, that 
\begin{equation} \label{contradiction}
\nexists \, \overline{z} \quad \text{ such that } \quad \, \phi_K(y_1)=y_1\cdot \bar{z} \quad \text{ and } \quad \phi_K(y_2)=y_2\cdot \bar{z}.
\end{equation}
Let $z_1, z_2, z_3 \in \partial K$ be such that $\phi_K(y_1)=y_1 \cdot z_1$ and $\phi_K(y_2)=y_2\cdot z_2$, and
$$ 
\phi_K(y_1+y_2)= (y_1+y_2)\cdot z_3.
$$ 
Then, 
\begin{align*}
y_1\cdot z_3 \leq y_1\cdot z_1 \qquad \text{ and } \qquad y_2 \cdot z_3 \leq y_2 \cdot z_2.
\end{align*}
Note that, in particular, from \eqref{contradiction} we have 
that at least one of the above inequalities is strict. 
Thus, 
\begin{align*}
\phi_K(y_1+y_2)< \phi_K(y_1)+\phi_K(y_2), 
\end{align*}
which is a contradiction to $(i)$.
\end{proof}
\begin{lemma}\label{thm: page 16 mie note}
Let $K\subset \R^n$ be as in (\ref{HP per K}) and consider $\phi_K$ the associated surface tension. Let $y_0 \in \R^n$ and let $x_0 \in \partial K$. Then, 
$$
\phi_K(y_0) = y_0 \cdot x_0 \quad \Longleftrightarrow \quad \frac{y_0}{\phi_K(y_0)} \in \partial \phi^*_K(x_0),
$$
where, we recall, $\partial \phi^*_K(x_0)$ is the sub differential of $\phi^*_K$ at $x_0$.
\end{lemma}
\begin{proof}
We divide the proof into two steps, one for each implications.

\vspace{.2cm}

\noindent
\textbf{Step 1} Suppose 
$$
\frac{y_0}{\phi_K(y_0)} \in \partial \phi^*_K(x_0).
$$
Then, since by (\ref{eq: Phi=1, Phi* =1}) we have $\phi^*_K(x_0) = 1$, we deduce that for every $z\in \R^n$
\begin{align*}
 \phi^*_K(z) &\geq \phi^*_K(x_0) + \frac{y_0}{\phi_K(y_0)} \cdot(z-x_0) 
 = 1 + \frac{y_0}{\phi_K(y_0)} \cdot(z-x_0).
\end{align*}
In particular, if $z \in \partial K$ we have $\phi^*_K(z) = 1$, and therefore 
$$
1 \geq 1 + \frac{y_0}{\phi_K(y_0)} \cdot(z-x_0), \qquad \text{ for every } z\in \partial K,
$$
so that $y_0 \cdot x_0 \geq y_0 \cdot z$ for every $z\in \partial K$.
Thus, $\phi_K(y_0) = y_0 \cdot x_0$.

\vspace{.2cm}

\noindent
\textbf{Step 2} Assume that $\phi_K(y_0) = y_0 \cdot x_0$. Then, by the Fenchel inequality, 
for every $z \in \R^n$ we have 
\begin{align*}
\phi_K(y_0) \phi_K^* (z) \geq y_0 \cdot z 
\quad \Longleftrightarrow \quad 
\phi_K^* (z) \geq \frac{y_0 \cdot z}{y_0 \cdot x_0}
\quad \Longleftrightarrow \quad
\phi_K^* (z) \geq 1+ \frac{y_0 \cdot (z-x_0)}{y_0 \cdot x_0}.
\end{align*}
Recalling that $\phi^*_K(x_0) =1$, we conclude.
\end{proof}

\begin{remark}\label{rem: lem page 16 mie note}
Let us observe that, given $y_0\in\R^n$ and $x_0\in\partial K$ then
\begin{align*}
\phi_K(y_0) = y_0 \cdot x_0 \quad \Longleftrightarrow \quad y_0 \in C^*_K(x_0),
\end{align*}
where $C^*_K(x_0)$ has been defined in (\ref{def:weak sub-differential}). Indeed, by the Lemma above and Definition \ref{def:weak sub-differential}, we immediately derive that if $\phi_K(y_0)=y_0\cdot x_0$ then $y_0/\phi_K(y_0)\in\partial \phi^*_K(x_0)$ that implies $y_0\in C^*_K(x_0)$. Whereas, if $y_0\in C^*_K(x_0)$ then there exists $\lambda=\lambda(y_0)>0$ such that $\lambda y_0\in\partial \phi^*_K(x_0)$ i.e.
\begin{align*}
\phi^*_K(z)\geq 1 +\lambda y_0\cdot(z-x_0)\qquad \forall\, z\in \R^n.
\end{align*}
In particular, if we choose $z\in \partial K$ we get 
\begin{align*}
\lambda y_0\cdot x_0\geq \lambda y_0\cdot z\qquad \forall\, z\in \partial K,
\end{align*} 
that implies $\phi_K(y_0)=y_0\cdot x_0$.
\end{remark}
\noindent
As a direct consequence of Lemmas \ref{prop:L sse B} and \ref{thm: page 16 mie note} we get the following proposition.
\begin{proposition}\label{prop:linearityphi}
Let $K\subset \R^n$ be as in (\ref{HP per K}), and let $y_1,y_2 \in \R^n$. Then, the following are equivalent:
\begin{itemize}
\item[(i)]$\phi_K(y_1)+\phi_K(y_2)=\phi_K(y_1+y_2)$; 
\vspace{.1cm}

\item[(ii)] $\exists \, \bar{z}\in \partial K $ s.t. $\phi_K(y_1)=y_1\cdot \bar{z}$ and $\phi_K(y_2)=y_2\cdot \bar{z}$,
\vspace{.1cm}

\item[(iii)]$\exists \bar{z}\in \partial K$ s.t. $\frac{y_1}{\phi_K(y_1)},\frac{y_2}{\phi_K(y_2)} \in \partial \phi^*_K(\bar{z}) $.
\end{itemize}
\end{proposition}

\begin{remark}\label{rem: more appealing way (iii)}
By Definition \ref{def:weak sub-differential} condition $(iii)$ in the above Proposition is equivalent to say that 
\begin{align}\label{eq: (iii) con il cono}
\exists \bar{z}\in \partial K \quad \textit{ s.t.}\quad y_1,y_2\in C^*_K(\bar{z}).
\end{align}
As noticed in Remark \ref{rem: convexity of C_K^*}, $C^*_K(\bar{z})$ is a convex set and so condition (\ref{eq: (iii) con il cono}) is equivalent to say that 
\begin{align}\label{eq: more appealing way (iii)}
\exists \bar{z}\in \partial K \quad \textit{ s.t.}\quad \left\{\lambda y_1+(1-\lambda)y_2:\,\lambda\in [0,1]\right\} \subset C^*_K(\bar{z}).
\end{align}
\end{remark}

\begin{lemma}\label{lem: sub-diff e bordo wulff}
Let $K\subset \R^n$ be as in (\ref{HP per K}) and consider $\phi_K$ the associated surface tension. Let $x_0 \in \partial K$ then, 
\begin{align}\label{eq: step 1 sub-dif bord wulff}
\phi_K(y)=1\qquad \forall\, y\in \partial \phi^*_K(x_0).
\end{align}
Moreover, 
\begin{align}\label{eq: step 2 sub-dif bord wulff}
\bigcup_{x\in\partial K}\partial \phi^*_K(x)=\partial K^*.
\end{align}
\end{lemma}
\begin{proof} We divide the proof in two steps.\\
\textbf{Step 1} In this first part we prove (\ref{eq: step 1 sub-dif bord wulff}). Let $y\in\partial \phi_K^*(x_0)$. By definition of sub-differential, we have that 
\begin{align*}
\phi_K^*(z)\geq 1 + y\cdot(z-x_0)\qquad \forall\, z\in\R^n.
\end{align*}
So, choosing $z=0$ we get that $y\cdot x_0 \geq 1$. Observe that $y\in \partial \phi^*_K(x_0)$ implies $y\in C^*_K(x_0)$ so that $\phi_K(y)=y\cdot x_0$ by Remark \ref{rem: lem page 16 mie note}. So, $\phi_K(y)=y\cdot x_0 \geq 1$. At the same time, by Lemma \ref{thm: page 16 mie note}, the fact that $\phi_K(y)=y\cdot x_0$  is equivalent to say that $y/\phi_K(y)\in \partial \phi^*_K(x_0)$. By the convexity property of the sub-differential of a convex function (see Remark \ref{rem: convexity of C_K^*}), we have $\lambda y \in \partial \phi^*_K(x_0)$ for every $\lambda\in [1/\phi_K(y),1]$, namely
\begin{align*}
\phi_K^*(z)\geq 1+\lambda y\cdot(z-x_0)\qquad \forall\,z\in \R^n,\,\forall\,\lambda\in [1/\phi_K(y),1].
\end{align*}  
Note that choosing $z=0$ we get $\lambda\geq 1/\phi_K(y)$, while choosing $z=2x_0$ we get, thanks to 1-homogeneity of $\phi_K^*$, that $\lambda\leq 1/\phi_K(y)$.  Thus, we deduce that $1/\phi_K(y)=1$. This concludes the proof of the first step.\\
\textbf{Step 2} In the last step we prove (\ref{eq: step 2 sub-dif bord wulff}). Thanks to step 1 and Remark \ref{rem: abouth gauge function} we have that 
$$
\bigcup_{x\in\partial K}\partial \phi^*_K(x)\subseteq \partial K^*.
$$
We are left to prove the other inclusion. Let $y\in \partial K^*$. By properties of convex sets there exists $\nu(y)\in\mathbb{S}^{n-1}$ such that $K^* \subset H^-_{y,\nu(y)}$ (see relations (\ref{Hxnu+})). So, $\forall\, z\in H^-_{y,\nu(y)}$ , and in particular $\forall\, z\in K^*$ we have
\begin{align*}
z\cdot \nu(y)\leq y\cdot \nu(y),
\end{align*} 
that implies, recalling Remark \ref{rem: abouth gauge function} that $\phi_K^*(\nu(y))= \nu(y)\cdot y$. Thus, thanks to Lemma \ref{thm: page 16 mie note}, recalling that $\phi_K(y)=1$ we get
\begin{align*}
\phi_K^*(\nu(y))=\nu(y)\cdot y \quad&\Leftrightarrow\quad \phi_K^*\left(\frac{\nu(y)}{\phi_K^*(\nu(y))} \right) =\frac{\nu(y)}{\phi_K^*(\nu(y))}\cdot y \quad\Leftrightarrow\quad 1=\frac{\nu(y)}{\phi_K^*(\nu(y))}\cdot y\\ &\Leftrightarrow\quad \phi_K(y)=\frac{\nu(y)}{\phi_K^*(\nu(y))}\cdot y\quad\Leftrightarrow\quad y\in \partial \phi_K^*\left(\frac{\nu(y)}{\phi_K^*(\nu(y))} \right).
\end{align*}
Since $\nu(y)/\phi_K^*(\nu(y))\in \partial K$ we conclude.
\end{proof}

\begin{figure}[!htb]
\centering
\def\svgwidth{13cm}
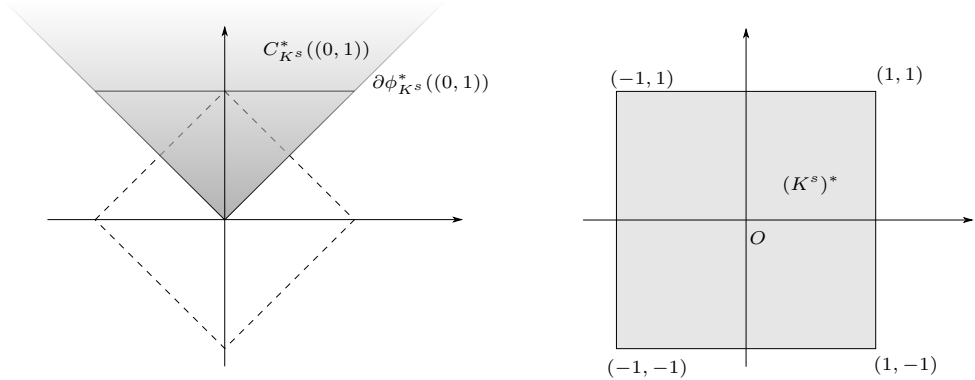
\caption{A pictorial idea of condition (\ref{eq: step 2 sub-dif bord wulff}) with respect to the Wulff shape $K^s$ presented in Figure \ref{fig:K e K^*}. Indeed, according to Lemma \ref{lem: sub-diff e bordo wulff} and (\ref{altra caratterizzazione del sub diff}), we see that $\partial\phi_{K^s}^*((0,1))$ is a convex subset of the boundary of $(K^s)^*$. The fact that $\partial\phi_{K^s}^*((0,1))$ actually contains the point $(0,1)$ is just a consequence of the specific Wulff shape considered in the example.}
\label{fig:K e K^* e C^*_K}
\end{figure}   

\begin{remark}
Having in mind the definition of $C^*_K(x)$ (see (\ref{def:weak sub-differential}), and as a consequence of (\ref{eq: step 2 sub-dif bord wulff}),  we have that 
\begin{align}\label{eq: unione coni = Rn}
\bigcup_{x\in\partial K}C^*_K(x)=\R^n.
\end{align}
\end{remark}

\begin{corollary}\label{cor:Euler equations}
Let $K\subset \R^n$ be as in (\ref{HP per K}) and consider $\phi_K$ the associated surface tension. Assume in addition that $\phi_K \in C^1(\R^n_0)$. Then,
\begin{align}\label{eq:Euler equations}
\phi_K(x)=\nabla\phi_K(x)\cdot x \quad\textit{and}\quad \phi^*_K(\nabla\phi_K(x))=1\qquad \forall\,x\in \R^n_0.
\end{align}
\end{corollary}
\begin{proof}
Firstly, let us observe it is a well known fact that the first relation in (\ref{eq:Euler equations}) holds true for every positive and 1-homogeneous function. So, we are left to prove the second relation in (\ref{eq:Euler equations}). Let $x\in \partial K^*$. As we observed in the above Lemma, by properties of convex sets there exists $\nu(x)\in\mathbb{S}^{n-1}$ such that $K^* \subset H^-_{x,\nu(x)}$ and $\phi_K^*(\nu(x))= \nu(x)\cdot x$. By Lemma \ref{thm: page 16 mie note}, having in mind Remark \ref{rem: differentiability and sub differential} we have that
\begin{align}\label{eq: Euler preliminaries}
\phi_K^*(\nu(x))=\nu(x)\cdot x\quad
\Longleftrightarrow\quad \frac{\nu(x)}{\phi^*_K(\nu(x))}=\nabla\phi_K(x).
\end{align} 
By the 1-homogeneity of $\phi_K$ it follows that
\begin{equation}\label{eq: gradiente 0-homogeneo}
\nabla \phi_K(\lambda x)= \nabla\phi_K(x)\qquad \forall\, \lambda>0,\,\forall\,x\in \R^n_0,
\end{equation} 
therefore $\phi^*_K(\nabla\phi_K(x))=1$ for all $x\in\R^n_0$. This concludes the proof.
\end{proof}

\begin{remark}
Let $K\subset \R^n$ be as in (\ref{HP per K}), and consider $x\in\partial K$. Note that, thanks the above results we can deduce the following equivalent characterization for the subdifferential $\partial\phi_K^*(x)$, namely
\begin{align}\label{altra caratterizzazione del sub diff}
\partial \phi^*_K(x)=\left\{y\in\partial K^*:\,y\cdot\frac{x}{|x|}=\phi_K^*\left(\frac{x}{|x|}\right) \right\}.
\end{align} 
Indeed,  thanks to Lemma \ref{lem: sub-diff e bordo wulff} we know that $\partial\phi_K^*(x)\subset \partial K^*$ so that $\phi_K(y)=1$ for all $y\in \partial \phi^*_K(x)$. Whereas, thanks to Lemma \ref{thm: page 16 mie note} we have that $y\in \partial\phi_K^*(x)$ if and only if $1=\phi^*_K(x)\phi_K(y)=y\cdot x$, from which, we get $y\cdot\frac{x}{|x|}=\phi_K^*\left(\frac{x}{|x|}\right)$.
\end{remark}

\noindent
The following two results will be used for the proof of Proposition \ref{prop: R1,R2}.
\begin{lemma}\label{lem: propedeutico lemma R1,R2}
Let $K\subset\R^n$ be as in (\ref{HP per K}). Let $x_1,x_2\in\partial K$ and $\bar{y}\in\partial K^*$ be such that $\bar{y}\in \partial \phi^*_K(x_1)\cap \partial \phi^*_K(x_2)$. Let us now assume that there exist $y_1,y_2\in\partial \phi^*_K(x_2)$, with $y_1\neq \bar{y} \neq y_2$, such that $\bar{y}= (1-\bar{\lambda})y_1+\bar{\lambda} y_2$ for some $\bar{\lambda}\in(0,1)$. Then,
\begin{align}
(1-\lambda)y_1+ \lambda y_2\in\partial \phi^*_K(x_1)\quad \forall\,\lambda\in[0,1].
\end{align} 
\end{lemma}
\begin{proof}
Let us suppose by contradiction that there exists $\tilde{\lambda}\in [0,\bar{\lambda}]$ such that $\tilde{y}=(1-\tilde{\lambda})y_1+ \tilde{\lambda} y_2\notin\partial \phi^*_K(x_1)$. By the Fenchel inequality (\ref{eq: Fenchel inequality}), (\ref{altra caratterizzazione del sub diff}), and using that $\phi_K(\tilde{y})\leq (1-\tilde{\lambda})\phi_K(y_1)+ \tilde{\lambda}\phi_K(y_2)\leq 1$ we get 
\begin{align}\label{eq: lemma propedeutico 1}
\tilde{y}\cdot\frac{x_1}{|x_1|}<\bar{y}\cdot\frac{x_1}{|x_1|}=\phi_K^*\left(\frac{x_1}{|x_1|}  \right).
\end{align}
Recall that, by (\ref{def: Wulff shape K}) applied to $K^*$ we have that 
$$
\overline{K^*}=\bigcap_{\omega \in \mathbb{S}^{n-1}}\left\{x\in \R^n:\, x\cdot\omega \leq \phi^*_K(\omega)   \right\}.
$$
By relation (\ref{eq: lemma propedeutico 1}) we have that the continuous linear function 
$$
\varphi(\lambda):= ((1-\lambda)y_1+ \lambda y_2)\cdot\frac{x_1}{|x_1|}>\phi_K^*\left(\frac{x_1}{|x_1|}  \right)
$$ 
for every $\lambda\in (\bar{\lambda},1]$, but this is a contradiction since
$$\{(1-\lambda)y_1+ \lambda y_2:\,\lambda\in [0,1] \}\subset\partial \phi^*_K(x_2)\subset \overline{K^*}.
$$ 
The case when $\tilde{\lambda}\in [\bar{\lambda},1]$ is symmetric, and thus the proof is complete.
\end{proof}

\begin{corollary}\label{cor: propedeutico Lemma R1,R2}
Let $K\subset\R^n$ be as in (\ref{HP per K}). Let $x\in\partial K$ be such that the subdifferential of $\phi_K^*$ in $x$ has only one point, namely $\partial \phi^*_K(x)=\{y\}$. Then, $\forall\, z\in\mathcal{Z}_K(y)$, where $\mathcal{Z}_K(y)$ is defined in (\ref{def: Z_K(y)}), and for every $y_1,y_2\in C_K^*(z)$, with $y_1/\phi_K(y_1)\neq y_2/\phi_K(y_2)$, if $\exists\,\lambda\in(0,1)$ s.t. $y=(1-\lambda)y_1+\lambda y_2$, then $y_1=\lambda_1 y$, $y_2=\lambda_2 y$ for some $\lambda_1,\lambda_2>0$.
\end{corollary}
\begin{proof}
Let us fix $z\in \mathcal{Z}_K(y)$ and $y_1,y_2\in C_K^*(z)$ and let us assume that $y=(1-\bar{\lambda})y_1+\bar{\lambda} y_2$, for some $\bar{\lambda}\in(0,1)$. Since $y\in \mathcal{Z}_K(y)$, then $y_1,y,y_2\in C^*_K(z)$, and thus
\begin{align}\label{eq: y1, y,y2 nel sub}
\frac{y_1}{\phi_K(y_1)},\frac{y}{\phi_K(y)}, \frac{y_2}{\phi_K(y_2)} \in \partial \phi^*_K(z).
\end{align}
Let us observe that $\phi_K(y)=1$ since we know $y\in \partial\phi^*_K(x)$. As a consequence of (\ref{eq: y1, y,y2 nel sub}), together with the convexity of $\partial\phi^*_K(z)$, we deduce that 
\begin{align}
y&\in \partial\phi^*_K(z)\cap \partial\phi^*_K(x),\\
y&= (1-t)\frac{y_1}{\phi_K(y_1)}+t\frac{y_2}{\phi_K(y_2)},\label{eq: famigerata}
\end{align}
where $t\in (0,1)$. Therefore, thanks to Lemma \ref{lem: propedeutico lemma R1,R2}, we have that 

$$
(1-t)\frac{y_1}{\phi_K(y_1)}+t\frac{y_2}{\phi_K(y_2)}\in \partial\phi^*_K(x)\quad \forall\,t\in [0,1],
$$
but this is possible if and only if $y_i/\phi_K(y_i)=y$ for $i=1,2$. This concludes the proof.
\end{proof}

We conclude this section recalling few more definitions and a couple of results very well known in convex analysis. Such tools, will play a key role in the understanding of (\ref{rigidity anisotropic steiner}).

\begin{definition}\label{def: extreme point}
Let $C\subset\R^n$ be a convex set. We say that $x\in C$ is an \emph{extreme point} of $C$ if and only if there is no way to express  $x$ as a convex combination $(1-\lambda)y + \lambda z$ such that $y,z \in C$ and $0<\lambda<1$, except by taking $y=z=x$. 
\end{definition}

\begin{definition}\label{def: exposed point}
Let $C\subset\R^n$ be a convex set. We say that $x\in C$ is an \emph{exposed point} of $C$ if and only if there exists an hyperplane  of the form $H_{x,\nu}$, with $\nu\in\mathbb{S}^{n-1}$, such that $C\subset H^-_{x,\nu}$ and $\overline{C}\cap H_{x,\nu}=\{x\}$. Observe that if $x$ is an exposed point of $C$, then $x$ belongs to the boundary of $C$.
\end{definition}

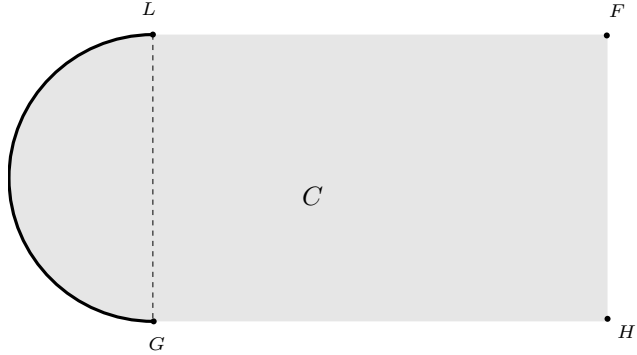
\begin{figure}[!htb]
\centering
\def\svgwidth{10cm}
%% Creator: Inkscape inkscape 0.92.3, www.inkscape.org
%% PDF/EPS/PS + LaTeX output extension by Johan Engelen, 2010
%% Accompanies image file 'espostoestremo.pdf' (pdf, eps, ps)
%%
%% To include the image in your LaTeX document, write
%%   \input{<filename>.pdf_tex}
%%  instead of
%%   \includegraphics{<filename>.pdf}
%% To scale the image, write
%%   \def\svgwidth{<desired width>}
%%   \input{<filename>.pdf_tex}
%%  instead of
%%   \includegraphics[width=<desired width>]{<filename>.pdf}
%%
%% Images with a different path to the parent latex file can
%% be accessed with the `import' package (which may need to be
%% installed) using
%%   \usepackage{import}
%% in the preamble, and then including the image with
%%   \import{<path to file>}{<filename>.pdf_tex}
%% Alternatively, one can specify
%%   \graphicspath{{<path to file>/}}
%% 
%% For more information, please see info/svg-inkscape on CTAN:
%%   http://tug.ctan.org/tex-archive/info/svg-inkscape
%%
\begingroup%
  \makeatletter%
  \providecommand\color[2][]{%
    \errmessage{(Inkscape) Color is used for the text in Inkscape, but the package 'color.sty' is not loaded}%
    \renewcommand\color[2][]{}%
  }%
  \providecommand\transparent[1]{%
    \errmessage{(Inkscape) Transparency is used (non-zero) for the text in Inkscape, but the package 'transparent.sty' is not loaded}%
    \renewcommand\transparent[1]{}%
  }%
  \providecommand\rotatebox[2]{#2}%
  \newcommand*\fsize{\dimexpr\f@size pt\relax}%
  \newcommand*\lineheight[1]{\fontsize{\fsize}{#1\fsize}\selectfont}%
  \ifx\svgwidth\undefined%
    \setlength{\unitlength}{820.77788427bp}%
    \ifx\svgscale\undefined%
      \relax%
    \else%
      \setlength{\unitlength}{\unitlength * \real{\svgscale}}%
    \fi%
  \else%
    \setlength{\unitlength}{\svgwidth}%
  \fi%
  \global\let\svgwidth\undefined%
  \global\let\svgscale\undefined%
  \makeatother%
  \begin{picture}(1,0.47843407)%
    \lineheight{1}%
    \setlength\tabcolsep{0pt}%
    \put(0,0){\includegraphics[width=\unitlength,page=1]{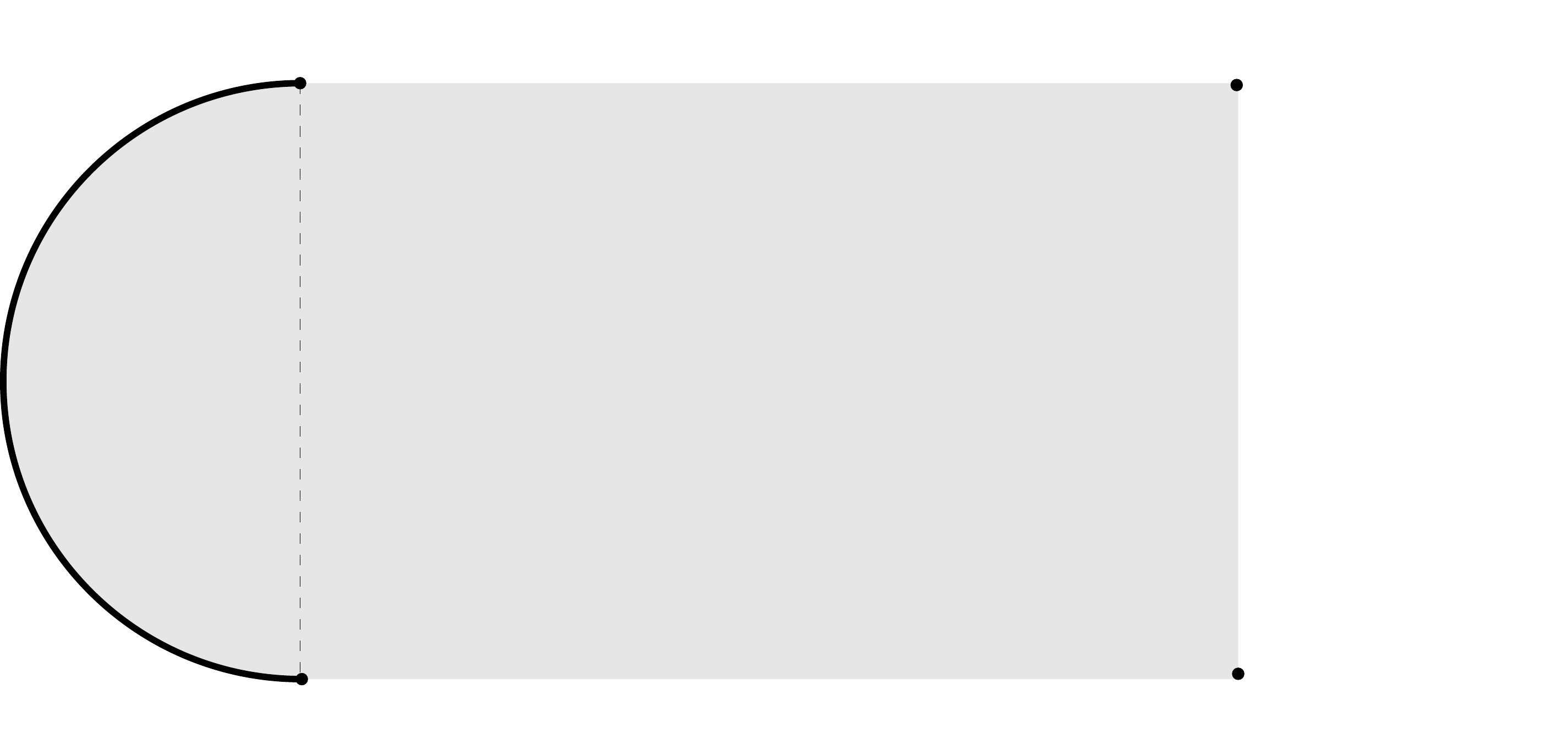}}%
    \put(0.1765202,0.45066411){\color[rgb]{0,0,0}\makebox(0,0)[lt]{\lineheight{1.25}\smash{\begin{tabular}[t]{l}\tiny{$L$}\end{tabular}}}}%
    \put(0.79126774,0.448818){\color[rgb]{0,0,0}\makebox(0,0)[lt]{\lineheight{1.25}\smash{\begin{tabular}[t]{l}\tiny{$F$}\end{tabular}}}}%
    \put(0.80234428,0.02237152){\color[rgb]{0,0,0}\makebox(0,0)[lt]{\lineheight{1.25}\smash{\begin{tabular}[t]{l}\tiny{$H$}\end{tabular}}}}%
    \put(0.18575065,0.00760283){\color[rgb]{0,0,0}\makebox(0,0)[lt]{\lineheight{1.25}\smash{\begin{tabular}[t]{l}\tiny{$G$}\end{tabular}}}}%
    \put(0.38799207,0.20003075){\color[rgb]{0,0,0}\makebox(0,0)[lt]{\lineheight{1.25}\smash{\begin{tabular}[t]{l}\small{$C$}\end{tabular}}}}%
  \end{picture}%
\endgroup%

\caption{Given a closed convex set $C$ as in the figure above, its set of \emph{extreme} points is the one that contains the parts of the boundary of $C$ that are in bold (the four points $L,F,H,G$ are included). Whereas, the set of \emph{exposed} points of $C$ is the set of extreme points of $C$ without the two points $L$ and $G$.}
\label{fig: exposed extreme}
\end{figure}

\begin{remark}\label{exposed dense in extreme}
If $C\subset\R^n$ is a closed convex set, then by \cite[Theorem 18.6]{Rock}, the set of exposed points of $C$ is dense in the set of extreme points of $C$, namely, every extreme point is the limit of a sequence of exposed points (see for instance Figure \ref{fig: exposed extreme}). 
\end{remark}
\noindent
Let us now recall a useful result about the characterization of the exposed points of a closed convex set (see for instance \cite[Corollary 25.1.3]{Rock}).

\begin{lemma}\label{lem: corollary 25.1.3 Rockafellar}
Let $C\subset \R^n$ be a non empty, closed, convex set, and let $g:\R^n\rightarrow [0,\infty)$ be any 1-homogeneous, convex function, such that 
$$
C=\{z\in\R^n:\, z\cdot y\leq g(y) \quad \forall\,y\in\R^n    \}.
$$
Then, $z\in C$ is an exposed point of $C$ if and only if there exists a point $y\in\R^n$ such that $g$ is differentiable at $y$ and $\nabla g(x)=z$.
\end{lemma}

\section{Characterization of the anisotropic total variation}\label{section 4}\noindent
In this section we will study some properties of the anisotropic total variation (see Definition \ref{def: anisotropic total variation}), proving also a characterization result (see Theorem \ref{thm:Charac.anis.total.var.}). Such characterization result is already known in the literature but we decided to give a proof for the sake of completeness since we couldn't find a precise reference. The main result of this Section \ref{section 4} is the following.

\begin{theorem}\label{thm:Charac.anis.total.var.}
Let $K\subset \R^n$ be as in (\ref{HP per K}).
%Let $\phi_K:\R^n \rightarrow [0;+\infty)$ be a one-homogeneous, convex and coercive function.
Let $\mu$ be a $\R^n$-valued Radon measure on $\R^m$, $m\geq 1$, $m\in\mathbb{N}$. Then, we have
\begin{align*}
|\mu|_K(\Omega)=\sup \left\{\int_{\Omega}\varphi(x)\cdot d\mu(x):\, \varphi\in C^1_c(\Omega;\R^n),\,\phi^*_K(\varphi)\leq 1 \right\}\quad \forall\, \Omega \subset \R^m\textit{ open.} 
\end{align*} 
\end{theorem}
\noindent
In order to prove Theorem \ref{thm:Charac.anis.total.var.} we need some intermediate results. 
\begin{lemma}\label{lem: 1.2 of my notes}
Let $\left\{ K_h \right\}_{h\in \mathbb{N}}\subset \R^n$, $K\subset \R^n$ be such that $K_h, K$ are as in (\ref{HP per K}) $\forall\, h\in\mathbb{N}$. Assume moreover that
\begin{itemize}
\item[i)] the sequence $(K_h)_{h\in\mathbb{N}}$ is either of the form $K_h \subset K_{h+1}\subset K$, or $K \subset K_{h+1} \subset K_{h}$, $\forall h\in \mathbb{N}$, 
\item[ii)]$ \lim_{h\rightarrow +\infty} \dist_H(K_h,K)=0$.
\end{itemize}
Then, the sequence  $\{\phi_{K_h} \}$ converges uniformly to $\phi_{K}$ in $\mathbb{S}^{n-1}$.
% that is $\phi_{K_h}\Longrightarrow \phi_K$.
\end{lemma}
\begin{proof}Without loss of generality we can consider the case when $K_h \subset K_{h+1}\subset K$ $\forall h\in\mathbb{N}$.\\
For every $x\in \mathbb{S}^{n-1}$ and $h\in\mathbb{N}$, let $y(x)\in \partial K$ and $y_h(x)\in \partial K_h$ be such that $\phi_K(x)=y(x)\cdot x$ and $\phi_{K_h}(x)=y_h(x)\cdot x$, respectively. Then, since $K_h\subset K$,
$$ \sup_{x\in \mathbb{S}^{n-1}}|\phi_K(x)-\phi_{K_h}(x)|=\sup_{x\in \mathbb{S}^{n-1}}\left[x\cdot (y(x)-y_h(x))  \right]. $$
Note now that, by definition of $y_h$, we have  $-x\cdot y_h(x)\leq -x\cdot \bar{y}\;$ $\forall\, \bar{y}\in \partial K_h$. In particular, choosing $\bar{y}=z(x)\in\partial K_h$ such that $|y(x)-z(x)|= \dist(y(x),\partial K_h)$, we have 
\begin{align*}
\sup_{x\in\mathbb{S}^{n-1}}|\phi_K(x)-\phi_{K_h}(x)|\leq \sup_{x\in \mathbb{S}^{n-1}}\left[x\cdot (y(x)-z(x))  \right]\leq  \dist(y(x),\partial K_h)\leq\dist_H(K,K_h),
\end{align*}
where in the last inequality we used the fact that $K_h\subset K$. Passing to the limit as $h\rightarrow +\infty$ we conclude.  
\end{proof}

\begin{lemma}\label{lem:1.3 of my notes}
Let $K\subset \R^n$ be as in (\ref{HP per K}). Then there exists a sequence $\{K_h  \}_{h\in \mathbb{N}} \subset \R^n$ with $K_h$ as in (\ref{HP per K}) for every $h\in\mathbb{N}$,  such that 
\begin{itemize}
\item[i)]$K_h$ is $C^{1,1}$, $\forall\,h\in \mathbb{N}$;
\item[ii)] $K \subset \dots \subset K_{h+1}\subset K_{h}\quad \forall h \in \mathbb{N}$;
\item[iii)]$\lim_{h\rightarrow +\infty}\dist_H(K_h,K)=0$.
\end{itemize}
\end{lemma}
\begin{proof}
We divide the proof in few steps. Take any $\epsilon >0$ and let $K_\epsilon=\bigcup_{x\in K}B(x,\epsilon)$ denote the $\epsilon$-neighbourhood of $K$.\\
\textbf{Step 1} In this Step we want to prove that $K_\epsilon$ is convex, open, bounded and it contains the origin. By construction, we need just to prove that it is convex. Consider two generic points $x_1,x_2 \in K_{\epsilon}$, let us show that 
\begin{align*}
\lambda x_1 + (1-\lambda)x_2 \in K_{\epsilon}\quad \forall \lambda \in [0,1].
\end{align*}
Observe that, since $x_1, x_2\in K_\epsilon $ there exist $c_1, c_2\in K$ such that $|x_1-c_1|< \epsilon $ and  $|x_2-c_2|< \epsilon $. Thus,
\begin{align*}
\lambda x_1
+ (1-\lambda)x_2 &= \lambda[c_1+(x_1-c_1)]+(1-\lambda)[c_2 + (x_2-c_2)]\\ &= \lambda c_1 + (1-\lambda)c_2 + \lambda(x_1-c_1) + (1-\lambda)(x_2-c_2).
\end{align*}
Since $\lambda c_1 + (1-\lambda) c_2 \in K$ and $|\lambda (x_1-c_1)+(1-\lambda)(x_2-c_2)|<\epsilon$ we conclude the proof of step 1.\\
\textbf{Step 2} In this step we are going to prove that $K_\epsilon$ satisfies the $\epsilon$-ball property. This is true by construction. Indeed, since $K_\epsilon$ is as in (\ref{HP per K}), we can associate to it the function $\phi_{K_\epsilon}$. So, having in mind (\ref{def: Wulff shape K}) we know that for every $y \in \partial K_\epsilon$ there exists $\nu \in \mathbb{S}^{n-1}$ and an hyperplane $H_{\phi_{K_\epsilon}(\nu)}=\{z\in \R^n:\, z\cdot \nu = \phi_{K_\epsilon}(\nu)  \}$ such that $y \in H_{\phi_{K_\epsilon}(\nu)}$ and $K_\epsilon$ lies on one side of $H_{\phi_{K_\epsilon}(\nu)}$ (this is because $K_\epsilon$ is a convex set). So, we can construct on the exterior of $K_\epsilon$ a ball of arbitrary radius tangent to the hyperplane $H_{\phi_{K_\epsilon}(\nu)}$ in the point $y$. Let us now consider $z\in K_\epsilon$ such that $|z-y|=\epsilon$ in particular, $z\in \partial K$. By construction we have that $B(z,\epsilon)\subset K_\epsilon$ and this concludes the proof of step 2.

\begin{figure}[!htb]
\centering
\def\svgwidth{7cm}
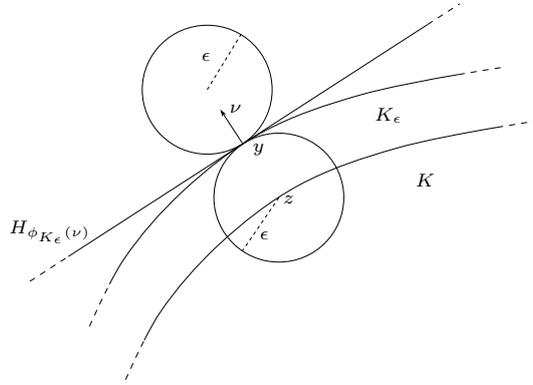
\caption{A pictorial idea for the proof of Lemma \ref{lem:1.3 of my notes}.}
\label{fig:palladentropallafuori}
\end{figure}
\noindent
\textbf{Step 3} We have to prove that $\partial K_{\epsilon} $ is an hypersurface $C^{1,1}$ regular. This result is a straightforward consequence of \cite[Theorem 1.8]{PolarettiD}. \\
\textbf{Step 4} We are left to prove that $\dist_H(\bigcup_{x\in K}B(x,\epsilon),K)\leq \epsilon$. By definition of Hausdorff distance we have that 
\begin{align*}
\dist_H\left(K_{\epsilon},K\right)&= \max \left\{ \sup_{y\in K_{\epsilon}}d(y,K);\sup_{y\in K}d(y,K_{\epsilon})  \right\}\\
&=\max\left\{\epsilon;0  \right\}.
\end{align*}
To conclude the proof of the Lemma let us observe the following. Let us fix a decreasing sequence of positive real numbers $(\epsilon_h)_{h\in\mathbb{N}}$. We can construct the sequence $(K_h)_{h\in\mathbb{N}}$ where $K_h=K_{\epsilon_h}$ is the $\epsilon_h$-neighbourhood of $K$ $\forall\, h\in \mathbb{N}$. By all previous steps, the sequence $(K_h)_{h\in\mathbb{N}}$ satisfies $i), ii)$ and $iii)$ of the Lemma and this concludes the proof. 
\end{proof}
\begin{proposition}\label{prop:1.4b mie note}
Let $K$ be as in (\ref{HP per K}) and let $K^*$ be its dual. Consider $(K_h^*)_{h\in\mathbb{N}}$ a sequence as in (\ref{HP per K}), such that either $K^*_h\subset K^*_{h+1}\subset K^*$ or $K^*\subset K^*_{h+1}\subset K^*_h$, $\forall h\in\mathbb{N}$. Then, denoting with $K_h= (K_h^*)^*$ we have 
\begin{align*}
\lim_{h\rightarrow +\infty} \dist_H(K^*_h,K^*)=0 \quad\textit{ if and only if }\quad \lim_{h\rightarrow +\infty} \dist_H(K_h,K)=0.
\end{align*}
\end{proposition}
\begin{proof} Let us assume that $\lim_{h\rightarrow +\infty} \dist_H(K^*_h,K^*)=0$ and, without loss of generality, that $K^*\subset K^*_{h+1}\subset K^*_h$, $\forall h\in\mathbb{N}$. We can apply immediately Lemma \ref{lem: 1.2 of my notes} to the sequence $\left\{ K_h^* \right\}_{h\in \mathbb{N}}$ to obtain that $\phi_{K_h^*}$ uniformly converges to $\phi_{K^*}$.
% namely
%\begin{align*}
%\phi_{K_h^*} \Rightarrow \phi_{K^*}. 
%\end{align*}
Consider the following quantity
\begin{align*}
\dist_H(K_h,K)=\max\left\{\sup_{x\in K_h}d(x,K); \sup_{x\in K}d(x,K_h)   \right\}.
\end{align*}
Now, by the way the $K_h^*$ are constructed, and having in mind $iii)$ of Proposition \ref{prop: Maggi 20.10}, we have 
\begin{align*}
K_h\subset K_{h+1}\subset \dots \subset K \quad \forall h \in \mathbb{N}.
\end{align*} 
This fact immediately tells us that 
$$ \sup_{x \in K_h}d(x,K)=0. $$
Let us focus our attention now on $\sup_{x\in K}d(x,K_h)$, thus
\begin{align*}
\sup_{x\in K}d(x,K_h)&= \sup_{x\in \partial K}d(x,K_h)= \max_{x\in \partial K}d(x,K_h)\leq \max_{x\in \partial K}|x-x_{K_h}|,
\end{align*}
where $x_{K_h}= \left\{tx:\, t>0   \right\}\cap \partial K_h$. By observing that  
$\phi_{K_h}^*(x)=\frac{|x|}{|x_{K_h}|}\phi_{K_h}^*(x_{K_h})=\frac{|x|}{|x_{K_h}|}$, and since $|x|-|x_{K_h}|=|x-x_{K_h}|$, we get
\begin{align*}
|x-x_{K_h}|\frac{1}{|x_{K_h}|}= \left(\phi_{K_h}^*(x)- \phi_{K}^*(x)  \right).
\end{align*}

\begin{figure}
\centering
\def\svgwidth{7cm}
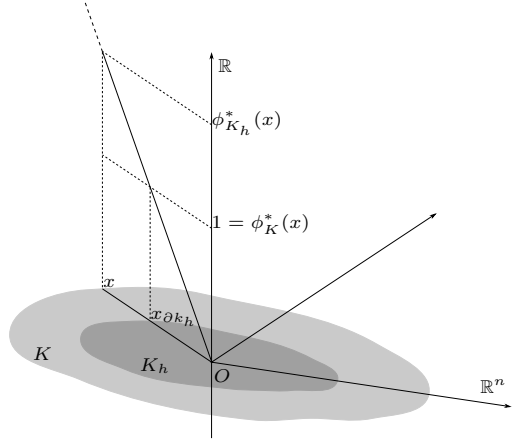
\caption{A pictorial idea for the proof of Proposition \ref{prop:1.4b mie note}.}
\label{fig:pagina50}
\end{figure}

Thus,
\begin{align*}
\lim_{h\rightarrow +\infty}|x-x_{ K_h}|= \lim_{h \rightarrow +\infty} |x_{K_h}|\left(\phi_{K_h}^*(x)- \phi_{K}^*(x)  \right)= 0\quad \forall\,x\in\partial K
\end{align*}
thanks to the uniform convergence of $\phi_{K_h}^*$ to $\phi_{K}^*$. This shows that $\left\{ K_h  \right\}\subset \R^n$ converges in Hausdorff distance to $K$. Since $(K^*)^*=K$, $(K_h^*)^*=K_h$ the proof is complete.
\end{proof}
\noindent
We can now prove Theorem \ref{thm:Charac.anis.total.var.}.
\begin{proof} For the sake of clarity we decided to divide the proof in several steps.\\ \noindent
\textbf{Step 1} Assume $\Omega\subset \R^n$ to be an open, bounded set. We start proving
\begin{align*}
\int_{\Omega}\phi_{K}\left(\frac{d\mu}{d|\mu|}(x)  \right)d|\mu|(x)\geq \sup \left\{\int_{\Omega}\varphi(x)\cdot d\mu(x):\, \varphi\in C^1_c(\Omega;\R^n),\,\phi_{K}^*(\varphi)\leq 1 \right\}.
\end{align*}
Let us observe that by definition of $\phi_K$ we have 

\begin{align*}
|\mu|_K(\Omega)&=\int_{\Omega}\phi_{K}\left(\frac{d\mu}{d|\mu|}(x)  \right) d|\mu|(x)
=\int_{\Omega}\left(\sup_{y\in K}y\cdot \frac{d\mu}{d|\mu|}(x)\right) d|\mu|(x)\\
&
\geq \int_{\Omega}\varphi(x)\cdot \frac{d\mu}{d|\mu|}(x)\,d|\mu|(x),
\end{align*}
where $\varphi \in C^1_c(\Omega;\R^n)$, $\phi_{K}^*(\varphi)\leq 1$. Passing to the sup on the right hand side we conclude the first step.\\
\textbf{Step 2} We want to prove the reverse inequality, namely
\begin{align*}
|\mu|_{K}(\Omega)\leq \sup \left\{\int_{\Omega}\varphi(x)\cdot d\mu(x):\, \varphi\in C^1_c(\Omega;\R^n),\,\phi_{K}^*(\varphi)\leq 1 \right\},
\end{align*}
In order to do so, we consider at first the case when $\phi_K$ is in addition $C^1(\R^n_0)$. Recalling relations (\ref{eq:Euler equations}), we have 
\begin{align*}
|\mu|_K(\Omega)= \int_{\Omega}\phi_K\left(\frac{d\mu}{|d\mu|}(x)\right)d|\mu|(x)=\int_{\Omega}\nabla\phi_K\left(\frac{d\mu}{|d\mu|}(x)  \right)\cdot \frac{d\mu}{|d\mu|}(x)\,d|\mu|(x).
\end{align*}
Since $\nabla\phi_K \in C^0(\R^n_0)$, the composition $\nabla\phi_K\left(\frac{d\mu}{|d\mu|}(x)  \right)$ is well defined, and moreover, 
$$\nabla\phi_K\left(\frac{d\mu}{|d\mu|}(\cdot)  \right)\in L^1_{loc}(\Omega,|\mu|;\R^n),$$
with $\phi_K^*\left( \nabla\phi_K\left(\frac{d\mu}{|d\mu|}(x)  \right) \right)=1$ for $|\mu|$-a.e. $x \in \Omega$. Recall that 
\begin{align*}
\phi_K^*\left( \nabla\phi_K\left(\frac{d\mu}{|d\mu|}(x)  \right) \right)=1\quad \textit{\emph{implies} }\quad \nabla\phi_K\left(\frac{d\mu}{|d\mu|}(x)\right) \in \partial K,\quad \textit{for $|\mu|$-a.e. $x \in \Omega$,}
\end{align*} 
so that $\nabla\phi_K\left(\frac{d\mu}{|d\mu|}(x)\right)\in L^{\infty}(\Omega,|\mu|;\R^n)$. By the fact that $\Omega$ is a bounded set we have that 
\begin{align*}
\nabla\phi_K\left( \frac{d\mu}{|d\mu|}(\cdot) \right) \in L^p(\Omega,|\mu|;\R^n) \quad \forall p \geq 1.
\end{align*} 
%%%%% comment di Nicola %%%%%%
Let us call $f:=\nabla\phi_K\left( \frac{d\mu}{|d\mu|} \right) $. By \cite[Remark 1.46]{AFP} there exist a sequence $(g_h)_h\in C^0_c(\Omega;\R^n)$ such that $g_h\to f$ in $L^1(\Omega,|\mu|;\R^n)$. Since every function in $C^0_c$ can be uniformly approximated 
by functions in $C^1_c$ we can suppose without loss of generality that the sequence $(g_h)_h\in C^1_c(\Omega;\R^n)$. Now we consider the sequence $(\tilde{g})_h\in C^0_c(\Omega;\R^n)$ defined as
$$
\tilde{g}_h(x):=\frac{g_h(x)}{\phi_K^*(g_h(x))+1/h}\quad \forall\,h\in \mathbb{N}.
$$
By construction, up to a subsequence, we have that $\tilde{g}_h\to f$ $|\mu|$-a.e. on $\Omega$ and, thanks to the term $1/h$ in the denominator, $\tilde{g}_h(x)\in \mathring{K}$, so that $\phi_K^*(\tilde{g}_h(x))<1$ for every $h\in\mathbb{N}$ and for $|\mu|$-a.e. $x\in\Omega$. By the continuity of the functions $\tilde{g}_h$, for every $h\in\mathbb{N}$ there exists $\lambda=\lambda(h)>0$ such that $0<\lambda(h)<1$ and $\tilde{g}_h(x)\in \lambda(h)K$ for every $x\in \Omega$. Again, using the fact that $C^1_c(\Omega;\R^n)$ is dense in $C^0_c(\Omega;\R^n)$ we can proceed as follow: let $(\epsilon_h)_{h\in\mathbb{N}}$ be such that $\epsilon_h>0$ for every $h\in\mathbb{N}$ and $\epsilon_h\to 0$ for $h\to \infty$. For every $h\in\mathbb{N}$ let $f_h\in C^1_c(\Omega;\R^n)$ be such that 
$$
\sup_{x\in\Omega}\left|f_h(x)-\tilde{g}_h(x) \right|<\epsilon_h.
$$
Since $\dist(\partial (\lambda(h)K);\partial K)>0$ for every $h\in\mathbb{N}$ , choosing $\epsilon_h$ small enough we get that $\forall\, h\in\mathbb{N}$ $f_h(x)\in K$ for every $x\in\Omega$ . 
%%%%%%%%%%%%%%%%%%%%%%%%%%%%%
Thus, by the Lebesgue dominated convergence theorem
\begin{align*}
|\mu|_K(\Omega)&=\int_{\Omega} \phi_K\left(\frac{d\mu}{d|\mu|}(x)  \right)d|\mu|(x)=\int_{\Omega}\lim_{h\rightarrow \infty}f_h(x)\cdot \frac{d\mu}{d|\mu|}(x)d|\mu|(x)
\\ &= \lim_{h\rightarrow \infty}\int_{\Omega}f_h(x)\cdot \frac{d\mu}{d|\mu|}(x)d|\mu|(x)\leq \sup_{h\in \mathbb{N}}\int_{\Omega}f_h(x)\cdot \frac{d\mu}{d|\mu|}(x)d|\mu|(x)\\
&\leq \sup_{\underset{\phi_K^*(\varphi)\leq 1}{\varphi \in C^1_c(\Omega;\R^n),}}\int_{\Omega}\varphi(x)\cdot \frac{d\mu}{d|\mu|}(x)d|\mu|(x).
\end{align*}
This concludes step 2.\\
\textbf{Step 3} We want now to prove the statement for a generic $\phi_K$. Thus, thanks to Lemma \ref{lem:1.3 of my notes} consider $\{K_h  \}_{h\in \mathbb{N}}\subset \R^n$ a sequence as in (\ref{HP per K}) with $K_h\subset K_{h+1}\subset \dots \subset K$ and such that the sequence satisfies the assumptions of Lemma \ref{lem: 1.2 of my notes} so that  $\phi_{K_h}$ uniformly converges to $\phi_{K}$.
Therefore, applying step 2 we get
\begin{align*}
|\mu|_{K_h}(\Omega)&= \int_{\Omega}\phi_{K_h}\left(\frac{d\mu}{d|\mu|}  \right)d|\mu|(x)
=\sup_{\underset{\phi^*_{K_h}(\varphi)\leq 1}{\varphi \in C^1_c(\Omega;\R^n),}}\int_{\Omega}\varphi(x)\cdot \frac{d\mu}{d|\mu|}(x)d|\mu|(x)\\
&\leq \sup_{\underset{\phi^*_K(\varphi)\leq 1}{\varphi \in C^1_c(\Omega;\R^n),}}\int_{\Omega}\varphi(x)\cdot \frac{d\mu}{d|\mu|}(x)d|\mu|(x),
\end{align*}
where we used the fact that $\phi^*_K(\varphi)\leq 1$ as a consequence of $\phi^*_{K_h}(\varphi)\leq 1$ and of $K_h\subset K$. Now, thanks to the uniform convergence of the functions $\phi_{K_h}$ to $\phi_K$ we get
\begin{align*}
|\mu|_K(\Omega)&=\int_{\Omega}\phi_{K}\left(\frac{d\mu}{d|\mu|}  \right)d|\mu|(x)=\lim_{h\rightarrow +\infty}\int_{\Omega}\phi_{K_h}\left(\frac{d\mu}{d|\mu|}  \right)d|\mu|(x)\\
&\leq \sup_{\underset{\phi^*_K(\varphi)\leq 1}{\varphi \in C^1_c(\Omega;\R^n),}}\int_{\Omega}\varphi(x)\cdot \frac{d\mu}{d|\mu|}(x)d|\mu|(x).
\end{align*}
This concludes the proof in the case $\Omega$ open and bounded. From standard considerations about outer measures, the extension of this result for unbounded open set follows.
\end{proof}
\noindent
The following result is the anisotropic version of \cite[Lemma 3.7]{CCPMSteiner}.
\begin{lemma}\label{lem:1.9 pag 70}
If $\nu$ and $\mu$ are $\R^n$-valued Radon measure on $\R^{m}$, then
\begin{align}\label{eq:1.9 pag 70}
2|\mu|_{K}(G)\leq |\mu+\nu|_{K}(G)+ |\mu-\nu|_{K}(G)
\end{align}
for every Borel set $G\subset \R^{m}$.
\end{lemma}
\begin{proof}
Fix a generic partition of $G$ made by bounded Borel sets $\{ G_i \}_{i\in \mathbb{N}}$, by subadditivity we have
\begin{align*}
\phi_{K}\left(2\mu(G_i)\right)&= \phi_{K}\left(\mu(G_i)+ \nu(G_i) +\mu(G_i) -\nu(G_i)   \right)\\
&\leq \phi_{K}\left((\mu+\nu)(G_i)  \right) + \phi_{K}\left((\mu-\nu)(G_i)  \right).
\end{align*} 
Thus,
\begin{align*}
\sum_{i\in\mathbb{N}}\phi_{K}\left(2\mu(G_i)\right)&\leq \sum_{i\in\mathbb{N}}\left[ \phi_{K}\left((\mu+\nu)(G_i)  \right) + \phi_{K}\left((\mu-\nu)(G_i)  \right) \right].
\end{align*}
Then thanks to Lemma \ref{lem: AFP page 3} and passing to the $\sup$ in both sides we get
\begin{align*}
|2\mu|_{K}(G)&\leq \sup_{\{G_i\}}\sum_{i\in\mathbb{N}}\left[ \phi_{K}\left((\mu+\nu)(G_i)  \right) + \phi_{K}\left((\mu-\nu)(G_i)  \right) \right]\\
&\leq \sup_{\{G_i\}}\sum_{i\in\mathbb{N}} \phi_{K}\left((\mu+\nu)(G_i)  \right) + \sup_{\{G_k\}}\sum_{k\in\mathbb{N}}\phi_{K}\left((\mu-\nu)(G_k)  \right)\\
&=|\mu+\nu|_{K}(G)+|\mu-\nu|_{K}(G). 
\end{align*}
This concludes the proof.
\end{proof}
\begin{remark}
Let $\mu_1,\mu_2$ be $\R^n$-valued Radon measures on $\R^{m}$. Let us observe that, by (\ref{eq:1.9 pag 70}) with $\mu=\mu_1+\mu_2$ and $\nu=\mu_1-\mu_2$ we obtain
\begin{align}\label{eq: disuguaglianza triangola |mu|K}
|\mu_1+\mu_2|_K\leq |\mu_1|_K+|\mu_2|_K.
\end{align}
On the other hand, let $\nu_1,\nu_2$ be $\R^n$-valued Radon measures on $\R^{m}$.
Then, by the above relation with $\mu_1=\nu_1+\nu_2$ and $\mu_2=-\nu_2$ we get
\begin{align}\label{eq: disuguaglianza triangola inversa |mu|K}
|\nu_1+\nu_2|_K\geq |\nu_1|_K-|-\nu_2|_K.
\end{align}
\end{remark}

\begin{remark}\label{rem: case equality radon measure}
In this Remark we discuss the equality case for relation (\ref{eq:1.9 pag 70}). Let us assume that
\begin{align}\label{eq: case equality radon measure}
2|\mu|_K(G)= |\mu+\nu|_K(G)+ |\mu-\nu|_K(G)\qquad \forall\,\textit{Borel set }G\subset \R^{m}.
\end{align} 
We immediately observe that  if $|\mu|_{K}(G)=0$ then $|\mu+\nu|_{K}(G)=|\mu-\nu|_{K}(G)=|\nu|_{K}(G)=0$, so that
\begin{align*}
|\nu|_{K}\ll|\mu|_{K}.
\end{align*}
Thanks to Radon-Nikodym Theorem we know that $\exists\, g,h\in L^1_{loc}(\R^{m},|\mu|_{K};\R^n)$ s.t.
\begin{align*}
\nu=g|\mu|_{K}\quad\textit{and}\quad
\mu=h|\mu|_{K},
\end{align*}
thus, 
\begin{align*}
\mu\pm\nu=(h\pm g)|\mu|_{K}.
\end{align*}
Observing that 
\begin{align*}
|\mu \pm \nu|_{K}(G)=\int_G \phi_{K}\left(\frac{d(\mu\pm \nu)}{d|\mu\pm\nu|}(x)  \right)d|\mu\pm\nu|(x)=\int_G \phi_{K}\left(\frac{(h\pm g)(x)}{|h\pm g|(x)}  \right)|h\pm g|(x)\,d|\mu|_{K}(x),
\end{align*}
we can now rewrite 
(\ref{eq: case equality radon measure}) as
\begin{align*}
\int_{G}2\phi_{K}\left(h(x)\right)\,d|\mu|_{K}(x)= \int_{G}\phi_{K}\left((h + g)(x)\right)\,d|\mu|_{K}(x)+\int_{G}\phi_{K}\left((h- g)(x)\right)\,d|\mu|_{K}(x).
\end{align*}
that is 
\begin{align*}
\int_{G}\phi_{K}\left(2h(x)\right)-\phi_{K}\left((h + g)(x)\right)-\phi_{K}\left((h- g)(x)\right)\,d|\mu|_{K}(x)=0\quad \forall\,G\subset\R^{m}\textit{ Borel.}
\end{align*}
By subadditivity we get 
\begin{align*}
\phi_{K}\left(2h(x)\right)-\phi_{K}\left((h + g)(x)\right)-\phi_{K}\left((h- g)(x)\right)\leq 0\quad |\mu|_{K}\textit{-a.e.$\,x\in\R^{m}$},
\end{align*}
thus,
\begin{align}\label{eq:1.9star equality}
\phi_{K}(2h(x))=\phi_{K}\left((h+g)(x)\right)+\phi_{K}\left((h-g)(x)  \right) \quad |\mu|_{K}\textit{-a.e.$\,x\in\R^{m}$}.
\end{align}
Thus condition (\ref{eq: case equality radon measure}) is equivalent to (\ref{eq:1.9star equality}) that is equivalent to say, thanks to Proposition \ref{prop:linearityphi}, Remark \ref{rem: more appealing way (iii)} and relation (\ref{eq: more appealing way (iii)}) with $y_1=h+g$ and $y_2=h-g$, that for $|\mu|_{K}\textit{-a.e.$\,x\in\R^{m}$}$ $\exists\, z(x)\in \partial K$ s.t.
\begin{align}\label{eq:conecondition}
\left\{h(x)+tg(x):\,t\in[-1,1]\right\}\subset C^*_{K}(z(x)).
\end{align}
\begin{figure}
\centering
\def\svgwidth{9cm}
%% Creator: Inkscape inkscape 0.92.3, www.inkscape.org
%% PDF/EPS/PS + LaTeX output extension by Johan Engelen, 2010
%% Accompanies image file 'ilgrandeK.pdf' (pdf, eps, ps)
%%
%% To include the image in your LaTeX document, write
%%   \input{<filename>.pdf_tex}
%%  instead of
%%   \includegraphics{<filename>.pdf}
%% To scale the image, write
%%   \def\svgwidth{<desired width>}
%%   \input{<filename>.pdf_tex}
%%  instead of
%%   \includegraphics[width=<desired width>]{<filename>.pdf}
%%
%% Images with a different path to the parent latex file can
%% be accessed with the `import' package (which may need to be
%% installed) using
%%   \usepackage{import}
%% in the preamble, and then including the image with
%%   \import{<path to file>}{<filename>.pdf_tex}
%% Alternatively, one can specify
%%   \graphicspath{{<path to file>/}}
%% 
%% For more information, please see info/svg-inkscape on CTAN:
%%   http://tug.ctan.org/tex-archive/info/svg-inkscape
%%
\begingroup%
  \makeatletter%
  \providecommand\color[2][]{%
    \errmessage{(Inkscape) Color is used for the text in Inkscape, but the package 'color.sty' is not loaded}%
    \renewcommand\color[2][]{}%
  }%
  \providecommand\transparent[1]{%
    \errmessage{(Inkscape) Transparency is used (non-zero) for the text in Inkscape, but the package 'transparent.sty' is not loaded}%
    \renewcommand\transparent[1]{}%
  }%
  \providecommand\rotatebox[2]{#2}%
  \newcommand*\fsize{\dimexpr\f@size pt\relax}%
  \newcommand*\lineheight[1]{\fontsize{\fsize}{#1\fsize}\selectfont}%
  \ifx\svgwidth\undefined%
    \setlength{\unitlength}{781.72318106bp}%
    \ifx\svgscale\undefined%
      \relax%
    \else%
      \setlength{\unitlength}{\unitlength * \real{\svgscale}}%
    \fi%
  \else%
    \setlength{\unitlength}{\svgwidth}%
  \fi%
  \global\let\svgwidth\undefined%
  \global\let\svgscale\undefined%
  \makeatother%
  \begin{picture}(1,0.66490274)%
    \lineheight{1}%
    \setlength\tabcolsep{0pt}%
    \put(0,0){\includegraphics[width=\unitlength,page=1]{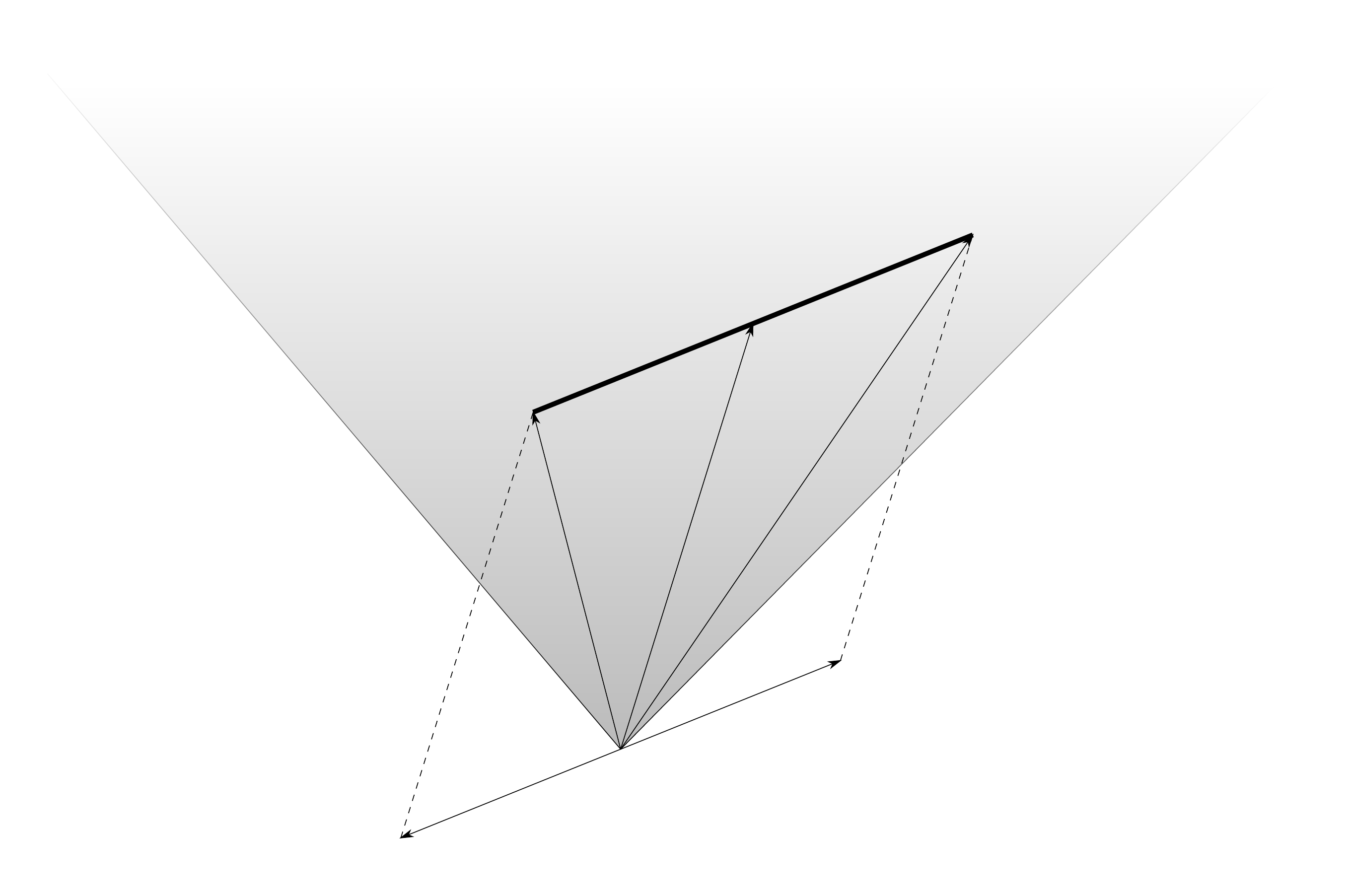}}%
    \put(0.13936692,0.52528448){\color[rgb]{0,0,0}\makebox(0,0)[lt]{\lineheight{0.82591271}\smash{\begin{tabular}[t]{l}\tiny{$C_K^*(\bar{z})$}\end{tabular}}}}%
    \put(0.55612493,0.37797221){\color[rgb]{0,0,0}\makebox(0,0)[lt]{\lineheight{0.82591271}\smash{\begin{tabular}[t]{l}\tiny{$h$}\end{tabular}}}}%
    \put(0.62396616,0.13955903){\color[rgb]{0,0,0}\makebox(0,0)[lt]{\lineheight{0.82591271}\smash{\begin{tabular}[t]{l}\tiny{$-g$}\end{tabular}}}}%
    \put(0.29251356,0.01162995){\color[rgb]{0,0,0}\makebox(0,0)[lt]{\lineheight{0.82591271}\smash{\begin{tabular}[t]{l}\tiny{$g$}\end{tabular}}}}%
    \put(0.46890063,0.06977949){\color[rgb]{0,0,0}\makebox(0,0)[lt]{\lineheight{0.82591271}\smash{\begin{tabular}[t]{l}\tiny{$O$}\end{tabular}}}}%
  \end{picture}%
\endgroup%

\caption{In this picture we give a 2-dimensional representation of condition (\ref{eq:conecondition}) where $h\in C^*_K(\bar{z})$ and $\bar{z}$ is a fixed point in the boundary of the Wulff shape $K$.}
\label{fig:ilgrandeK}
\end{figure}
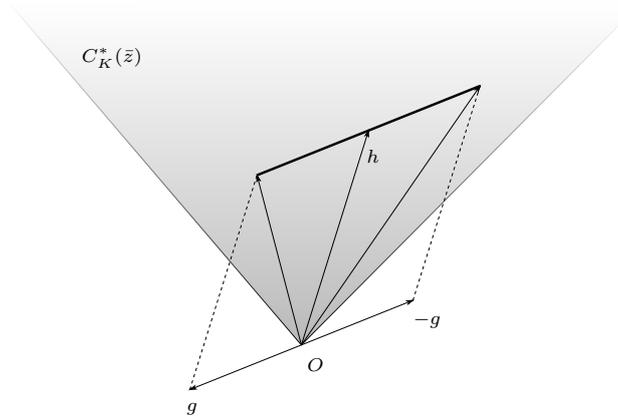
%
%
%\begin{align*}
%\frac{(h+g)(x)}{\phi_{K}\left((h+g)(x)\right)},\frac{(h-g)(x)}{\phi_{K}\left((h-g)(x)\right)}\in \partial \phi^*_{K}(z(x))
%\end{align*}
%
%
%\noindent
%i.e. $(h+g)(x),(h-g)(x)\in C^*_{K}(z(x))$ (see Figure \ref{fig:ilgrandeK}). As observed in Remark \ref{rem: more appealing way (iii)} if $(h+g)(x), (h-g)(x)\in C^*_{K}(z(x))$ then any convex combination of them belongs to $C^*_{K}(z(x))$. , having in mind (\ref{eq: more appealing way (iii)}) with $y_1=h+g$ and $y_2=h-g$, that for $|\mu|_K\textit{-a.e.$\,x\in\R^{n-1}$}$ $\exists\, z(x)\in \partial K$ s.t.
%\begin{align}\label{eq:conecondition}
%\left\{h(x)+tg(x):\,t\in[-1,1]\right\}\subset C^*_{K}(z(x)).
%\end{align}
\end{remark}

\section{A formula for the anisotropic perimeter}\label{section 5}\noindent
In this section we will prove a formula for the anisotropic perimeter valid for a specific class of sets of finite perimeter. We recall that, given $u:\R^{n-1}\rightarrow \R$, we denote by $\Sigma_u=\{x\in \R^n:\, \q x > u(\p x)  \}$ and $\Sigma^u=\{x\in \R^n:\, \q x < u(\p x)  \}$ the epigraph and the subgraph of $u$, respectively. As proved in \cite[Proposition 3.1]{CCPMGauss}, $\Sigma_u$ is a set of locally finite perimeter if and only if $\tau_M(u)\in BV_{loc}(\R^{n-1})$ for every $M>0$.
\noindent
Through all this section, given $u\in BV_{loc}(\R^{n-1})$ we consider $\eta:=(Du,-\mathcal{L}^{n-1})$ a $\R^n$-valued Radon measure on $\R^{n-1}$.
\begin{theorem}\label{thm: 1 pag 59 note}
Let $K\subset\R^n$ as in (\ref{HP per K}) and let $u\in BV_{loc}(\R^{n-1})$, then
\begin{align*}
|\eta|_K(B) = |D1_{\Sigma^u}|_{K}(B\times \R) \quad \forall B\subset \R^{n-1}\; Borel.
\end{align*}
\end{theorem}
\begin{proof}
Thanks to Theorem \ref{thm:Charac.anis.total.var.}, the identity follows from a careful inspection of the proof of \cite[Theorem 1 in (part 4, Section 1.5)]{GiaMoSou}. It is important to notice that in the present situation one should replace condition $|\varphi|\leq 1$ with $\phi_{K^s}^*(\varphi)\leq 1$ with $\varphi \in C^1_c(\R^n;\R^n)$.
\end{proof}
\noindent
We recall now an important result concerning how to determine $\nu^{\Sigma^u}$ i.e. the outer normal to the reduced boundary of  the subgraph of the function $u$. Recall that thanks to Radon-Nykodym Theorem we have 
\begin{align*}
Du=D^au+ D^ju + D^cu.
\end{align*}
With a little abuse of notation let us call
$D^{ac}u=D^au+D^cu$, so that 
$$
D^cu=D^{ac}u\mres Z_u
$$
where, 
\begin{align*}
Z_u=\left\{x \in \Omega:\, \frac{d|D^{ac}u|}{d\mathcal{L}^{n-1}}(x)=+\infty  \right\}.
\end{align*}
\begin{theorem}\label{thm:1.6 pag 57 note}
Let $u\in BV(\Omega)$ with $\Omega \subset \R^{n-1}$ open and bounded, then
\begin{itemize}
\item[i)] for $|\eta|$-a.e. $x \in \Omega\setminus J_u$ we have
\begin{align*}
\frac{d\eta}{d|\eta|}(x)= -\nu^{\Sigma^u}(x,u(x)),
\end{align*}
\item[ii)] for $|\eta|$-a.e. $x \in J_u$ we have
\begin{align*}
\frac{d\eta}{d|\eta|}(x)=\left(\frac{dD^ju}{d|D^ju|}(x),0  \right)= (\nu_u(x),0) =-\nu^{\Sigma^u}(x,y) \quad \forall\, y\,s.t.\,(x,y)\in\partial^{*}\Sigma^u,
\end{align*}
\item[iii)] for $|\eta|$-a.e. $x \in \left((\Omega\setminus J_u )\cap \left\{x\in \Omega:\, \q\nu^{\Sigma^u}(x,u\Upp(x))=0  \right\}\right)$ we have 
\begin{align*}
\frac{d\eta}{d|\eta|}(x)= \left(\frac{dD^cu}{d|D^cu|}(x),0  \right).
\end{align*}
\end{itemize}
\end{theorem}
\begin{proof}
Statement $(i)$ is proved in $(i)$ of \cite[Theorem 4 in (Part 4, Section 1.5)]{GiaMoSou}. Statement $(ii)$ follows by combining $(ii)$ of \cite[Theorem 4 in (Part 4, Section 1.5)]{GiaMoSou} with $(ii)$ of \cite[Theorem~3 in (Part 4, Section 1.5)]{GiaMoSou}. We will give a proof of point $iii)$. Let $x\in\Omega$ and consider $\rho>0$, and recall (\ref{eq: disc}), then
\begin{align*}
|\eta|(D_{x,\rho})&= \underset{f\in C^0_c(D_{x,\rho},\R^n)}{\sup_{|f|\leq 1}}\int_{D_{x,\rho}}f(y)\cdot d\eta(y)\\
&=\underset{f\in C^0_c(D_{x,\rho},\R^n)}{\sup_{|f|\leq 1}}\left( \int_{D_{x,\rho}}(f_1(y),\dots,f_{n-1}(y))\cdot d Du(y)- \int_{D_{x,\rho}}f_n(y) dy\right)\\
&\leq \underset{f\in C^0_c(D_{x,\rho},\R^n)}{\sup_{|f|\leq 1}}\int_{D_{x,\rho}}(f_1(y),\dots,f_{n-1}(y))\cdot d Du(y) + \underset{f\in C^0_c(D_{x,\rho},\R^n)}{\sup_{|f|\leq 1}}\int_{D_{x,\rho}}f_n(y) dy\\
&= |Du|(D_{x,\rho})+\mathcal{L}^{n-1} (D_{x,\rho}). 
\end{align*}
At the same time we get
\begin{align*}
|\eta|(D_{x,\rho})&= \underset{f\in C^0_c(D_{x,\rho},\R^n)}{\sup_{|f|\leq 1}}\int_{D_{x,\rho}}f(y)\cdot d\eta(y)\\
&\geq \int_{D_{x,\rho}}(f_1(y),\dots,f_{n-1}(y))\cdot d Du(y),
\end{align*}
so that, passing to the sup in the right hand side, it holds
\begin{align*}
|\eta|(D_{x,\rho}) \geq |Du|(D_{x,\rho}).
\end{align*}
Putting together these two inequalities we get
\begin{align}\label{eq: PRE 2 carabienieri}
|Du|(D_{x,\rho})\leq |\eta|(D_{x,\rho})\leq |Du|(D_{x,\rho}) + \mathcal{L}^{n-1}(D_{x,\rho}).
\end{align}
Let now $x\in Z_u$ and let $\rho>0$. Then,
\begin{align*}
\frac{\eta(D_{x,\rho})}{|\eta|(D_{x,\rho})}=\frac{\eta(D_{x,\rho})}{|Du|(D_{x,\rho})}\frac{|Du|(D_{x,\rho})}{|\eta|(D_{x,\rho})}.
\end{align*}
Since 
$$
\lim_{\rho\rightarrow 0^+}\frac{\eta(D_{x,\rho})}{|Du|(D_{x,\rho})}=\left( \frac{d D^cu}{d| D^cu|}(x),0\right),
$$
we are left to prove that 
\begin{align}\label{eq: Cantor e Ballor}
\lim_{\rho\rightarrow 0^+}\frac{|Du|(D_{x,\rho})}{|\eta|(D_{x,\rho})}=1.
\end{align}
Thanks to (\ref{eq: PRE 2 carabienieri}) we have 
\begin{align}\label{eq:due carabienieri}
\frac{|Du|(D_{x,\rho})}{|Du|(D_{x,\rho})+|D_{x,\rho}|}\leq \frac{|Du|(D_{x,\rho})}{|\eta|(D_{x,\rho})}\leq \frac{|Du| (D_{x,\rho})}{|Du|(D_{x,\rho})}=1.
\end{align}
Recall that $x\in Z_u$, so that 
\begin{align*}
\lim_{\rho\rightarrow 0^+} \frac{|D_{x,\rho}|}{|Du|(D_{x,\rho})}= 0.
\end{align*}
Thus, we can calculate the following limit for the left hand side of (\ref{eq:due carabienieri})
\begin{align*}
\lim_{\rho\rightarrow 0^+}\frac{|Du|(D_{x,\rho})}{|Du|(D_{x,\rho})+|D_{x,\rho}|}&= \lim_{\rho\rightarrow 0^+}\frac{1}{1+\frac{|D_{x,\rho}|}{|Du|(D_{x,\rho})}}
= 1.
\end{align*}
By the above calculation and relation (\ref{eq:due carabienieri}) we proved (\ref{eq: Cantor e Ballor}) and so we conclude the proof.
\end{proof}

\begin{proposition}\label{prop: perimetro anisotropico sopragrafico di u}
Let $u\in BV_{loc}(\R^{n-1})$ and let $K\subset \R^n$ be as in (\ref{HP per K}). Then, for every Borel set $B\subset \R^{n-1}$ we have
\begin{align}
P_{K}(\Sigma^u;B\times \R)&= \int_{B\setminus (J_u \cup Z_u)}\phi_{K}(-\nabla u(x),1)dx\\
&+ \int_{B\cap J_u}[u](x) \phi_{K}\left(-\frac{dD^ju}{d|D^ju|}(x),0\right) d\mathcal{H}^{n-2}(x)\nonumber\\
&+\int_{B\cap Z_u}\phi_{K}\left( -\frac{dD^cu}{d|D^cu|}(x),0 \right) d|D^cu|(x),\nonumber
\end{align} 
where $Z_u$ has been defined at the beginning of this Section.
\end{proposition}
\begin{proof}
Let us consider a generic Borel set $B\subset \R^{n-1}$. Then, thanks to the De Giorgi structure Theorem, Theorem \ref{thm: 1 pag 59 note}, and Theorem \ref{thm:1.6 pag 57 note} we get 
\begin{align*}
P_{K}(\Sigma^u;B \times \R)&=\int_{\partial^* \Sigma^u \cap (B\times \R)}\phi_{K}(\nu^{\Sigma^u}(x)) d\mathcal{H}^{n-1}(x)\\
&= \int_{\partial^* \Sigma^u \cap (B\times \R)}\phi_{K}\left( -\frac{dD1_{\Sigma^u}}{d|D1_{\Sigma^u}|}(x) \right)d|D1_{\Sigma^u}|(x)\\
&=\int_{B}\phi_{K}\left(-\frac{d\eta}{d|\eta|}(x)  \right) d|\eta|(x).
\end{align*} 
Let us split the last integral in the following way
\begin{align}\label{eq: alpha}
\int_{B}\phi_{K}\left(-\frac{d\eta}{d|\eta|}(x)  \right) d|\eta|(x)&= \int_{B\setminus (J_u \cup Z)}\phi_{K}\left(-\frac{d\eta}{d|\eta|}(x)  \right) d|\eta|(x)\nonumber\\
&+\int_{B\cap J_u}\phi_{K}\left(-\frac{d\eta}{d|\eta|}(x)  \right) d|\eta|(x)\nonumber\\
&+\int_{B\cap Z}\phi_{K}\left(-\frac{d\eta}{d|\eta|}(x)  \right) d|\eta|(x).
\end{align}
About the first integral on the right hand side we observe that
\begin{align*}
\eta \mres \R^{n-1}\setminus(J_u \cup Z_u)= (D^au,-\mathcal{L}^{n-1})\mres \R^{n-1} = (\nabla u,-1)\mathcal{L}^{n-1}\mres \R^{n-1}.
\end{align*}
Therefore, recalling Remark \ref{rem: 4.8 Maggi} we have
\begin{align*}
\eta(B)= \int_B (\nabla u,-1)\,dx\quad\textit{and}\quad |\eta|(B)=\int_B\sqrt{|\nabla u|^2+1}\, dx\quad\forall\,\textit{Borel set }B\subset \R^{n-1}\setminus(J_u \cup Z).
\end{align*}
Thus,
\begin{align}
\int_{B\setminus (J_u \cup Z)}\phi_{K}\left(-\frac{d\eta}{d|\eta|}(x)  \right) d|\eta|(x)&=\int_{B\setminus (J_u \cup Z)}\phi_{K}\left( \frac{(-\nabla u(x),1)}{\sqrt{|\nabla u|^2+1}} \right)\sqrt{|\nabla u|^2 +1}\, dx\nonumber\\
&=\int_{B\setminus (J_u \cup Z)}\phi_{K}(-\nabla u(x),1)\, dx.\label{eq:abscont}
\end{align}
Let us observe now that, thanks to $(ii)$ of Theorem \ref{thm:1.6 pag 57 note} 
\begin{align*}
\eta \mres J_u = (D^ju,-\mathcal{L}^{n-1})\mres J_u = (D^ju,0)\mres J_u.
\end{align*}
Thus, 
\begin{align*}
|\eta|(B)= |D^ju|(B)\quad \forall\,\textit{Borel set } B\subset J_u.
\end{align*}
Then,
\begin{align}\label{eq:jump}
\int_{B\cap J_u}\phi_{K}\left(-\frac{d\eta}{d|\eta|}(x)  \right) d|\eta|(x)&=\int_{B\cap J_u}\phi_{K}\left(-\frac{d D^ju}{d| D^ju|}(x),0  \right) d|D^ju|(x)\nonumber\\
&=\int_{B\cap J_u}\phi_{K}\left(-\frac{d D^ju}{d| D^ju|}(x),0  \right)[u](x) d\mathcal{H}^{n-2}(x).
\end{align}
A similar argument holds for the integral over $B\cap Z_u$, so that 
\begin{align}\label{eq:Cantor}
\int_{B\cap Z_u}\phi_{K}\left(-\frac{d\eta}{d|\eta|}(x)  \right) d|\eta|(x)= \int_{B\cap Z_u}\phi_{K}\left(-\frac{d D^cu}{d| D^cu|}(x),0  \right) d|D^cu|(x).
\end{align}
Combining equations (\ref{eq: alpha}), (\ref{eq:abscont}), (\ref{eq:jump}) and (\ref{eq:Cantor}) we conclude.
\end{proof}
\begin{remark}
We can also use the notation of the anisotropic total variation to obtain a more compact formula for the perimeter,
\begin{align*}
P_{K}(\Sigma^u;B\times \R)&=\int_{B}\phi_{K}(-\nabla u(x),1)dx + |(-D^ju,0)|_K(B) + |(-D^cu,0)|_K(B).
\end{align*}
\end{remark}
\begin{remark}\label{rem:10 Correzioni filippo}
Note that, since $\Sigma_u=\R^n\setminus \Sigma^u$, we have $\partial^*\Sigma_u=\partial^*\Sigma^u$ and $\nu^{\Sigma_u}(x)=-\nu^{\Sigma^u}(x)$ for $\h^{n-1}$-a.e. $x\in\partial^*\Sigma_u$, and so
$$
P_{K}(\Sigma_u;B\times \R)=\int_{B}\phi_{K}(\nabla u(x),-1)dx + |(D^ju,0)|_K(B) + |(D^cu,0)|_K(B)
$$
for every Borel set $B\subset \R^{n-1}$. 

\end{remark}

\noindent
Before stating the next result, we recall that given a Borel function $f:\R^{n-1}\to \R$, we indicate with $\tilde{f}$ the approximate average of $f$ defined as in (\ref{eq: limit relation}).
\begin{lemma}\label{lem:1.8 pag 65}
Let $K\subset \R^n$ be as in (\ref{HP per K}). If $u_1,u_2 \in BV_{loc}(\R^{n-1})$ with $u_1\leq u_2$ and $E=\Sigma_{u_1}\cap \Sigma^{u_2}$ has finite volume, then $E$ is a set of locally finite perimeter in $\R^n$ and for every Borel set $B\subset \R^{n-1}$  
\begin{align}\label{eq:formula perimeter BV}
P_{K}(E;B\times \R)
&=\int_{B\cap \{\widetilde{u_1} < \widetilde{u_2}\}}\phi_{K}(\nabla u_1(x),-1) dx +\int_{B\cap \{\widetilde{u_1} < \widetilde{u_2}\}}\phi_{K}(-\nabla u_2(x),1) dx \\
&+\int_{B\cap J_{u_1}}\phi_{K}\left(\nu_{u_1}(z),0  \right)\left(\min(u_1\Upp(z),u_2\Low(z))-u_1\Low(z)  \right)d\mathcal{H}^{n-2}(z)\nonumber \\
&+\int_{B\cap J_{u_2}}\phi_{K}\left(\nu_{u_2}(z),0  \right)\left(u_2\Upp(z)-\max( u_2\Low(z),u_1\Upp(z))  \right)d\mathcal{H}^{n-2}(z)\nonumber \\
&+|(D^cu_1,0)|_K(B\cap \{\widetilde{u_1} < \widetilde{u_2}\})+|(-D^cu_2,0)|_K(B\cap \{\widetilde{u_1} < \widetilde{u_2}\})\nonumber
\end{align}
\end{lemma}
\begin{proof}
We will follow the strategy of \cite[Theorem 3.1]{CCPMSteiner}. By \cite[Theorem 16.3]{MaggiBOOK}, if $F_1$, $F_2$ are sets of locally finite perimeter in $\R^n$, then
\begin{align}\label{eq:3.12 Cagnetti}
\partial^*(F_1\cap F_2)=_{\mathcal{H}^{n-1}} \left(F_1^{(1)}\cap \partial^*F_2  \right)\cup \left(F_2^{(1)}\cap \partial^*F_1  \right)\cup \left(\partial^*F_1 \cap \partial^*F_2 \cap\{\nu^{F_1}=\nu^{F_2}  \}  \right).
\end{align} 
Moreover, in the particular case of $F_1\subset F_2$, then $\nu^{F_1}=\nu^{F_2}$ $\mathcal{H}^{n-1}$-a.e. on $\partial^*F_1 \cap \partial^*F_2$. Let us observe that $u_1\leq u_2$ implies $\Sigma_{u_2}\subset \Sigma_{u_1}$ and that $\Sigma^{u_2}=\R^n\setminus \Sigma_{u_2}$ implying $\nu_{\Sigma_{u_2}}=-\nu_{\Sigma^{u_2}}$ $\mathcal{H}^{n-1}$-a.e. on $\partial^* \Sigma_{u_2}$. We thus find
\begin{align}\label{eq:3.13 Cagnetti}
\nu^{\Sigma_{u_1}}=-\nu^{\Sigma^{u_2}},\quad \mathcal{H}^{n-1}\textit{-a.e. on }\partial^*\Sigma_{u_1}\cap \partial^*\Sigma^{u_2} .
\end{align}
By, (\ref{eq:3.12 Cagnetti}) and (\ref{eq:3.13 Cagnetti}), since $E=\Sigma_{u_1}\cap \Sigma^{u_2}$ we find
\begin{align*}
\E=_{\mathcal{H}^{n-1}}\left(\partial^*\Sigma_{u_1}\cap (\Sigma^{u_2})^{(1)}  \right) \cup \left(\partial^*\Sigma^{u_2}\cap (\Sigma_{u_1})^{(1)}  \right).
\end{align*}
Recall the definition of the \emph{approximate discontinuity set} of $u_i$ with $i=1,2$, that we denote $S_{u_i}$ (see (\ref{eq: approx disc set})). Thanks to \cite[Section 4.1.5]{GiaMoSou} we know that $\Sigma_{u_1}$ and $\Sigma^{u_2}$ are sets of locally finite perimeter in $\R^n$ with
\begin{align}
\partial^*\Sigma^{(1)}_{u_1}\cap (S^c_{u_1}\times \R)&=_{\h^{n-1}}\left\{x\in\R^n:\,\tilde{u}_1(\p x)=\q x  \right\},\label{eq: 3.6 Filippo}\\
\partial^*\Sigma^{(1)}_{u_1}\cap (S_{u_1}\times \R)&=_{\h^{n-1}}\left\{x\in\R^n:\,u\Low_1(\p x)<\q x <u\Upp_1(\p x) \right\},\label{eq: 3.7 Filippo}\\
\Sigma^{(1)}_{u_1}\cap (S^c_{u_1}\times \R)&=_{\h^{n-1}}\left\{x\in\R^n:\,\tilde{u}_1(\p x)<\q x  \right\},\label{eq: 3.8 Filippo}\\
\Sigma^{(1)}_{u_1}\cap (S_{u_1}\times \R)&=_{\h^{n-1}}\left\{x\in\R^n:\,u\Upp_1(\p x)<\q x  \right\},\label{eq: 3.9 Filippo}\\
(\Sigma^{u_2})^{(1)}\cap (S^c_{u_1}\times \R)&=_{\h^{n-1}}\left\{x\in\R^n:\,\tilde{u}_2(\p x)>\q x  \right\},\label{eq: 3.10 Filippo}\\
(\Sigma^{u_2})^{(1)}\cap (S_{u_1}\times \R)&=_{\h^{n-1}}\left\{x\in\R^n:\,u\Low_2(\p x)>\q x  \right\}.\label{eq: 3.11 Filippo}
\end{align}
We now focus on the set $\partial^*\Sigma_{u_1}\cap (\Sigma^{u_2})^{(1)}$. Observe that,
\begin{align*}
P_{K}\left(\Sigma_{u_1}; (\Sigma^{u_2})^{(1)}\cap (B\times \R)  \right)&= P_{K}\left(\Sigma_{u_1}; (\Sigma^{u_2})^{(1)}\cap [(B\cap J_{u_1}^c \cap J_{u_2}^c )\times \R]  \right)\\
&+P_{K}\left(\Sigma_{u_1}; (\Sigma^{u_2})^{(1)}\cap [(B\cap J_{u_1} \cap J_{u_2}^c )\times \R]  \right)\\
&+P_{K}\left(\Sigma_{u_1}; (\Sigma^{u_2})^{(1)}\cap [(B\cap J_{u_1}\cap J_{u_2} )\times \R]  \right)\\
&+P_{K}\left(\Sigma_{u_1}; (\Sigma^{u_2})^{(1)}\cap [(B\cap J_{u_1}^c \cap J_{u_2} )\times \R]  \right).
\end{align*}
Applying (\ref{eq: 3.6 Filippo}) to $u_1$ and (\ref{eq: 3.10 Filippo}) to $u_2$ we find
\begin{align}\label{eq: 3.14 Filippo}
\left(\partial^*\Sigma_{u_1}\cap (\Sigma^{u_2})^{(1)}\right)\cap \left((J_{u_1}^c\cap J_{u_2}^c)\times\R\right)=_{\h^{n-1}}\left\{(z,\tilde{u}_1(z)):\,z\in (J_{u_1}^c\cap J_{u_2}^c),\,  \tilde{u}_1(z)<\tilde{u}_2(z) \right\}.
\end{align}
Applying (\ref{eq: 3.7 Filippo}) to $u_1$ and (\ref{eq: 3.10 Filippo}) to $u_2$ we obtain
\begin{align}\label{eq: 3.15 Filippo}
&\left(\partial^*\Sigma_{u_1}\cap (\Sigma^{u_2})^{(1)}\right)\cap \left((J_{u_1}\cap J_{u_2}^c)\times\R\right)\\&=_{\h^{n-1}}\left\{(z,t):\,z\in (J_{u_1}\cap J_{u_2}^c),\, u_1\Low (z)< t< \min(u_1\Upp(z),\widetilde{u}_2(z)) \right\}.  \nonumber
\end{align}
Combining (\ref{eq: 3.7 Filippo}) to $u_1$ and (\ref{eq: 3.11 Filippo}) to $u_2$ we obtain
\begin{align}\label{eq: 3.16 Filippo}
&\left(\partial^*\Sigma_{u_1}\cap (\Sigma^{u_2})^{(1)}\right)\cap \left((J_{u_1}\cap J_{u_2})\times\R\right)\\&=_{\h^{n-1}}\left\{(z,t):\,z\in (J_{u_1}\cap J_{u_2}),\, u_1\Low (z)< t< \min(u_1\Upp(z),u\Low_2(z)) \right\}.  \nonumber
\end{align}
Finally, applying (\ref{eq: 3.6 Filippo}) to $u_1$ and (\ref{eq: 3.11 Filippo}) to $u_2$ we get
\begin{align}\label{eq: 3.17 Filippo}
&\left(\partial^*\Sigma_{u_1}\cap (\Sigma^{u_2})^{(1)}\right)\cap \left((J_{u_1}^c\cap J_{u_2})\times\R\right)\\&=_{\h^{n-1}}\left\{(z,\widetilde{u}_1(z)):\,z\in (J_{u_1}^c\cap J_{u_2}),\, \widetilde{u}_1(z)< u\Low_2(z) \right\}.  \nonumber
\end{align}
Thus, thanks to Remark \ref{rem:10 Correzioni filippo} and (\ref{eq: 3.14 Filippo}) we get
\begin{small}
\begin{align*}
P_{K}\left(\Sigma_{u_1}; (\Sigma^{u_2})^{(1)}\cap [(B\cap J_{u_1}^c \cap J_{u_2}^c )\times \R]  \right)= \int_{\partial^*\Sigma_{u_1}\cap \left[(B\cap J_{u_1}^c\cap J_{u_2}^c \cap \{\widetilde{u_1} < \widetilde{u_2}\})\times \R\right]}\phi_{K}(-\nu^{\Sigma^{u_1}}(x))d\mathcal{H}^{n-1}(x)\\
=\int_{B\cap \{\widetilde{u_1} < \widetilde{u_2}\}}\phi_{K}(\nabla u_1(x),1) dx 
+ |(D^cu_1,0)|_K(B\cap \{\widetilde{u_1} < \widetilde{u_2}\}).
\end{align*}
\end{small}
Using Fubini theorem and (\ref{eq: 3.15 Filippo}) we get
\begin{align*}
&P_{K}\left(\Sigma_{u_1}; (\Sigma^{u_2})^{(1)}\cap [(B\cap J_{u_1} \cap J_{u_2}^c )\times \R]  \right)\\
&= \int_{\partial^*\Sigma_{u_1}\cap [((\Sigma^{u_2})^{(1)}\cap B\cap J_{u_1} \cap J_{u_2}^c )\times \R]}\phi_{K}(-\nu^{\Sigma^{u_1}}(y))d\mathcal{H}^{n-1}(y)\\
&=\int_{\{x\in \R^n:\,\p x\in B\cap J_{u_1} \cap J_{u_2}^c,\, u_1\Low (\p x)< \q x< \min(u_1\Upp(\p x),\widetilde{u}_2(\p x))   \}}\phi_{K}(-\nu^{\Sigma^{u_1}}(y))d\mathcal{H}^{n-1}(y)\\
&=\int_{(B\cap J_{u_1} \cap J_{u_2}^c)\times \R}\phi_{K}(-\nu^{\Sigma^{u_1}}(y))1_{\{\q x> u_1\Low(\p x)\}}(y)1_{\{\q x< \min(u_1\Upp(\p x),\widetilde{u}_2(\p x))\}}(y)d\mathcal{H}^{n-1}(y)\\
&=\int_{B\cap J_{u_1} \cap J_{u_2}^c}d\mathcal{H}^{n-2}(z)\int_{\R}\phi_{K}(-\nu^{\Sigma^{u_1}}(z,t))1_{\{s> u_1\Low(z)\}}(z,t)1_{\{s< \min(u_1\Upp(z),\widetilde{u}_2(z))\}}(z,t)d\mathcal{H}^{1}(t)\\
&=\int_{B\cap J_{u_1} \cap J_{u_2}^c}d\mathcal{H}^{n-2}(z)\int_{\R}\phi_{K}(\nu_{u_1}(z),0)1_{\{t> u_1\Low(z)\}}(z,t)1_{\{t< \min(u_1\Upp(z),\widetilde{u}_2(z))\}}(z,t)d\mathcal{H}^{1}(t)\\
&=\int_{B\cap J_{u_1} \cap J_{u_2}^c}\phi_{K}\left(\nu_{u_1}(z),0  \right)\left(\min(u_1\Upp(z),\widetilde{u}_2(z))-u_1\Low(z)  \right)d\mathcal{H}^{n-2}(z).
\end{align*}
Observe that we could have used $u_2\Low$ or $u_2\Upp$ instead of $\widetilde{u_2}$ since we are working in $B\cap J_{u_1} \cap J_{u_2}^c$. For similar arguments, using (\ref{eq: 3.16 Filippo}) we get that 
\begin{align*}
&P_{K}\left(\Sigma_{u_1}; (\Sigma^{u_2})^{(1)}\cap [(B\cap J_{u_1}\cap J_{u_2} )\times \R]  \right)\\&=\int_{B\cap J_{u_1} \cap J_{u_2}}\phi_{K}\left(\nu_{u_1}(z),0  \right)\left(\min(u_1\Upp(z),u_2\Low(z))-u_1\Low(z)  \right)d\mathcal{H}^{n-2}(z).
\end{align*}
Furthermore, thanks to (\ref{eq: 3.17 Filippo}) we deduce that $\h^{n-1}\left(\partial^*\Sigma_{u_1}\cap (\Sigma^{u_2})^{(1)}\cap (J_{u_1}^c \cap J_{u_2} )\times \R]  \right)=0$.  Thus, we have that
\begin{align*}
P_{K}\left(\Sigma_{u_1}; (\Sigma^{u_2})^{(1)}\cap [(B\cap J_{u_1}^c \cap J_{u_2} )\times \R]  \right)=0.
\end{align*} 
Therefore, 
\begin{align}\label{eq: alpha1}
P_{K}\left(\Sigma_{u_1}; (\Sigma^{u_2})^{(1)}\cap (B\times \R)  \right)&=\int_{B\cap \{\widetilde{u_1} < \widetilde{u_2}\}}\phi_{K}(\nabla u_1(x),-1) dx \\
&+\int_{B\cap J_{u_1}}\phi_{K}\left(\nu_{u_1}(z),0  \right)\left(\min(u_1\Upp(z),u_2\Low(z))-u_1\Low(z)  \right)d\mathcal{H}^{n-2}(z)\nonumber\\
&+ |(D^cu_1,0)|_K(B\cap \{\widetilde{u_1} < \widetilde{u_2}\}).
\end{align}
By symmetry, we got that 
\begin{align}\label{eq: beta1}
P_{K}\left(\Sigma^{u_2}; (\Sigma_{u_1})^{(1)}\cap (B\times \R)  \right)&=\int_{B\cap \{\widetilde{u_1} < \widetilde{u_2}\}}\phi_{K}(-\nabla u_2(x),1) dx \\
&+ \int_{B\cap J_{u_2}}\phi_{K}\left(\nu_{u_2}(z),0  \right)\left(u_2\Upp(z)-\max( u_2\Low(z),u_1\Upp(z))  \right)d\mathcal{H}^{n-2}(z)\nonumber\\
&+ |(-D^cu_2,0)|_K(B\cap \{\widetilde{u_1} < \widetilde{u_2}\}).\nonumber
\end{align}
Putting together (\ref{eq: alpha1}) and (\ref{eq: beta1}) we obtain the formula for $P_{K}(E;B\times \R)$. 
\end{proof}
\noindent
We now extend Lemma \ref{lem:1.8 pag 65} to the case of $GBV$ functions.
\begin{theorem}\label{thm:perimetro GBV}
Let $K\subset \R^n$ be as in (\ref{HP per K}). If $u_1,u_2 \in GBV(\R^{n-1})$ with $u_1\leq u_2$ and $E=\Sigma_{u_1}\cap \Sigma^{u_2}$ has finite volume, then $E$ is a set of locally finite perimeter and for every Borel set $B\subset \R^{n-1}$  
\begin{align}\label{eq:formula perimetro GBV}
P_{K}(E;B\times \R)
&=\int_{B\cap \{u_1 < u_2\}}\phi_{K}(\nabla u_1(x),-1) dx +\int_{B\cap \{u_1 < u_2\}}\phi_{K}(-\nabla u_2(x),1) dx \nonumber \\
&+\int_{B\cap J_{u_1}}\phi_{K}\left(\nu_{u_1}(z),0  \right)\left(\min(u_1\Upp(z),u_2\Low(z))-u_1\Low(z)  \right)d\mathcal{H}^{n-2}(z) \nonumber\\
&+\int_{B\cap J_{u_2}}\phi_{K}\left(-\nu_{u_2}(z),0  \right)\left(u_2\Upp(z)-\max( u_2\Low(z),u_1\Upp(z))  \right)d\mathcal{H}^{n-2}(z) \\
&+ |(D^cu_1,0)|_K(B\cap \{\tilde{u_1} < \tilde{u_2}\})+ |(-D^cu_2,0)|_K(B\cap \{\tilde{u_1} < \tilde{u_2}\}). \nonumber 
\end{align}
\end{theorem}
\begin{proof}
To prove (\ref{eq:formula perimetro GBV}) it suffices to consider the case where $B$ is bounded since (\ref{eq:formula perimetro GBV}) is an identity between Borel measures on $\R^{n-1}$. Given $M>0$, let $E_M=\Sigma_{\tau_M(u_1)}\cap \Sigma^{\tau_M(u_2)}$. Since $\tau_M(u_i)\in BV_{loc}(\R^{n-1})$ for every $M>0$, $i=1,2$, by Lemma \ref{lem:1.8 pag 65} we find that $E_M$ is a set of locally finite perimeter and that (\ref{eq:formula perimeter BV}) holds true on $E_M$ with $\tau_M(u_1)$ and $\tau_M(u_2)$ in place of $u_1$ and $u_2$. To complete the proof of the theorem we are going to show the following identities
\begin{align}
P_{K}(E;B\times\R)&=\lim_{M\rightarrow +\infty}P_{K}(E_M;B\times \R) \label{eq:lim perimetro P(...)}\\
\int_{B\cap \{u_1 < u_2\}}\phi_{K}(\nabla u_1(x),-1) dx&= \lim_{M\rightarrow +\infty}\int_{B\cap \{\tau_{M}(u_1) < \tau_{M}(u_2)\}}\phi_{K}(\nabla \tau_{M}(u_1)(x),-1) dx\label{eq:lim perimetro AC u_1}\\
\int_{B\cap \{u_1< u_2\}}\phi_{K}(-\nabla u_2(x),1) dx&= \lim_{M\rightarrow +\infty}\int_{B\cap \{\tau_{M}(u_1) < \tau_{M}(u_2)\}}\phi_{K}(-\nabla \tau_{M}(u_2)(x),1) dx\label{eq:lim perimetro AC u_2}
\end{align}
\begin{align}\label{eq:lim perimetro Cantor u_1}
&|(D^cu_1,0)|_K(B\cap \{\tilde{u_1} < \tilde{u_2}\})=\\ &\lim_{M\rightarrow +\infty}\int_{B\cap \{\widetilde{\tau_{M}(u_1)} < \widetilde{\tau_{M}(u_2)}\}}\phi_{K}\left(\frac{dD^c\tau_{M}(u_1)}{d|D^c\tau_{M}(u_1)|}(x),0\right) d|D^c\tau_{M}(u_1)|(x)\nonumber
\end{align}
\begin{align}\label{eq:lim perimetro Cantor u_2}
&|(-D^cu_2,0)|_K(B\cap \{\tilde{u_1} < \tilde{u_2}\})=\\ &\lim_{M\rightarrow +\infty}\int_{B\cap \{\widetilde{\tau_{M}(u_1)} < \widetilde{\tau_{M}(u_2)}\}}\phi_{K}\left(-\frac{dD^c\tau_{M}(u_2)}{d|D^c\tau_{M}(u_2)|}(x),0\right) d|D^c\tau_{M}(u_2)|(x)\nonumber
\end{align}
\begin{small}
\begin{align}\label{eq:lim perimetro Jump u_1}
&\int_{B\cap J_{u_1}}\phi_{K}\left(\nu_{u_1}(z),0  \right)\left(\min(u_1\Upp(z),u_2\Low(z))-u_1\Low(z)  \right)d\mathcal{H}^{n-2}(z)=\\ &\lim_{M\rightarrow +\infty}
\int_{B\cap J_{\tau_{M}(u_1)}}\phi_{K}\left(\frac{dD^j\tau_{M}(u_1)}{d|D^j\tau_{M}(u_1)|}(z),0  \right) \left(\min(\tau_{M}(u_1)\Upp(z),\tau_{M}(u_2)\Low(z))-\tau_{M}(u_1)\Low(z)  \right)d\mathcal{H}^{n-2}(z)\nonumber
\end{align} 
\begin{align}\label{eq:lim perimetro Jump u_2}
&\int_{B\cap J_{u_2}}\phi_{K}\left(-\nu_{u_2}(z),0  \right)\left(u_2\Upp(z)-\max( u_2\Low(z),u_1\Upp(z))  \right)d\mathcal{H}^{n-2}(z)=\\ &\lim_{M\rightarrow +\infty}
\int_{B\cap J_{\tau_{M}(u_2)}}\phi_{K}\left(-\frac{dD^j\tau_{M}(u_2)}{d|D^j\tau_{M}(u_2)|}(z),0  \right)\left(\tau_{M}(u_2)\Upp(z)-\max( \tau_{M}(u_2)\Low(z),\tau_{M}(u_1)\Upp(z))  \right)d\mathcal{H}^{n-2}(z)\nonumber.
\end{align} 
\end{small}
Observe that by \cite[Theorem 3.99]{AFP} with $f=\tau_M$ we have for $i=1,2$
\begin{align}\label{eq: chain rule AFP 3.99}
D\left(\tau_M(u_i)  \right)= 1_{\{|u_i|<M\}}\nabla u_i\,\mathcal{L}^{n-1} + \left(\tau_M(u_i\Upp)-\tau_M(u_i\Low) \right)\nu_{u_i}\,\mathcal{H}^{n-2}\mres S_{u_i} + 1_{\{|\tilde{u_i}|<M\}}\,D^cu_i 
\end{align}
We divide the proof in few steps.\\
\textbf{Step 1 (Jump part)} By relations (\ref{eq: alphas})-(\ref{eq: gammas}) and relation (\ref{eq: chain rule AFP 3.99}) we get that $\{J_{\tau_{M}(u_i)}\}_{M>0}$ is a monotone increasing family of sets whose union is $J_{u_i}$, $i=1,2$. Moreover, observing that 
\begin{align*}
\min\left(\tau_M(s);\tau_M(t)  \right)=\tau_M\left(\min(s;t)  \right) \qquad\forall\, s,t\in\R\\
\max\left(\tau_M(s);\tau_M(t)  \right)=\tau_M\left(\max(s;t)  \right)\qquad\forall\, s,t\in\R
\end{align*}
and taking into account relation (\ref{eq: gammas}) we deduce that both 
\begin{align*}
\left(\min(\tau_{M}(u_1)\Upp(z),\tau_{M}(u_2)\Low(z))-\tau_{M}(u_1)\Low(z)  \right)&_{M>0},\\
\left(\tau_{M}(u_2)\Upp(z)-\max( \tau_{M}(u_2)\Low(z),\tau_{M}(u_1)\Upp(z))  \right)&_{M>0}
\end{align*}
are increasing family of functions. Thus, the proof of (\ref{eq:lim perimetro Jump u_1}) and (\ref{eq:lim perimetro Jump u_2}) is completed.\\
\textbf{Step 2 (Cantor part)} Firstly, let us notice that by definition of \emph{approximate average} (see Section 2) and relation (\ref{eq: alphas})
\begin{align*}
\left\{\widetilde{\tau_M(u_1)}<\widetilde{\tau_M(u_2)}  \right\}=\left\{\tau_M(u_2\Upp)-\tau_M(u_1\Upp)>0  \right\}\cup \left\{\tau_M(u_2\Low)-\tau_M(u_1\Low)>0  \right\}.
\end{align*}
Thus, by relation (\ref{eq: deltas}) we deduce that $\{\widetilde{\tau_M(u_1)}<\widetilde{\tau_M(u_2)}  \}_{M>0}$ is a monotone increasing family of sets whose union is  $\{\widetilde{u_1}<\widetilde{u_2}  \}$. Let us call $A_M=\{\widetilde{\tau_M(u_1)}< \widetilde{\tau_M(u_2)}  \}$ and $A=\{\widetilde{u_1}< \widetilde{u_2}  \}$. By relation (\ref{def: Cantor part GBV}) and by the monotonicity of the sets $\{A_M\}_{M>0}$ we have that
\begin{align}\label{eq: Canton mix}
\lim_{M\rightarrow +\infty}|D^cu_i|\left(B\cap \{A_M \} \right)= |D^cu_i|(B\cap A) 
=\lim_{M\rightarrow +\infty}|D^c\tau_Mu_i|(B\cap A).
\end{align}
Again by the monotonicity of the family of sets $\{A_M\}_{M>0}$ and by (\ref{eq: chain rule AFP 3.99}) we have  
\begin{align*}
|D^cu_i|(A_{M})\leq |D^c\tau_{M}u_i|(A_{M})\leq |D^c\tau_{M}u_i|(A).
\end{align*}
Thus, taking the limit for $M\rightarrow +\infty$ in the above relation we obtain
\begin{align*}
|D^cu_i|(A)\leq \liminf_{M\rightarrow \infty} |D^c\tau_{M}u_i|(A_{M})\leq \limsup_{M\rightarrow \infty} |D^c\tau_{M}u_i|(A_{M})\leq |D^cu_i|(A),
\end{align*}
proving that 
\begin{align*}
\lim_{M\rightarrow +\infty} |D^c\tau_{M}u_i|(A_{M})=|D^cu_i|(A).
\end{align*}
Analogously, having in mind Remark \ref{rem: GBV anisotropic cantor part} we get that
\begin{align*}
&|(D^cu_1,0)|_K(B\cap A)= \lim_{M\rightarrow +\infty}|(D^c\tau_Mu_1,0)|_K(B\cap \{A_M\}),\\
&|(-D^cu_2,0)|_K(B\cap A)= \lim_{M\rightarrow +\infty}|(-D^c\tau_Mu_2,0)|_K(B\cap \{A_M\}).
\end{align*}
This concludes the proof for both (\ref{eq:lim perimetro Cantor u_1}) and (\ref{eq:lim perimetro Cantor u_2}).\\ \\
\textbf{Step 3 (Absolutely Continuos part)}  By (\ref{eq: chain rule AFP 3.99}) we get 
\begin{align*}
\int_{B\cap \{\tau_{M}(u_1) < \tau_{M}(u_2)\}}\phi_{K}(\nabla \tau_{M}(u_1)(x),-1) dx&=\int_{B\cap \{\tau_{M}(u_1)< \tau_{M}(u_2)\}\cap \{|u_1|<M \}}\phi_{K}(\nabla u_1(x),-1)\, dx \\&+\int_{B\cap \{\tau_{M}(u_1) < \tau_{M}(u_2)\}\cap\{|u_1|\geq M \}}\phi_{K}(0,-1) \,dx\\&=I_1^M+I_2^M.
\end{align*}
Notice that
\begin{align*}
|I_2^M|&=\phi_K(0,-1)\mathcal{L}^{n-1}\left( B\cap \{\tau_{M}(u_1) < \tau_{M}(u_2)\}\cap \{|u_1| \geq M \}  \right)\\&\leq \phi_K(0,-1)\mathcal{L}^{n-1}\left( B\cap \{|u_1|\geq M \}  \right).
\end{align*}
By the fact that $\{|u_1|\geq M \}_{M>0}$ is a decreasing family of sets whose intersection is $\{|u_1|=+\infty \}$ we deduce that 
\begin{align*}
\lim_{M\rightarrow \infty}|I_2^M|=0.
\end{align*}
Since both $\{ |u|<M \}_{M>0}$ and  $\{\tau_M(u_1)<\tau_M(u_2)  \}_{M>0}$ are increasing family of sets, we apply the monotone convergence theorem to get that 
\begin{align*}
\lim_{M\rightarrow \infty}I_1^M= \int_{B\cap \{u_1 < u_2\}}\phi_{K}(\nabla u_1(x),-1)\, dx.  
\end{align*} 
An analogous argument can be used for relation (\ref{eq:lim perimetro AC u_2}) and so this concludes the proof for both (\ref{eq:lim perimetro AC u_1}) and (\ref{eq:lim perimetro AC u_2}).\\
\\
\textbf{Step 4 (Perimeter functional part)}  Lastly, let us consider the family of sets $E_{M_h}=E\cap \{|x_n|<M_h  \}$ where the sequence of real numbers $\{M_h\}_{h\in\mathbb{N}}$ has been chosen s.t. 
\begin{align}\label{eq:3.23 Filippo page 24}
\lim_{h\rightarrow +\infty}\mathcal{H}^{n-1}\left(E^{(1)} \cap \{|\q x|=M_h  \}  \right)=0, \quad \mathcal{H}^{n-1}\left(\partial^e E\cap \{|\q x|=M_h  \}  \right)=0 \quad \forall h\in \mathbb{N}.
\end{align} 
Observe that the the existence of such a sequence $\{M_h\}_{h\in\mathbb{N}}$ is guaranteed by the fact that $|E|<\infty$ and by the fact that $\h^{n-1}\mres\partial^e E$ is a Radon measure. Thanks to the above two relations and \cite[Theorem 16.3]{MaggiBOOK} we have that
\begin{align*}
P_{K}\left(E_{M_h}; B\times \R  \right)&= \int_{\partial^* E_{M_h}\cap (B\times \R)}\phi_{K}(\nu^{E_{M_h}}(x))d\mathcal{H}^{n-1}(x)\\
&=\int_{\partial^* E_{M_h}\cap (B\times \R)\cap \{|\q x|<M_h  \}}\phi_{K}(\nu^{E_{M_h}}(x))d\mathcal{H}^{n-1}(x)\\&+ \int_{E^{(1)}\cap \{ |\q x|=M_h \}\cap (B\times
\R)}\phi_{K}(\nu^{E_{M_h}}(x))d\mathcal{H}^{n-1}(x).
\end{align*}
Observing that,
\begin{align*}
\int_{E^{(1)}\cap \{ |\q x|=M_h \}\cap (B\times
\R)}\phi_{K}(\nu^{E_{M_h}}(x))d\mathcal{H}^{n-1}(x)\leq C\mathcal{H}^{n-1}(E^{(1)}\cap \{ |\q x|=M_h \}),
\end{align*}
and considering the first relation in (\ref{eq:3.23 Filippo page 24}) we finally get
\begin{align*}
\lim_{h\rightarrow +\infty}\int_{\partial^* E_{M_h}\cap (B\times \R)\cap \{|\q x|<M_h  \}}\phi_{K}(\nu^{E_{M_h}}(x))d\mathcal{H}^{n-1}(x)= P_{K}(E;B\times\R).
\end{align*}
This concludes the proof.
\end{proof}
 
\begin{figure}[!htb]
\centering
\def\svgwidth{12cm}
\input{saltavsaltabprova.pdf_tex}
\caption{}
\label{fig:jumps}
\end{figure}
\noindent
Before stating the next result, we recall that given a Borel function $f:\R^{n-1}\to \R$, we indicate with $\nu_{f}(x)$ the approximate jump direction of $f$ at $x$ (see Section \ref{preliminaries}).
\begin{lemma}\label{lem:2.1 page 102 mie note}
If $v\in (BV\cap L^{\infty})(\R^{n-1};[0,\infty)) $, $b\in GBV(\R^{n-1})$ and we set $u_1=b-(v/2)\in GBV(\R^{n-1})$, $u_2 = b+ (v/2)\in GBV(\R^{n-1})$, then for $\mathcal{H}^{n-2}$-a.e. $x\in J_v\cap J_b$ we have
\begin{align}
\textit{if }x\in \left\{[b]< \left[ \frac{v}{2} \right]:\,\nu_b=\nu_v\right\} \cup \left\{ \nu_b=-\nu_v \right\}\quad\textit{then }\quad\frac{dD^ju_1}{d|D^ju_1|}(x)=-\nu_v(x)\label{eq:1 page 102}\\
\textit{if }x\in \left\{[b]> \left[ \frac{v}{2} \right]:\, \nu_b=\nu_v \right\}\quad\textit{then }\quad\frac{dD^ju_1}{d|D^ju_1|}(x)=+\nu_v(x) \label{eq: 2page 102}\\
\textit{if }x\in \left\{[b]< \left[ \frac{v}{2} \right]:\, \nu_b=-\nu_v  \right\}\cup \left\{\nu_b=\nu_v  \right\}\quad\textit{then }\quad\frac{dD^ju_2}{d|D^ju_2|}(x)=+\nu_v(x) \label{eq: 3page 102}\\
\textit{if }x\in \left\{[b]> \left[ \frac{v}{2} \right]:\, \nu_b=-\nu_v  \right\}\quad\textit{then }\quad\frac{dD^ju_2}{d|D^ju_2|}(x)=-\nu_v(x). \label{eq: 4page 102}
\end{align}
Moreover,
\begin{align}
\textit{if }x\in\left\{[b]=\frac{1}{2}[v]:\,\nu_b=\nu_v \right\}\quad\textit{then }\quad x\notin J_{u_1}\label{eq: J*1}\\
\textit{if }x\in\left\{[b]=\frac{1}{2}[v]:\,\nu_b=-\nu_v \right\}\quad\textit{then }\quad x\notin J_{u_2}.\label{eq: J*2}
\end{align}
\end{lemma}
\begin{proof}
Firstly, let us notice that thanks to \cite[Proposition 10.5]{MaggiBOOK} we already know that for $\mathcal{H}^{n-2}$-a.e. $x\in J_v\cap J_b$ either we have 
\begin{align*}
\nu_v(x)=\nu_b(x)\qquad\textit{or} \qquad\nu_v(x)=-\nu_b(x).
\end{align*}
Let us start by proving relation (\ref{eq:1 page 102}). In particular, using the definition of upper and lower limits, we want to prove that when $x\in \left\{[b]< \left[ \frac{v}{2} \right]:\,\nu_b=\nu_v\right\}$ (see figure \ref{fig:jumps} \textbf{C}) then 
\begin{align}\label{eq: come scomporre sopra e sotto di u}
u\Upp_1(x) = -\left(\frac{v}{2}\right)\Low(x) + b\Low(x),\quad
u_1\Low(x)= -\left(\frac{v}{2}\right)\Upp(x) + b\Upp(x),\quad
\nu_{u_1}(x)=-\nu_v(x).
\end{align}
As we said, we just need to verify if the definition of jump direction for the upper and lower limit is satisfied, namely if for every $\epsilon > 0$ we have that
\begin{align}\label{eq:def jump direction}
\lim_{\rho\rightarrow +\infty}\frac{\mathcal{H}^{n-1}\left(\left\{y\in \R^{n-1}:\,\left|u_1(y)-\left(-\left(\frac{v}{2}\right)\Low(x)+b\Low(x)  \right)\right|>\epsilon  \right\}\cap H^+_{x,-\nu_v}\cap D_{x,\rho}  \right)}{\omega_{n-1}\rho^{n-1}}=0.
\end{align}
Let us substitute in the numerator of (\ref{eq:def jump direction}) $u_1= b-\frac{v}{2}$ and observe that by the triangular inequality we have that 
\begin{align*}
&\left\{y\in \R^{n-1}:\,\left|b(y)-\frac{v}{2}+\left(\frac{v}{2}\right)(y)\Low(x)-b\Low(x) \right|>\epsilon  \right\}\\&\subseteq \left\{y\in \R^{n-1}:\,\left|b(y)-b\Low(x)\right|+\left|\frac{v}{2}(y)-\left(\frac{v}{2}\right)\Low(x)\right|>\epsilon  \right\}:=A.
\end{align*}
Consider now the following partition of $A$,
\begin{align}
&\left\{y\in \R^{n-1}:\,|b(y)-b\Low(x)|>\frac{\epsilon}{2}  \right\}\cap A:=A_{>\epsilon}\label{eq:A>epsilon1},\\
&\left\{y\in \R^{n-1}:\,|b(y)-b\Low(x)|\leq\frac{\epsilon}{2}  \right\}\cap A:=A_{<\epsilon},\\
&\left\{y\in \R^{n-1}:\,|b(y)-b\Low(x)|=\frac{\epsilon}{2}  \right\}\cap A:=A_{=\epsilon}.
\end{align}
So, using the above partition we can estimate the quantity in the limit relation (\ref{eq:def jump direction}) as follows
\begin{align}\label{eq:estimate jump direction}
&\frac{\mathcal{H}^{n-1}\left(\left\{y\in \R^{n-1}:\,\left|u_1(y)-\left(-\left(\frac{v}{2}\right)\Low(x)+b\Low(x)  \right)\right|>\epsilon  \right\}\cap H^+_{x,-\nu_v}\cap D_{x,\rho}  \right)}{\omega_{n-1}\rho^{n-1}}\nonumber\\&\leq
\frac{\mathcal{H}^{n-1}\left(A\cap H^+_{x,-\nu_v}\cap D_{x,\rho}  \right)}{\omega_{n-1}\rho^{n-1}}\leq \frac{\mathcal{H}^{n-1}\left(A_{>\epsilon}\cap H^+_{x,-\nu_v}\cap D_{x,\rho}  \right)}{\omega_{n-1}\rho^{n-1}}\\&+\frac{\mathcal{H}^{n-1}\left(A_{<\epsilon}\cap H^+_{x,-\nu_v}\cap D_{x,\rho}  \right)}{\omega_{n-1}\rho^{n-1}}+\frac{\mathcal{H}^{n-1}\left(A_{=\epsilon}\cap H^+_{x,-\nu_v}\cap D_{x,\rho}  \right)}{\omega_{n-1}\rho^{n-1}}\nonumber.
\end{align}
By relation (\ref{eq:A>epsilon1}) we have that  
\begin{align*}
A_{>\epsilon}\subseteq \left\{y\in\R^{n-1}:\,|b(y)-b\Low(x)|>\frac{\epsilon}{2}  \right\}.
\end{align*}
Thus,
\begin{align}\label{eq: A magg epsilon}
&\lim_{\rho\rightarrow +\infty}\frac{\mathcal{H}^{n-1}\left(A_{>\epsilon}\cap H^+_{x,-\nu_v}\cap D_{x,\rho}  \right)}{\omega_{n-1}\rho^{n-1}}\nonumber\\&\leq \lim_{\rho\rightarrow +\infty} \frac{\mathcal{H}^{n-1}\left(\left\{y\in \R^{n-1}:\,|b(y)-b\Low(x)|>\frac{\epsilon}{2}  \right\}\cap H^+_{x,-\nu_v}\cap D_{x,\rho} \right)}{\omega_{n-1}\rho^{n-1}}=0, 
\end{align}
where the latter equality holds true by definition of $b\Low(x)$ having in mind that $\nu_b=\nu_v$ by assumption.
Concerning $A_{<\epsilon}$ we have that 
\begin{align*}
A_{<\epsilon}&=\left\{y\in\R^{n-1}:\,\left|\frac{v}{2}(y)-\left(\frac{v}{2} \right)\Low(x)  \right|>\epsilon- |b(y)-b\Low(x)|\geq \frac{\epsilon}{2}  \right\}\\
&\subseteq \left\{y\in \R^{n-1}:\, \left|\frac{v}{2}(y)-\left(\frac{v}{2}  \right)\Low(x)  \right|>\frac{\epsilon}{2}  \right\}.
\end{align*}
Thus,
\begin{align}\label{eq: A less epsilon}
&\lim_{\rho\rightarrow +\infty}\frac{\mathcal{H}^{n-1}\left(A_{<\epsilon}\cap H^+_{x,-\nu_v}\cap D_{x,\rho}  \right)}{\omega_{n-1}\rho^{n-1}}\nonumber\\&\leq \lim_{\rho\rightarrow +\infty} \frac{\mathcal{H}^{n-1}\left(\left\{y\in \R^{n-1}:\, \left|\frac{v}{2}(y)-\left(\frac{v}{2}  \right)\Low(x)  \right|>\frac{\epsilon}{2}  \right\}\cap H^+_{x,-\nu_v}\cap D_{x,\rho} \right)}{\omega_{n-1}\rho^{n-1}}=0.
\end{align}
Thanks to the estimate (\ref{eq:estimate jump direction}), putting together (\ref{eq: A magg epsilon}) and (\ref{eq: A less epsilon}) we get that (\ref{eq:def jump direction}) holds true for every $\epsilon>0$. To conclude we have to prove estimate (\ref{eq:def jump direction}) for $u_1\Low(x)$ namely we have to prove that 
\begin{align*}
\lim_{\rho\rightarrow +\infty}\frac{\mathcal{H}^{n-1}\left(\left\{y\in \R^{n-1}:\,\left|u_1(y)-\left(-\left(\frac{v}{2}\right)\Upp(x)+b\Upp(x)  \right)\right|>\epsilon  \right\}\cap H^-_{x,-\nu_v}\cap D_{x,\rho}  \right)}{\omega_{n-1}\rho^{n-1}}=0\quad\forall\,\epsilon>0.
\end{align*}
In order  to prove that, just use the same argument used for (\ref{eq:def jump direction}), noticing that $H^-_{x,-\nu_{v}}=H^+_{x,\nu_{v}}=H^+_{x,\nu_{b}}$. To prove the remaining statements (\ref{eq: 2page 102})-(\ref{eq: 4page 102}), it is sufficient to consider the same argument adopted for (\ref{eq:def jump direction}), considering in each case the right function either $\frac{v}{2}$ or $b$ with which construct the partition $A_{>\epsilon}$ and $A_{<\epsilon}$.\\
Let us now prove relation (\ref{eq: J*1}). Let $x\in \{[b]=\frac{1}{2} [v]:\, \nu_b=\nu_v  \}$ and let us consider the functions $b_k, u_{1,k}\in GBV(\R^{n-1})$, $k\in \mathbb{N}$ defined as
\begin{align*}
b_k(z)=
\begin{cases}
 b(z), & \mbox{if }z\in H^-_{x,\nu_b(x)}\\ \\ b(z)-\frac{1}{k}[b](x), & \mbox{if }z\in H^+_{x,\nu_b(x)}.
\end{cases}\quad
u_{1,k}(z)=
\begin{cases}
 u_1(z), & \mbox{if }z\in H^-_{x,\nu_b(x)}\\ \\ u_1(z)-\frac{1}{k}[b](x), & \mbox{if }z\in H^+_{x,\nu_b(x)}.
\end{cases}
\end{align*}
Let us note that $u_{1,k}=b_k-\frac{1}{2} v$. Moreover, note that, $b_k\Low(x)=b\Low(x)$, $b_k\Upp(x)=b\Upp(x)-\frac{1}{k}[b](x)$ and so $[b_k](x)=[b](x)-\frac{1}{k}[b](x)$. In particular, we have that  $x\in \{[b_k]<1/2[v]:\,\nu_b=\nu_v  \}$. Thus, by relations (\ref{eq:1 page 102}) and (\ref{eq: come scomporre sopra e sotto di u}) applied to $u_{1,k}$ we get that
\begin{align}
u\Upp_{1,k}(x) &= -\frac{1}{2}v\Low(x) + b_k\Low(x)=-\frac{1}{2}v\Low(x) + b\Low(x)\label{eq: come scomporre u_k sopra} ,\\
u_{1,k}\Low(x)&= -\frac{1}{2}v\Upp(x) + b_k\Upp(x)=-\frac{1}{2}v\Upp(x) + b\Upp(x)-\frac{1}{k}[b](x)\nonumber\\&=-\frac{1}{2}v\Upp(x) + b\Low(x)+\left(1-\frac{1}{k}\right)[b](x)\label{eq: come scomporre u_k sotto} 
\end{align} 
Moreover, by (\ref{def fvee}) and (\ref{def fwedge}) we have that 
\begin{align}
u_{1,k}\Upp(x)&=\inf\left\{t\in\R:\,\lim_{\rho\rightarrow 0^+}\frac{\h^{n-1}\left(\{u_{1,k}>t \}\cap D_{x,\rho} \right)}{\omega_{n-1}\rho^{n-1}}=0\right\}\label{eq: armadio1}\\
u_{1}\Upp(x)&=\inf\left\{t\in\R:\,\lim_{\rho\rightarrow 0^+}\frac{\h^{n-1}\left(\{u_{1}>t \}\cap D_{x,\rho} \right)}{\omega_{n-1}\rho^{n-1}}=0\right\}\label{eq: armadio2}\\
u_{1,k}\Low(x)&=\sup\left\{t\in\R:\,\lim_{\rho\rightarrow 0^+}\frac{\h^{n-1}\left(\{u_{1,k}<t \}\cap D_{x,\rho} \right)}{\omega_{n-1}\rho^{n-1}}=0
\right\}\label{eq: armadio3}\\
u_{1}\Low(x)&=\sup\left\{t\in\R:\,\lim_{\rho\rightarrow 0^+}\frac{\h^{n-1}\left(\{u_{1}<t \}\cap D_{x,\rho} \right)}{\omega_{n-1}\rho^{n-1}}=0
\right\}.\label{eq: armadio4}
\end{align}
Observe that the sequence $(u_{1,k})_{k\in \mathbb{N}}$ is non decreasing in $k$. Thus, we can deduce the following inclusions $\forall\,k>1$
\begin{align*}
\left\{t\in\R:\,\lim_{\rho\rightarrow 0^+}\frac{\h^{n-1}\left(\{u_{1,k}>t \}\cap D_{x,\rho} \right)}{\omega_{n-1}\rho^{n-1}}=0\right\}&\subset\left\{t\in\R:\,\lim_{\rho\rightarrow 0^+}\frac{\h^{n-1}\left(\{u_{1}>t \}\cap D_{x,\rho} \right)}{\omega_{n-1}\rho^{n-1}}=0\right\} \\
\left\{t\in\R:\,\lim_{\rho\rightarrow 0^+}\frac{\h^{n-1}\left(\{u_{1,k}<t \}\cap D_{x,\rho} \right)}{\omega_{n-1}\rho^{n-1}}=0
\right\}&\subset\left\{t\in\R:\,\lim_{\rho\rightarrow 0^+}\frac{\h^{n-1}\left(\{u_{1}<t \}\cap D_{x,\rho} \right)}{\omega_{n-1}\rho^{n-1}}=0
\right\}.
\end{align*} 
Thanks to the above inclusions, having in mind definitions (\ref{eq: armadio1})-(\ref{eq: armadio4}) together with relations (\ref{eq: come scomporre u_k sopra}), (\ref{eq: come scomporre u_k sotto}) we get
\begin{align*}
-\frac{1}{2}v\Upp(x)+b\Low(x)+\left(1-\frac{1}{k}  \right)[b](x)=u_{1,k}\Low(x)\leq u_1\Low(x)\leq u_1\Upp(x)\leq u_{1,k}\Upp(x)=-\frac{1}{2}v\Low(x)+b\Low(x).
\end{align*}
Since $-\frac{1}{2}v\Upp(x)=-\frac{1}{2}v\Low(x)-\frac{1}{2}[v](x)$, passing through the limit as $k\rightarrow +\infty$ in the above relation, we conclude that $u_1\Low(x)=u_1\Upp(x)$ and so $x\notin J_{u_1}$. This concludes the proof of (\ref{eq: J*1}). Using a similar argument as the one used for (\ref{eq: J*1}), we can prove (\ref{eq: J*2}).   
\end{proof}

\begin{remark}\label{rem: Jv/Jb e Jb/Jv}
The cases where $[b](x)=0$ i.e. $x\in J_v\setminus J_b$  can be seen as degenerate situations in Lemma \ref{lem:2.1 page 102 mie note} considering in those characterizations $[b]=0$. A similar argument can be applied to show that for $\mathcal{H}^{n-2}$-a.e. $x\in J_b\setminus J_v$ we have $\nu_{u_i}=\nu_b $, $i=1,2$.
\end{remark}

\begin{remark}\label{rem: A B1 B2 B3 B4 B5 B6 C}
Let us introduce the following compact notation.
\begin{align*}
&\textbf{A}=J_v\setminus J_b,\\
&\textbf{B}_1=\left\{J_v\cap J_b:\,\nu_v=\nu_b,\, [b]<\frac{1}{2}[v] \right\},\quad\textbf{B}_2=\left\{J_v\cap J_b:\,\nu_v=\nu_b,\, [b]=\frac{1}{2}[v] \right\},\\&\textbf{B}_3=\left\{J_v\cap J_b:\,\nu_v=\nu_b,\, [b]>\frac{1}{2}[v] \right\},\\
&\textbf{B}_4=\left\{J_v\cap J_b:\,\nu_v=-\nu_b,\, [b]<\frac{1}{2}[v] \right\}, \quad\textbf{B}_5=\left\{J_v\cap J_b:\,\nu_v=-\nu_b,\, [b]=\frac{1}{2}[v] \right\},
\\&\textbf{B}_6=\left\{J_v\cap J_b:\,\nu_v=-\nu_b,\, [b]>\frac{1}{2}[v] \right\},
\\&\textbf{C}=J_b\setminus J_v.
\end{align*}
Note that we have
\begin{align}\label{eq: come spezzare Jv U Jb}
J_v\cup J_b= \textbf{A}\cup \left(\bigcup_{i=1}^6\textbf{B}_i\right) \cup \textbf{C}.
\end{align}
Moreover, following the argument explained in the proof of Lemma \ref{lem:2.1 page 102 mie note} we can prove the following relations
\begin{align}\label{eq: x in A}
\textit{if }x\in \textbf{A}\quad\textit{then }\quad &u_1\Upp(x)=-\frac{1}{2}v\Low(x)+\tilde{b}(x);\,u_1\Low(x)=-\frac{1}{2}v\Upp(x)+\tilde{b}(x)\\&u_2\Upp(x)=\frac{1}{2}v\Upp(x)+\tilde{b}(x);\,u_2\Low(x)=\frac{1}{2}v\Low(x)+\tilde{b}(x).
\end{align}
\begin{align}\label{eq: x in B1 U B2}
\textit{if }x\in \textbf{B}_1\cup\textbf{B}_2 \quad\textit{then }\quad &u_1\Upp(x)=-\frac{1}{2}v\Low(x)+b\Low(x);\,u_1\Low(x)=-\frac{1}{2}v\Upp(x)+b\Upp(x)\\&u_2\Upp(x)=\frac{1}{2}v\Upp(x)+b\Upp(x);\,u_2\Low(x)=\frac{1}{2}v\Low(x)+b\Low(x).
\end{align}
\begin{align}\label{eq: x in B3}
\textit{if }x\in \textbf{B}_3\quad\textit{then }\quad &u_1\Upp(x)=-\frac{1}{2}v\Upp(x)+b\Upp(x);\,u_1\Low(x)=-\frac{1}{2}v\Low(x)+b\Low(x)\\&u_2\Upp(x)=\frac{1}{2}v\Upp(x)+b\Upp(x);\,u_2\Low(x)=\frac{1}{2}v\Low(x)+b\Low(x).
\end{align}
\begin{align}\label{eq: x in B4 U B5}
\textit{if }x\in \textbf{B}_4\cup\textbf{B}_5\quad\textit{then }\quad &u_1\Upp(x)=-\frac{1}{2}v\Low(x)+b\Upp(x);\,u_1\Low(x)=-\frac{1}{2}v\Upp(x)+b\Low(x)\\&u_2\Upp(x)=\frac{1}{2}v\Upp(x)+b\Low(x);\,u_2\Low(x)=\frac{1}{2}v\Low(x)+b\Upp(x).
\end{align}
\begin{align}\label{eq: x in B6}
\textit{if }x\in \textbf{B}_6\quad\textit{then }\quad &u_1\Upp(x)=-\frac{1}{2}v\Low(x)+b\Upp(x);\,u_1\Low(x)=-\frac{1}{2}v\Upp(x)+b\Low(x)\\&u_2\Upp(x)=\frac{1}{2}v\Low(x)+b\Upp(x);\,u_2\Low(x)=\frac{1}{2}v\Upp(x)+b\Low(x).
\end{align}
\begin{align}\label{eq: x in C}
\textit{if }x\in \textbf{C}\quad\textit{then }\quad &u_1\Upp(x)=-\frac{1}{2}\tilde{v}(x)+b\Upp(x);\,u_1\Low(x)=-\frac{1}{2}\tilde{v}(x)+b\Low(x)\\&u_2\Upp(x)=\frac{1}{2}\tilde{v}(x)+b\Upp(x);\,u_2\Low(x)=\frac{1}{2}\tilde{v}(x)+b\Low(x).
\end{align}
\end{remark}

\begin{corollary}\label{cor: formula perimetro con b v}
If $v\in (BV\cap L^{\infty})(\R^{n-1};[0,\infty)) $, $b\in GBV(\R^{n-1})$ and \
\begin{align}
W=W[v,b]=\left\{x\in \R^n:\, |\q x - b(\p x)|<\frac{v(\p x)}{2}  \right\},
\end{align}
then $u_1=b-(v/2)\in GBV(\R^{n-1})$, $u_2 = b+ (v/2)\in GBV(\R^{n-1})$, $W$ is a set of locally finite perimeter with finite volume and for every Borel set $B\subset \R^{n-1}$ we have
\begin{align}
\!\!\!\!\!P_{K}(W;B\times \R)&=\int_{B\cap \{v>0 \}}\phi_{K}\left(\nabla\left(b-\frac{v}{2}\right),-1\right)+ \phi_{K}\left(-\nabla\left(b+\frac{v}{2}\right),1\right)d\mathcal{H}^{n-1}\label{eq: abs cont part formula v e b}\\
&+\int_{B\cap J_v}\min\left(v\Upp,\left(\left[\frac{v}{2}\right]+[b]+\max\left(\left[\frac{v}{2}\right]- [b],0  \right)   \right)   \right)\phi_{K}(-\nu_v,0)\, d\mathcal{H}^{n-2}\label{eq: jump part formula v e b1}\\
&+\int_{B\cap J_v}\min\left(v\Low,\max\left( 0,[b]-\left[\frac{v}{2}\right] \right)\right)\phi_{K}(\nu_v,0)\, d\mathcal{H}^{n-2}\label{eq: jump part formula v e b2}\\
&+\int_{B\cap (J_b\setminus J_v )}\min\left([b],\tilde{v}  \right)\left(\phi_{K}(-\nu_b,0)+\phi_{K^s}(\nu_b,0)  \right) d\mathcal{H}^{n-2}\label{eq: jump part formula v e b3}\\
&+\left|\left(D^c\left(b-\frac{v}{2} \right),0\right)   \right|_K(B\cap \{\tilde{v}>0  \})
\label{eq: cantor part formula v e b1}\\
&+\left|\left(-D^c\left(b+\frac{v}{2} \right),0\right)   \right|_K(B\cap \{\tilde{v}>0  \}).\label{eq: cantor part formula v e b2}
\end{align}
\end{corollary}

\begin{proof}
The absolutely continuous part and the Cantor parts of the formula, namely relations (\ref{eq: abs cont part formula v e b}), (\ref{eq: cantor part formula v e b1}) and (\ref{eq: cantor part formula v e b2}) are obtained directly by substitution of $u_1=b-\frac{1}{2}v$ and $u_2=b+\frac{1}{2}v$ in the formula (\ref{eq:formula perimetro GBV}). To prove the jump parts of the formula i.e. (\ref{eq: jump part formula v e b1}), (\ref{eq: jump part formula v e b2}) and (\ref{eq: jump part formula v e b3}) we have first to notice that (see (\ref{eq: come spezzare Jv U Jb})) 
$$J_{u_1}\cup J_{u_2}=J_v\cup J_{b}=J_v\setminus J_b \cup (J_v\cap J_b) \cup J_b\setminus J_v = \textbf{A}\cup \left(\bigcup_{i=1}^6\textbf{B}_i\right) \cup \textbf{C}. $$ 
Thanks to this relation, we can rewrite the second and third line of the formula (\ref{eq:formula perimetro GBV}) as
\begin{align*}
&\int_{B\cap (J_{u_1}\cup J_{u_2})}\phi_{K}\left(\nu_{u_1}(z),0  \right)\left(\min(u_1\Upp(z),u_2\Low(z))-u_1\Low(z)  \right) \nonumber\\
&+\phi_{K}\left(-\nu_{u_2}(z),0  \right)\left(u_2\Upp(z)-\max( u_2\Low(z),u_1\Upp(z))  \right)d\mathcal{H}^{n-2}(z)\\= &\int_{B\cap (J_{u_1}\cup J_{u_2})} I_1(z)+I_2(z)d\mathcal{H}^{n-2}(z)= \int_{\textbf{A}}I_1(z)+I_2(z)d\mathcal{H}^{n-2}(z)\\&+\sum_{i=1}^6\int_{\textbf{B}_i}I_1(z)+I_2(z)d\mathcal{H}^{n-2}(z)+ \int_{\textbf{C}}I_1(z)+I_2(z)d\mathcal{H}^{n-2}(z).
\end{align*}
Using then Lemma \ref{lem:2.1 page 102 mie note}, Remark \ref{rem: Jv/Jb e Jb/Jv} and Remark \ref{rem: A B1 B2 B3 B4 B5 B6 C} we deduce relations (\ref{eq: jump part formula v e b1}), (\ref{eq: jump part formula v e b2}) and (\ref{eq: jump part formula v e b3}). This concludes the proof.
\end{proof}
\begin{corollary}\label{cor: formula per F[v]}
If $v$ is as in (\ref{due tilde}), then
\begin{align*}
P_{K}(F[v];G\times \R)&= \int_{G\cap\{v>0\}}\phi_{K}\left(-\frac{1}{2}\nabla\left(v\right),-1\right)d\mathcal{H}^{n-1}+\int_{G\cap\{v>0\}}\phi_{K}\left(-\frac{1}{2}\nabla\left(v\right),1\right)d\mathcal{H}^{n-1}\\
&+\int_{G\cap J_v}[v]\phi_{K^s}(-\nu_v,0) d\mathcal{H}^{n-2}+2\left|\left(-\frac{1}{2}
D^cv,0\right)   \right|_K(G).
\end{align*}
\end{corollary}
\begin{proof}
The proof follows by applying Corollary \ref{thm:perimetro GBV} with $ u_1=-\frac{1}{2}v$ and $u_2=\frac{1}{2}v$, and by recalling that by Lemma  \ref{lem: cantor total variation of v on v=0 is null}, $\left|\left(-\frac{1}{2}
D^cv,0\right)   \right|_K(G)=\left|\left(-\frac{1}{2}
D^cv,0\right)   \right|_K(G\cap\{\tilde{v}>0\} )$.
\end{proof}

\section{Characterization of equality cases for the anisotropic perimeter inequality}\label{section 6}
%\begin{theorem}\label{thm:2.2 pag 117}
%Let $v \in BV(\R^{n-1};[0,+\infty))$ with $\mathcal{H}^{n-1}(\{v > 0  \})<\infty$ and let E be a $v$-distributed set of finite perimeter. Then, $E\in \mathcal{M}_{K^s}(v)$ if and only if
%\begin{align}
%\textit{$E_x$ is $\mathcal{H}^1$-equivalent to a segment,}\quad &\textit{for $\mathcal{H}^{n-1}$-a.e. $x\in\R^{n-1}$}, \label{eq:1.15 pag 7 Filippo} \\
%\left(\nabla b_E(x),0  \right)\in \mathcal{K}\left(\left(-\frac{\nabla v}{2}(x),1\right)  \right)- \left(- \frac{\nabla v}{2}(x),1\right),\quad  &\textit{for $\mathcal{H}^{n-1}$-a.e. $x\in\R^{n-1}$},\label{eq:1.16 pag 7 Filippo}  \\
%[b_E](x)\leq \frac{[v](x)}{2},\quad &\textit{for $\mathcal{H}^{n-2}$-a.e. $x\in\{v\Low >0\}$},\label{eq:1.17 pag 7 Filippo}  \\
%g(x)\in \mathcal{K}(h(x))-h(x),\quad &\textit{for $|\mu|$-a.e. $x\in G\cap \{ |b_{\delta}|<M \}^\noindent{(1)}\cap \{ v>\delta \}^{(1)}$}\nonumber \\
%&\textit{and for every bounded Borel set $G\subset \R^{n-1}$} ,\label{eq:1.18 pag 7 Filippo}
%\end{align}
%where $\mu=h|\mu|_{K^s}$, $\nu=g|\mu|_{K^s}$ with 
%$\mu=\left(D^c\frac{v}{2},0  \right)$ and $\nu=\left( D^c\tau_Mb_{\delta},0 \right)$ for $\mathcal{H}^1$-a.e. $\delta,M >0$.
%\end{theorem}
%This implies that all inequalities in 
%relation (\ref{eq: P_K(F[v])!=P_K(E)}) must hold as equalities. In particular, by the latter of these equalities we get that $N(z)=2$ for $\mathcal{H}^{n-1}$-a.e. $z\in\R^{n-1}$ implying that $E_z$ is $\mathcal{H}^1$-equivalent to a segment for $\mathcal{H}^{n-1}$-a.e. $z\in\R^{n-1}$ that is condition (\ref{eq: 1.15 Filippo})
\noindent
This section is dedicated to the proof of Theorem \ref{thm:2.2 pag 117}. This proof is on the spirit of the proof of Theorem \ref{thm:1.2 Filippo} (see \cite[Theorem 1.9]{CCPMSteiner}). We split the proof of Theorem \ref{thm:2.2 pag 117} in the \emph{necessary part} and in the \emph{sufficient part}.
\begin{proof}[Proof of Theorem \ref{thm:2.2 pag 117}: Necessary conditions.]
Let $E\in \mathcal{M}_{K^s}(v)$. Condition (\ref{eq: 1.15 Filippo}) was already proved in \cite[Theorem 2.9]{cianchifusco2}. As a consequence , by Theorem \ref{thm: baricentro Filippo}, we have that $b_{\delta}=1_{\{v> \delta \}}b_E \in GBV(\R^{n-1})$ for every $\delta>0$ such that $\{v> \delta  \}$ is a set of finite perimeter in $\R^{n-1}$. Let us consider the same sets defined in \cite[page 1568]{CCPMSteiner} namely
\begin{align}
I&=\left\{\delta > 0:\, \{v < \delta  \}\textit{ \emph{and} }\{v> \delta  \}\textit{ \emph{are sets of finite perimeter}}  \right\},\label{eq:3.50 pag 31 Filippo}\\
J_{\delta}&=\left\{M > 0:\, \{b_{\delta} < M  \}\textit{ \emph{and} }\{b_{\delta}> -M  \}\textit{ \emph{are sets of finite perimeter}}  \right\},\label{eq:3.51 pag 31 Filippo}
\end{align}
where $J_\delta$ is defined for $\delta\in\ I$.
Let us observe that $\mathcal{H}^1((0,\infty)\setminus I)=0$ since $v\in BV(\R^{n-1})$ and that $\mathcal{H}^1((0,\infty)\setminus J_{\delta})=0$ for every $\delta \in I$, as for every $\delta \in I$ we have $b_{\delta}\in GBV(\R^{n-1})$. Let us fix $\delta, L\in I$ and $M\in J_{\delta}$ and set 
\begin{align*}
\Sigma_{\delta,L,M}= \left\{\delta < v < L  \right\}\cap \left\{|b_E|<M  \right\}= \left\{ |b_{\delta}|< M \right\}\cap \left\{\delta < v < L  \right\},
\end{align*}
so that $\Sigma_{\delta,L,M}$ is a set of finite perimeter. Since $\tau_Mb_{\delta}\in (BV\cap L^{\infty})(\R^{n-1})$, $1_{\Sigma_{\delta,L,M}}\in (BV\cap L^{\infty})(\R^{n-1})$ and $\tau_Mb_{\delta}=b_{\delta}=b_E$ on $\Sigma_{\delta,L,M}$, we set 
\begin{align*}
b_{\delta,L,M}=1_{\Sigma_{\delta,L,M}}b_E.
\end{align*}
Note that $b_{\delta,L,M}\in (BV\cap L^{\infty})(\R^{n-1}).$\\
\textbf{Step 1} In this step we are going to prove that for $\mathcal{H}^{n-1}$-a.e. $x\in\R^{n-1}$ there exists $z(x)\in \partial K^s$ such that
\begin{align}
\left\{\left(-\frac{1}{2}\nabla v(x)+t \nabla b_{\delta,L,M}(x),1  \right):\,t\in [-1,1]   \right\} \subset C^*_{K^s}(z(x)).\label{eq:3.52 pag 31 Filippo}
\end{align}
Indeed, let us set $v_{\delta,L,M}=1_{\Sigma_{\delta,L,M}}v$. Since $v_{\delta,L,M},b_{\delta,L,M}\in(BV\cap L^{\infty})(\R^{n-1})$, we can apply Corollary \ref{cor: formula perimetro con b v} to $W=W[v_{\delta,L,M},b_{\delta,L,M}]$. Moreover observe that  $W[v_{\delta,L,M},b_{\delta,L,M}]=E\cap(\Sigma_{\delta,L,M}\times \R)$ and thus \begin{align*}
\partial^e E \cap (\Sigma_{\delta,L,M}^{(1)}\times \R)= \partial^e W[v_{\delta,L,M},b_{\delta,L,M}]\cap (\Sigma_{\delta,L,M}^{(1)}\times \R),
\end{align*}
and so, for every Borel set $G\subset \Sigma_{\delta,L,M}^{(1)}\setminus (S_{v_{\delta,L,M}}\cup S_{b_{\delta,L,M}})$ we find that 
\begin{align*}
&P_{K^s}(E;G\times \R)=P_{K^s}(W[v_{\delta,L,M},b_{\delta,L,M}];G\times \R)\\
&=\int_{G}\phi_{K^s}\left(\nabla\left(b_{\delta,L,M}-\frac{v_{\delta,L,M}}{2}\right),1\right)+ \phi_{K^s}\left(-\nabla\left(b_{\delta,L,M}+\frac{v_{\delta,L,M}}{2}\right),1\right)d\mathcal{H}^{n-1}\\
&+\left|\left(D^c\left(b_{\delta,L,M}-\frac{v_{\delta,L,M}}{2}\right),0\right)   \right|_{K^s}(G)+\left|\left(-D^c\left(b_{\delta,L,M}+\frac{v_{\delta,L,M}}{2}\right),0\right)   \right|_{K^s}(G),
\end{align*}
where in the first addendum of the second line we have used Remark \ref{rem: K simmetrica e quindi pure phi_K}.
We can use Lemma \ref{lem: 2.3 FIlippo page 18} applied with $v_{\delta,L,M}=1_{\Sigma_{\delta,L,M}}v$, to find that 
\begin{align*}
\nabla v_{\delta,L,M}&=1_{\Sigma_{\delta,L,M}}\nabla v,\quad \textit{$\mathcal{H}^{n-1}$\emph{-a.e. on $\R^{n-1}$}},\\
D^cv_{\delta,L,M} &= D^cv\mres\Sigma^{(1)}_{\delta,L,M},\\
S_{v_{\delta,L,M}}\cap \Sigma_{\delta,L,M}^{(1)} &= S_v\cap \Sigma_{\delta,L,M}^{(1)}.
\end{align*}
Thus,
\begin{align*}
P_{K^s}(E;G\times \R)&=\int_{G}\phi_{K^s}\left(\nabla\left(b_{\delta,L,M}-\frac{v}{2}\right),1\right)+ \phi_{K^s}\left(-\nabla\left(b_{\delta,L,M}+\frac{v}{2}\right),1\right)d\mathcal{H}^{n-1}\\
&+\left|\left(D^c\left(b_{\delta,L,M}-\frac{v}{2}\right),0\right)   \right|_{K^s}(G)+\left|\left(-D^c\left(b_{\delta,L,M}+\frac{v}{2}\right),0\right)   \right|_{K^s}(G),
\end{align*}
for every Borel set $G\subset \Sigma_{\delta,L,M}^{(1)}\setminus (S_{v_{\delta,L,M}}\cup S_{b_{\delta,L,M}})$.
We are assuming that $E \in \mathcal{M}_{K^s}(v)$ and so for every Borel set $G\subset \R^{n-1}$ we have that $P_{K^s}(E;G\times \R)=P_{K^s}(F[v];G\times \R)$. In particular, having in mind the formula for $P_{K^s}(F[v];G\times \R)$ given by Corollary \ref{cor: formula per F[v]}, for every Borel set $G\subset \Sigma_{\delta,L,M}^{(1)}\setminus (S_{v_{\delta,L,M}}\cup S_{b_{\delta,L,M}})$ we get
%\begin{align}
%0&= \int_{G}\phi_{K^s}\left(\nabla\left(b_{\delta,L,M}-\frac{v}{2}\right),1\right)+ \phi_{K^s}\left(-\nabla\left(b_{\delta,L,M}+\frac{v}{2}\right),1\right)-2\phi_{K^s}\left(-\nabla\left(\frac{v}{2}\right),1\right)d\mathcal{H}^{n-1} \label{eq:3.57 page 32 FIlippo}\\
%&+\int_{G}\phi_{K^s}\left(\frac{dD^c\left(b_{\delta,L,M}-\frac{v}{2}\right)}{d\left|D^c\left(b_{\delta,L,M}-\frac{v}{2}\right)\right|},0  \right) d\left|D^c\left(b_{\delta,L,M}-\frac{v}{2}\right)\right| \nonumber\\
%&+\int_{G}\phi_{K^s}\left(-\frac{dD^c\left(b_{\delta,L,M}+\frac{v}{2}\right)}{d\left|D^c\left(b_{\delta,L,M}+\frac{v}{2}\right)\right|},0  \right) d\left|D^c\left(b_{\delta,L,M}+\frac{v}{2}\right)\right|\nonumber \\
%&-2\int_{G}\phi_{K^s}\left(-\frac{dD^c\left(\frac{v}{2}\right)}{d\left|D^c\left(\frac{v}{2}\right)\right|},0  \right) d\left|D^c\left(\frac{v}{2}\right)\right|.
%\end{align}
\begin{small}
\begin{align}
0&= \int_{G}\phi_{K^s}\left(\nabla\left(b_{\delta,L,M}-\frac{v}{2}\right),1\right)+ \phi_{K^s}\left(-\nabla\left(b_{\delta,L,M}+\frac{v}{2}\right),1\right)-2\phi_{K^s}\left(-\nabla\left(\frac{v}{2}\right),1\right)d\mathcal{H}^{n-1} \label{eq: (1) sub addivita}\\
&+\left|\left(D^c\left(b_{\delta,L,M}-\frac{v}{2}\right),0\right)\right|_{K^s}\!\!\!\!\!\!\!(G)+\left|\left(-D^c\left(b_{\delta,L,M}+\frac{v}{2}\right),0\right)\right|_{K^s}\!\!\!\!\!\!\!(G)-2\left|\left(-D^c\left(\frac{v}{2}\right),0\right)\right|_{K^s}\!\!\!\!\!\!\!(G)\label{eq: (2)>0 grazie al lem misure anisotrope}
\end{align}
\end{small}
Let us notice that the first line in the above relation, namely (\ref{eq: (1) sub addivita}) is greater than or equal to zero by the sub additivity of $\phi_K$. Also the second line in the above relation, namely (\ref{eq: (2)>0 grazie al lem misure anisotrope}), is greater than or equal to zero thanks to Lemma \ref{lem:1.9 pag 70} with $\mu=\left(-\frac{1}{2}D^cv,0   \right)$ and $\nu=\left(D^cb_{\delta,L,M},0   \right)$. Thus, we have that 
\begin{small}
\begin{align}
0&= \int_{G}\phi_{K^s}\left(\nabla\left(b_{\delta,L,M}-\frac{v}{2}\right),1\right)+ \phi_{K^s}\left(-\nabla\left(b_{\delta,L,M}+\frac{v}{2}\right),1\right)-2\phi_{K^s}\left(-\nabla\left(\frac{v}{2}\right),1\right)d\mathcal{H}^{n-1} \label{eq:3.57 page 32 FIlippo}\\
0&=\left|\left(D^c\left(b_{\delta,L,M}-\frac{v}{2}\right),0\right)\right|_{K^s}\!\!\!\!\!\!\!(G)+\left|\left(-D^c\left(b_{\delta,L,M}+\frac{v}{2}\right),0\right)\right|_{K^s}\!\!\!\!\!\!\!(G)-2\left|\left(-D^c\left(\frac{v}{2}\right),0\right)\right|_{K^s}\!\!\!\!\!\!\!(G).\label{eq:3.58 page 32  FIlippo}
\end{align}
\end{small}
Let us observe that the relation (\ref{eq:3.57 page 32 FIlippo}) is satisfied if and only if $\h^{n-1}$-a.e. in $G$ we have
\begin{align*}
\phi_{K^s}\left(\nabla\left(b_{\delta,L,M}-\frac{v}{2}\right)(x),1  \right) + \phi_{K^s}\left(-\nabla\left(b_{\delta,L,M}+\frac{v}{2}\right)(x),1  \right)=2\phi_{K^s}\left(-\frac{\nabla v(x)}{2},1  \right). 
\end{align*}
Thanks to Proposition \ref{prop:linearityphi} the condition above is satisfied if and only if for $\mathcal{H}^{n-1}\textit{\emph{-a.e. $x\in G$}}$, $\exists \bar{z}(x) \in \partial K^s$ s.t.
\begin{align*}
\frac{\left(\nabla\left(b_{\delta,L,M}-\frac{v}{2}\right)(x),1  \right)}{\phi_{K^s}\left(\nabla\left(b_{\delta,L,M}-\frac{v}{2}\right)(x),1  \right)},\frac{\left(-\nabla\left(b_{\delta,L,M}+\frac{v}{2}\right)(x),1  \right)}{\phi_{K^s}\left(-\nabla\left(b_{\delta,L,M}+\frac{v}{2}\right)(x),1  \right)}\in \partial \phi^*_{K^s}(\bar{z}(x)).
\end{align*}
As we observed in Remark \ref{rem: more appealing way (iii)}, and in particular using relation (\ref{eq: more appealing way (iii)}) with $y_1=\left(-\frac{1}{2}\nabla(x)+\nabla b_{\delta,L,M}(x) ,1\right)$ and $y_2=\left(-\frac{1}{2}\nabla(x)-\nabla b_{\delta,L,M}(x),1\right)$ the condition above is equivalent to say that for $\mathcal{H}^{n-1}$-a.e. $x\in G$, there exists $\bar{z}(x) \in \partial K^s$ s.t.
\begin{align}
\left\{\left(-\frac{1}{2}\nabla(x)+t\nabla b_{\delta,L,M}(x),1 \right):\,t\in[-1,1]   \right\} \subset C_K^*(\bar{z}(x)).
\end{align}
This concludes the first step.\\
\textbf{Step 2} In this step we prove that there exists a Borel measurable function $g_{\delta,L,M}:\R^{n-1}\rightarrow \R^{n-1}$ such that
$$
D^cb_{\delta,L,M}\mres \Sigma_{\delta,L,M}^{(1)}= g_{\delta,L,M}\left|\frac{1}{2}D^cv  \right|_{K^s}\mres\Sigma_{\delta,L,M}^{(1)}.
$$
We prove also an intermediate relation for (\ref{eq:1.18 pag 7 Filippo}). Indeed, let us rewrite relation (\ref{eq:3.58 page 32  FIlippo}) as 
\begin{align*}
\left|\left(-D^c v,0  \right)  \right|_{K^s}\!(G)=\left|\left(D^c \left(b_{\delta,L,M}-\frac{v}{2} \right) ,0  \right)  \right|_{K^s}\!\!\!\!\!\!\!(G) + \left|\left(-D^c \left(b_{\delta,L,M}+\frac{v}{2} \right),0  \right)  \right|_{K^s}\!\!\!\!\!\!\!(G).
\end{align*}
As already observed, by calling
\begin{align*}
\mu&=\left(-\frac{D^c v}{2},0  \right),\\
\nu&= \left(D^cb_{\delta,L,M},0  \right)
\end{align*}
the above equality can be written as
\begin{align*}
2|\mu|_{K^s}(G)= |\mu+\nu|_{K^s}(G) + |\mu-\nu |_{K^s}(G). 
\end{align*}
Observe that we are in a case of equality in Lemma \ref{lem:1.9 pag 70}. Thus, by Remark \ref{rem: case equality radon measure}, for $\left|D^c v \right|$-a.e. $x\in G$ we define
\begin{align*}
g_{\delta,L,M}(x)= \frac{dD^c b_{\delta,L,M}}{d\left|(D^c v/2,0) \right|_{K^s}}(x),\quad
h(x)= \frac{-dD^c v/2 }{d\left|(D^c v/2,0) \right|_{K^s}}(x),
\end{align*}
and we conclude that for $|D^cv|$-a.e. $x\in G$ there exists $z(x)\in \partial K$ s.t.
\begin{align}\label{eq:3.53 page 31 Filippo}
\left\{(h(x)+tg_{\delta,L,M}(x),0):\,t\in[-1,1]\right\}\subset C^*_{K}(z(x)).
\end{align}
%
%\begin{align}
%g(x) \in \mathcal{K}(h(x))-h(x) \quad \textit{\emph{ for $\left|D^c v \right|$-a.e. $x\in G$}}. 
%\end{align}
This concludes the second step.\\
\textbf{Step 3} In this step we prove (\ref{eq:1.17 pag 7 Filippo}). We fix $\delta, L\in I$ and we define $\Sigma_{\delta, L}= \{\delta < v < L   \}$, $b_{\delta, L}= 1_{\Sigma_{\delta, L}}b_E$ and $v_{\delta, L}= 1_{\Sigma_{\delta, L}}v$. Since $\Sigma_{\delta, L}$ is a set of finite perimeter, it turns out that $b_{\delta, L}\in GBV(\R^{n-1})$, while, by construction, $v_{\delta, L}\in (BV\cap L^{\infty})(\R^{n-1})$. So, we can apply the formula of Corollary \ref{cor: formula perimetro con b v} to the set $W[v_{\delta, L},b_{\delta, L}]$. In particular, if $G\subset \Sigma^{(1)}_{\delta, L}\cap (S_{v_{\delta, L}}\cup S_{b_{\delta, L}})$, then
\begin{align}
&P_{K^s}(E;G\times \R)=P_{K^s}(W[v_{\delta, L},b_{\delta, L}];G\times \R)\nonumber\\
&=\int_{G\cap J_{v}}\min\left(v\Upp,\left(\left[\frac{v}{2}\right]+[b_{\delta, L}]+\max\left(\left[\frac{v}{2}\right]- [b_{\delta, L}],0  \right)   \right)   \right)\phi_{K^s}(-\nu_{v},0) d\mathcal{H}^{n-2}\label{eq: parte di salto lungo -nu_v}\\
&+\int_{G\cap J_{v}}\min\left(v\Low,\max\left( 0,[b_{\delta, L}]-\left[\frac{v}{2}\right] \right)\right)\phi_{K^s}(\nu_{v},0) d\mathcal{H}^{n-2}\nonumber\\
&+\int_{G\cap (J_{b_{\delta, L}}\setminus J_{v} )}\min\left([b_{\delta, L}],\tilde{v}  \right)\left(\phi_{K^s}(-\nu_{b_{\delta, L}},0)+\phi_{K^s}(\nu_{b_{\delta, L}},0)  \right) d\mathcal{H}^{n-2},\nonumber
\end{align}
where we used the fact that, thanks to (\ref{eq: 2.10 Filippo pag 16})
\begin{align*}
\Sigma_{\delta, L}^{(1)}\cap S_{v_{\delta, L}}= \Sigma_{\delta, L}^{(1)}\cap S_v,\quad
v_{\delta, L}\Upp = v\Upp \quad
v\Low_{\delta, L} = v\Low,\quad 
[v_{\delta, L}]&=[v]\quad \forall x\in \Sigma_{\delta, L}^{(1)}.
\end{align*}
Let us observe that, calling $\mathcal{I}$ the argument of the integral in relation (\ref{eq: parte di salto lungo -nu_v}) i.e.
\begin{align*}
\mathcal{I}=\min\left(v\Upp,\left(\left[\frac{v}{2}\right]+[b_{\delta, L}]+\max\left(\left[\frac{v}{2}\right]- [b_{\delta, L}],0  \right)   \right)   \right)
\end{align*}
we have that
\begin{align}
\textit{if }\quad [b_{\delta,L}]=0 \quad \textit{then}\quad \mathcal{I}=[v],\label{eq: I1}\\
\textit{if }\quad [b_{\delta,L}]\leq \frac{1}{2}[v] \quad \textit{then}\quad \mathcal{I}=[v],\label{eq: I2}\\
\textit{if }\quad [b_{\delta,L}]>\frac{1}{2}[v] \quad \textit{then}\quad \mathcal{I}>[v].\label{eq: I3}
\end{align}
Recall that 
\begin{align*}
P_{K^s}(F[v];G\times \R)&=\int_{G\cap J_v}[v]\phi_{K^s}(-\nu_v,0) d\mathcal{H}^{n-2}.
\end{align*}
Thus, since $\phi_{K^s}\geq 0$, imposing that $P_{K^s}(F[v];G\times \R)= P_{K^s}(E;G\times \R)$ and having in mind relations (\ref{eq: I1})-(\ref{eq: I3}) we obtain that 
\begin{align}
\min\left([b_{\delta, L}],\tilde{v}  \right)= 0,\quad \mathcal{H}^{n-2}\textit{\emph{-a.e. in }} G\cap (S_{b_{\delta,L}}\setminus &S_v)\label{eq: As}\\
\min\left(v\Low,\max\left( 0,[b_{\delta, L}]-\left[\frac{v}{2}\right] \right)\right)=0,\quad \mathcal{H}^{n-2}\textit{\emph{-a.e. in }} G\cap  &S_v\label{eq: Bs}\\
\mathcal{I}=\min\left(v\Upp,\left(\left[\frac{v}{2}\right]+[b_{\delta, L}]+\max\left(\left[\frac{v}{2}\right]- [b_{\delta, L}],0  \right)   \right)   \right)= [v]\quad \mathcal{H}^{n-2}\textit{\emph{-a.e. in }} G\cap &S_v.\label{eq: Cs}
\end{align}
Since $\tilde{v}\geq \delta>0$ in $\Sigma_{\delta, L}^{(1)}$, from (\ref{eq: As}) it follows that $S_{b_{\delta,L}}\cap \Sigma_{\delta,L}^{(1)}\subset_{\mathcal{H}^{n-2}}S_v$. Moreover, from (\ref{eq: I1}), (\ref{eq: I2}) together with (\ref{eq: As}) and (\ref{eq: Bs}) it follows that 
\begin{align}
[b_{\delta, L}]\leq \frac{[v]}{2} \quad \mathcal{H}^{n-2}\textit{\emph{-a.e. }}x\in G\cap S_v.
\end{align}
By (\ref{eq: 2.10 Filippo pag 16}), $[b_{\delta,L}]=[b_E]$ on $\Sigma_{\delta,L}^{(1)}$. By taking the union of $\Sigma_{\delta,L}^{(1)}$ on $\delta, L \in I$ and by taking (\ref{eq: 2.8 Filippo page 15}), (\ref{eq: 2.9 Filippo page 15}) into account we thus find that 
\begin{align*}
[b_E]\leq \frac{[v]}{2}\quad \mathcal{H}^{n-2}\textit{\emph{-a.e. on $\{v\Low >0  \}\cup \{v\Upp< \infty  \}$}}.
\end{align*}
Since, by \cite[4.5.9(3)]{Fed69} $\{ v\Upp =\infty \}$ is $\mathcal{H}^{n-2}$-negligible, we have proved (\ref{eq:1.17 pag 7 Filippo}).\\ 
\textbf{Step 4} In this step we prove (\ref{eq:1.16 pag 7 Filippo}). Let $\delta, L \in I$ and $M\in J_{\delta}$. Since $b_{\delta, L, M}= b_{E}$ $\h^{n-1}$-a.e. on $\Sigma_{\delta, L, M}$ by (\ref{eq:3.52 pag 31 Filippo}) and by (\ref{eq: 2.12 FIlippo pag 16}) we find that for $\mathcal{H}^{n-1}$-a.e. $x\in \Sigma_{\delta, L, M}$, there exists $z(x) \in \partial K^s$ s.t.
\begin{align*}
\left\{\left(-\frac{1}{2}\nabla v(x)+t \nabla b_{E}(x),1  \right):\,t\in [-1,1]   \right\} \subset C^*_{K^s}(z(x)).
\end{align*}
By taking a union first on $M\in J_{\delta}$ and then on $\delta, L \in I$, we find that for $\mathcal{H}^{n-1}$-a.e. $x\in \{v>0 \}$, there exists $z(x) \in \partial K^s$ s.t.
\begin{align*}
\left\{\left(-\frac{1}{2}\nabla v(x)+t \nabla b_{E}(x),1  \right):\,t\in [-1,1]   \right\} \subset C^*_{K^s}(z(x)).
\end{align*}
This concludes the proof of (\ref{eq:1.16 pag 7 Filippo}).

% Furthermore, let us observe that by definition (\ref{def:barycenter of E}), $b_E=0$ on $\{ v=0 \}$. Thus, by (\ref{eq: 2.12 FIlippo pag 16}), we have that $\nabla b_E=0$ $\mathcal{H}^{n-1}$-a.e. on $\{ v=0 \}$. Thanks to these considerations, and having in mind relation (\ref{eq: unione coni = Rn}), we deduce that for $\mathcal{H}^{n-1}$-a.e. $x\in \R^{n-1}$, there exists $z(x) \in \partial K^s$ s.t.
%\begin{align*}
%\left\{\left(-\frac{1}{2}\nabla v(x)+t \nabla b_{E}(x),1  \right):\,t\in [-1,1]   \right\} \subset C^*_{K^s}(z(x)).
%\end{align*}

\noindent
\textbf{Step 5 }In this step we prove (\ref{eq:1.18 pag 7 Filippo}). Let $\delta, L\in I$ and $M\in J_{\delta}$. Since $b_{\delta,L,M}= 1_{\Sigma_{\delta,L,M}} \tau_Mb_{\delta}$, by Lemma \ref{lem: 2.3 FIlippo page 18} we have 
\begin{align*}
D^cb_{\delta,L,M}=D^c(\tau_Mb_{\delta})\mres\Sigma_{\delta,L,M}^{(1)}.
\end{align*}
Combining this fact with (\ref{eq:3.53 page 31 Filippo}) we find that for every $G\subset \Sigma_{\delta,L,M}^{(1)}$, for $|D^cv|$-a.e. $x\in G$ there exists $z(x)\in \partial K$ s.t.
\begin{align*}
\left\{(h(x)+tg_{\delta,M}(x),0):\,t\in[-1,1]\right\}\subset C^*_{K}(z(x)),
\end{align*}
where for $\left|D^c v \right|$-a.e. $x\in G$ the functions $g_{\delta,M}$ and $h$ are given by
\begin{align*}
g_{\delta,M}(x)= \frac{dD^c (\tau_Mb_{\delta})}{d\left|(D^c v/2,0) \right|_{K^s}}(x),\quad
h(x)= \frac{-dD^c v/2 }{d\left|(D^c v/2,0) \right|_{K^s}}(x).
\end{align*}
Observe now that 
\begin{align*}
\bigcup_{L\in I}\Sigma_{\delta,L,M}^{(1)}&=\bigcup_{L\in I}\{ |b_{\delta}|<M \}^{(1)}\cap \{ v>\delta \}^{(1)}\cap \{ v<L \}^{(1)}\\
&=\left(\{ |b_{\delta}|<M \}^{(1)}\cap \{ v>\delta \}^{(1)}\right)\cap \bigcup_{L\in I} \{ v<L \}^{(1)}\\
&=\{ |b_{\delta}|<M \}^{(1)}\cap \{ v>\delta \}^{(1)} \cap \{ v\Upp<\infty \},
\end{align*}
where in the last identity we used (\ref{eq: 2.8 Filippo page 15}). Note that, as we pointed out at the end of step 3, $\mathcal{H}^{n-2}(\{ v\Upp =\infty \})=0$, so the set $\{ v\Upp = \infty  \}$ is negligible with respect to both $|D^c\tau_Mb_{\delta}| $ and $|D^cv|$. Thus, we proved that for every bounded Borel set $G\subset \{ |b_{\delta}|<M \}^{(1)}\cap \{ v>\delta \}^{(1)}$, for $|D^cv|$-a.e. $x\in G$ there exists $z(x)\in \partial K$ s.t.
\begin{align}\label{eq: 3.59 Filippos delta e M}
\left\{(h(x)+tg_{\delta,M}(x),0):\,t\in[-1,1]\right\}\subset C^*_{K}(z(x)).
\end{align}
Observe that for every $M'>M$ and $\delta' < \delta$ we have that $\tau_Mb_\delta=\tau_{M'}b_{\delta'}$ on $\{ |b_{\delta}|<M \}\cap \{ v>\delta \}$. So, by Lemma \ref{lem: 2.3 FIlippo page 18} we get that
\begin{align*}
D^c\left( \tau_{M}b_{\delta} \right)\mres \{ |b_{\delta}|<M \}^{(1)}\cap \{ v>\delta \}^{(1)}  =D^c\left(\tau_{M'}b_{\delta'}\right)\mres \{ |b_{\delta}|<M \}^{(1)}\cap \{ v>\delta \}^{(1)},
\end{align*}
and therefore the function $g_{\delta,M}$ actually does not depend on $\delta, M$.  So taking into account (\ref{eq: 3.59 Filippos delta e M}) we have that for $|D^cv|$-a.e. $x\in G$ there exists $z(x)\in \partial K$ s.t.
\begin{align}
\left\{(h(x)+tg(x),0):\,t\in[-1,1]\right\}\subset C^*_{K}(z(x)).
\end{align}
Lastly, let us notice that
\begin{align*}
\tau_Mb_{\delta}=M1_{\{b_\delta\geq M\}}-M1_{\{b_\delta\leq -M\}}+ 1_{\{|b_\delta|<M\}\cap \{v>\delta\}}\tau_Mb_{\delta},\quad \textnormal{on }\R^{n-1}
\end{align*}
is an identity between BV functions. Thus, thanks to \cite[Example 3.97]{AFP} we find that 
\begin{align*}
D^c\tau_Mb_{\delta}= D^c(\tau_Mb_{\delta})\mres\left(G\cap \{ |b_{\delta}|<M \}^{(1)}\cap \{ v>\delta \}^{(1)}  \right)
\end{align*}
i.e. the measure $D^c\tau_Mb_{\delta}$ is concentrated on $\{ |b_{\delta}|<M \}^{(1)}\cap \{ v>\delta \}^{(1)}$.
Therefore, we deduce that for every bounded Borel set $G\subset \R^{n-1}$, for $|D^cv|$-a.e. $x\in G\cap \{ |b_{\delta}|<M \}^{(1)}\cap \{ v>\delta \}^{(1)}$ there exists $z(x)\in \partial K$ s.t.
\begin{align}
\left\{(h(x)+tg(x),0):\,t\in[-1,1]\right\}\subset C^*_{K}(z(x)).
\end{align}
\end{proof}
\noindent
Before entering into the details of the proof for the \emph{sufficient conditions} part, we need a couple of technical results.
\begin{proposition}\label{prop:3.7 page 34 Filippo}
Let $K\subset \R^n$ be as in (\ref{HP per K})and let $v$ be as in (\ref{due tilde}). Then, if $E$ is a v-distributed set of finite perimeter with sections $E_z$ as segments $\h^{n-1}$-a.e on $\{v>0\}$ we have that
\begin{align}\label{eq:3.62 page 34 Filippo}
P_{K}(E;\{ v\Low=0 \}\times \R)=P_{K}(F[v];\{ v\Low=0 \}\times \R)= \int_{\{ v\Low=0 \}}v\Upp \phi_{K}(-\nu_v,0)d\mathcal{H}^{n-2}.
\end{align}
\end{proposition}
\begin{proof}
The proof of this result follows from a careful inspection of the proof of \cite[Proposition 3.8]{CCPMSteiner}, and for this reason is omitted.  
\end{proof}
\begin{lemma}\label{lem:3.8 page 37 Filippo}
If $v\in (BV\cap L^{\infty})(\R^{n-1})$, $b:\R^{n-1}\rightarrow \R$ is such that $\tau_Mb \in (BV\cap L^{\infty})(\R^{n-1})$ for a.e. $M>0$ and $\mu$ is a $\R^{n-1}$-valued Radon measure such that 
\begin{align}
\lim_{M\rightarrow \infty}|\mu -D^c\tau_Mb|(G)=0\quad \textit{ for every bounded Borel set $G\subset \R^{n-1}$},\label{eq:3.79 page 37 FIlippo}
\end{align}
then,
\begin{align}
|(D^c(b+v),0)|_{K^s}(G)\leq |(\mu+D^cv),0)|_{K^s}(G)\quad \textit{ for every bounded Borel set $G\subset \R^{n-1}$}\label{eq: 3.80 page 37 Filippo}.
\end{align}
\end{lemma}
\begin{proof}
Let $L>0$ be such that $|v|\leq L$ $\mathcal{H}^{n-1}$-a.e. on $\R^{n-1}$. If $f \in BV(\R^{n-1}),$ then 
\begin{align*}
\tau_Mf=M1_{\{f>M\}}-M1_{\{f<-M\}}+1_{\{|f|<M\}}f \in (BV \cap L^{\infty})(\R^{n-1}),
\end{align*}
for every $M$ such that $\{ f>M \}$ and $\{ f<-M \}$ are of finite perimeter and thus, by \cite[Theorem 3.96]{AFP}
\begin{align*}
D^c\tau_Mf=D^c\left(1_{\{|f|<M\}}f  \right)=1_{\{|f|<M\}^{(1)}}D^cf=D^cf \mres \{|f|<M\}^{(1)};
\end{align*} 
in particular,
\begin{align}
|(D^c\tau_Mf,0)|_{K^s}&=|(D^cf,0)|_{K^s}\mres \{|f|<M  \}^{(1)}\leq |(D^cf,0)|_{K^s}.\label{eq:3.81 anisotropa}
\end{align}
From the equality $\tau_M(\tau_{M+L}(b)+v)=\tau_M(b+v)$ and from (\ref{eq:3.81 anisotropa}) applied with $f=\tau_{M+L}(b)+v$ it follows that, for every Borel set $G\subset \R^{n-1}$,
\begin{align}
|(D^c(\tau_M(b+v)),0)|_{K^s}(G)&= |(D^c(\tau_M(\tau_{M+L}(b)+v)),0)|_{K^s}(G)\nonumber \\&\leq |(D^c(\tau_{M+L}(b)+v),0)|_{K^s}(G)\label{eq:3.82 anisotropa}.
\end{align}
Now observe that (\ref{eq:3.79 page 37 FIlippo}) implies that 
\begin{align}
\lim_{M\rightarrow \infty}|-\left(\mu -D^c\tau_Mb\right)|(G)=0\quad \textit{ for every bounded Borel set $G\subset \R^{n-1}$}.\label{eq:3.79 page 37 FIlippo col meno}
\end{align}
Thanks to Remark \ref{rem: |mu|_K << |mu|} together with (\ref{eq:3.79 page 37 FIlippo}) and (\ref{eq:3.79 page 37 FIlippo col meno}), for every bounded Borel set $G\subset \R^{n-1}$ we get 
\begin{align}\label{eq:3.79 anisotropa}
\lim_{M\rightarrow \infty}|-(\mu -D^c\tau_Mb,0)|_{K^s}(G)=\lim_{M\rightarrow \infty}|(\mu -D^c\tau_Mb,0)|_{K^s}(G)=0.
\end{align}
Since we can always write
$
D^c\left(\tau_{M}b\right)+D^cv=\left(D^c\left(\tau_{M}b\right)-\mu  \right) + \left( \mu+D^cv \right)
$ by applying relations (\ref{eq: disuguaglianza triangola |mu|K}) and (\ref{eq: disuguaglianza triangola inversa |mu|K}) we obtain
\begin{align}
|\left(\mu+D^cv,0\right)|_{K^s}(G)-|-\left(D^c\left(\tau_{M+L}b\right)-\mu,0\right)|_{K^s}(G)\leq|\left(D^c\left(\tau_{M+L}b\right)+D^cv,0\right)|_{K^s}(G)\\\leq |\left(D^c\left(\tau_{M+L}b\right)-\mu,0\right)|_{K^s}(G)+|\left(\mu+D^cv,0\right)|_{K^s}(G).
\end{align}
So, by (\ref{eq:3.79 anisotropa}) we get 
\begin{align*}
\lim_{M\rightarrow \infty}|(D^c(\tau_{M+L}(b)+v),0)|_{K^s}(G)=|(\mu+D^cv,0)|_{K^s}(G).
\end{align*}
By (\ref{eq:3.82 anisotropa}), we get that 
\begin{align*}
\limsup_{M\to\infty}|(D^c(\tau_M(b+v)),0)|_{K^s}(G)\leq|(\mu+D^cv,0)|_{K^s}(G), 
\end{align*}
so that using (\ref{def: Cantor part GBV anisotropica}) we conclude the proof. 
\end{proof}

\begin{proof}[Proof of Theorem \ref{thm:2.2 pag 117}: sufficient conditions.]
Let $E$ be a $v$-distributed set of finite perimeter satisfying (\ref{eq: Ez segment}), (\ref{eq:1.16 pag 7 Filippo}), (\ref{eq:1.17 pag 7 Filippo}) and (\ref{eq:1.18 pag 7 Filippo}). Let $I$ and $J_{\delta}$ be defined as in (\ref{eq:3.50 pag 31 Filippo}) and (\ref{eq:3.51 pag 31 Filippo}). Let $\delta, S\in I$ and let us set $b_{\delta,S}=1_{\{\delta <v < S  \}}b_E=1_{\{\delta <v < S  \}}b_{\delta}$. Then, for every $M\in J_{\delta}$, we have $\tau_M b_{\delta}\in (BV\cap L^{\infty})(\R^{n-1})$ and so we obtain that $\tau_{M}b_{\delta,S}\in (BV\cap L^{\infty})(\R^{n-1}).$ Let us consider the $\R^{n-1}$-valued Radon measure $\mu_{\delta, S}$ on $\R^{n-1}$ defined as
\begin{align*}
\mu_{\delta, S}(G)=\int_{G\cap \{\delta< v < S  \}^{(1)}\cap \{|b_E|\Upp < \infty  \}}g(x)d\left|\left(\frac{1}{2} D^cv,0\right)\right|_{K^s},
\end{align*}
for every bounded Borel set $G\subset \R^{n-1}$, where $g(x)$ is the function that appears in condition (\ref{eq:1.18 pag 7 Filippo}), namely 
\begin{align*}
D^c(\tau_M(b_\delta))(G)=\int_{G\cap\{ |b_{\delta}|<M \}^{(1)}\cap \{ v>\delta \}^{(1)} }g(x)d\left|\left(\frac{1}{2} D^cv,0\right)\right|_{K^s}.
\end{align*}
Since $\tau_M b_{\delta,S}= 1_{\{v< S \}}\tau_Mb_{\delta}$, by Lemma \ref{lem: 2.3 FIlippo page 18} we have $D^c(\tau_Mb_{\delta,S})=1_{\{ v< S \}^{(1)}}D^c(\tau_Mb_{\delta})$ and thus, for every Borel set $G\subset \R^{n-1}$,
\begin{align*}
\lim_{M\rightarrow \infty}|\mu_{\delta,S}-D^c(\tau_Mb_{\delta,S})|(G)&=\lim_{M\rightarrow \infty}|\mu_{\delta,S}-D^c(\tau_Mb_{\delta})|(G\cap \{v<S \}^{(1)})\\
&\leq \lim_{M\rightarrow \infty}\int_{G\cap \{\delta< v < S  \}^{(1)}\cap \left(\{|b_E|\Upp < \infty  \}\setminus \{|b_E| < M \}^{(1)}\right) }\!\!\!\!\!\!\!\!\!\!\!\!\!\!\!\!\!\!\!\!\!\!\!\!\!\!\!\!\!\!\!\!\!\!\!|g(x)| d|(D^cv/2,0)|_{K^s}(x)\\
&=0,
\end{align*}
where the last equality follows from the fact that $\{|b_E| < M \}^{(1)}_{M\in I}$ is an increasing family of sets whose union is $\{|b_E|\Upp < \infty  \}$.
Thus, for every bounded Borel set $G\subset \R^{n-1}$, we get
\begin{small}
\begin{align}\label{eq:3.83 page 38 Filippo}
&\left|\left(-D^c(b_{\delta,S}+\frac{1}{2}
v_{\delta,S}),0\right)   \right|_{K^s}(G)+\left|\left(D^c(b_{\delta,S}-\frac{1}{2}v_{\delta,S}),0\right)   \right|_{K^s}(G) \nonumber\\&\leq  \left| \left(-\mu_{\delta,S}-\frac{1}{2}D^cv_{\delta,S},0 \right)\right|_{K^s}(G) 
+\left|\left(\mu_{\delta,S}-\frac{1}{2}D^cv_{\delta,S},0\right)   \right|_{K^s}(G)\nonumber\\& =\left|\left(-D^cv_{\delta,S}),0\right)\right|_{K^s}(G),
\end{align}
\end{small}
where the first inequality comes from  Lemma \ref{lem:3.8 page 37 Filippo} applied to $b_{\delta,S}-\frac{1}{2}v_{\delta,S}$ and $-b_{\delta,S}-\frac{1}{2}v_{\delta,S}$ with $v_{\delta,S}= 1_{\{\delta< v< S\}}v)$, (see in particular (\ref{eq: 3.80 page 37 Filippo})), whereas the equality is a consequence of Lemma \ref{lem:1.9 pag 70} applied to the two Radon measures $\mu_{\delta,S}-\frac{1}{2}D^cv_{\delta,S}$ and $-\mu_{\delta,S}-\frac{1}{2}D^cv_{\delta,S}$ together with Remark \ref{rem: case equality radon measure} having in mind (\ref{eq:1.18 pag 7 Filippo}). Since $b_{\delta,S}\in GBV(\R^{n-1})$ and $v_{\delta,S}\in (BV\cap L^{\infty})(\R^{n-1})$, if $W=W[v_{\delta,S},b_{\delta,S}]$, then we can compute $P_{K^s}(W; G\times \R)$ for every Borel set $G\subset \R^{n-1}$ by Corollary \ref{cor: formula perimetro con b v}. In particular, if $G\subset \{\delta < v <S  \}^{(1)}$, then by $E\cap (\{\delta< v < S  \}\times \R)=W\cap (\{\delta< v < S  \}\times \R) $, we find that
\begin{align}
&P_{K^s}(E; G\times \R)= P_{K^s}(W; G\times \R)\\
&=\int_{G}\phi_{K^s}\left(\nabla\left(b_{\delta,S}-\frac{v_{\delta,S}}{2}\right),1\right)+ \phi_{K^s}\left(-\nabla\left(b_{\delta,S}+\frac{v_{\delta,S}}{2}\right),1\right)d\mathcal{H}^{n-1}\label{eq:3.84 page 38 Filippo}\\
&+\int_{G\cap J_v}\min\left(v_{\delta,S}\Upp,\left(\left[\frac{v_{\delta,S}}{2}\right]+[b_{\delta,S}]+\max\left(\left[\frac{v_{\delta,S}}{2}\right]- [b_{\delta,S}],0  \right)   \right)   \right)\phi_{K^s}(-\nu_v,0) d\mathcal{H}^{n-2}\label{eq:3.85a page 38 Filippo}\\
&+\int_{G\cap J_v}\min\left(v_{\delta,S}\Low,\max\left( 0,[b_{\delta,S}]-\left[\frac{v_{\delta,S}}{2}\right] \right)\right)\phi_{K^s}(\nu_v,0) d\mathcal{H}^{n-2}\label{eq:3.85b page 38 Filippo}\\
&+\int_{G\cap (J_b\setminus J_v )}\min\left([b_{\delta,S}],\tilde{v}  \right)\left(\phi_{K^s}(-\nu_b,0)+\phi_{K^s}(\nu_b,0)  \right) d\mathcal{H}^{n-2}\label{eq:3.85c page 38 Filippo}\\
&+\left|\left(D^c\left( b_{\delta,S}-\frac{v_{\delta,S}}{2}  \right),0     \right)   \right|_{K^s}(G)\label{eq:3.86a page 38 Filippo}\\
&+\left|\left(-D^c\left( b_{\delta,S}+\frac{v_{\delta,S}}{2}  \right),0     \right)   \right|_{K^s}(G)\label{eq:3.86b page 38 Filippo}
\end{align}
We can also compute $P_{K^s}(F[v_{\delta,S}];G\times \R)$. Taking also into account that $F[v]\cap (\{\delta<v<S \}\times \R)=F[v_{\delta,S}]\cap (\{\delta<v<S \}\times \R)$ we obtain that 
\begin{align*}
P_{K^s}(F[v];G\times \R)&=P_{K^s}(F[v_{\delta,S}];G\times \R)= 2\int_{G}\phi_{K^s}\left(-\nabla\left(\frac{v_{\delta,S}}{2}\right),1\right)d\mathcal{H}^{n-1}\\
&+\int_{G\cap J_{v_{\delta,S}}}[v]\phi_{K^s}(-\nu_v,0) d\mathcal{H}^{n-2}+2\int_{G}\phi_{K^s}\left(-\frac{dD^c\left(\frac{v_{\delta,S}}{2}\right)}{d\left|D^c\left(\frac{v_{\delta,S}}{2}\right)\right|},0  \right) d\left|D^c\left(\frac{v_{\delta,S}}{2}\right)\right|.
\end{align*}
Firstly, applying (\ref{eq: 2.12 FIlippo pag 16}) to $b_E$ and (\ref{eq: 2.10 Filippo pag 16}) to $v$ we get 
\begin{align*}
&\nabla b_{\delta,S}(x)=\nabla b_E(x),\quad  \textit{for $\mathcal{H}^{n-1}$-a.e. $x\in\{\delta<v<S  \}$},\\
&[v]=[v_{\delta,S}] ,\quad  \textit{for $\mathcal{H}^{n-2}$-a.e. on $\{\delta<v<S  \}^{(1)}$}.
\end{align*}
Putting together the above relations with the assumptions (\ref{eq:1.16 pag 7 Filippo}) and (\ref{eq:1.17 pag 7 Filippo}) we deduce that, for $\mathcal{H}^{n-1}$-a.e. $x\in\{\delta<v<S  \}$ there exists $z(x)\in\partial K^s$ s.t.
\begin{align}\label{eq:3.87 page38 Filippo}
\left\{\left(-\frac{1}{2}\nabla v(x)+t \nabla b_E(x),1  \right):\,t\in [-1,1]   \right\} \subset C^*_{K^s}(z(x)),
\end{align}
\begin{align}
2[b_{\delta,S}]=2[b_E]\leq [v]=[v_{\delta,S}] ,\quad  &\textit{for $\mathcal{H}^{n-2}$-a.e. on $\{\delta<v<S  \}^{(1)}$}\label{eq:3.88 page38 Filippo}.
\end{align}
Thanks to Proposition \ref{prop:linearityphi} and Remark \ref{rem: more appealing way (iii)}, condition (\ref{eq:3.87 page38 Filippo}) is equivalent to say that we can rewrite (\ref{eq:3.84 page 38 Filippo}) in the following way
\begin{align}
&\int_{G}\phi_{K^s}\left(\nabla\left(b_{\delta,S}-\frac{v_{\delta,S}}{2}\right),1\right)+ \phi_{K^s}\left(-\nabla\left(b_{\delta,S}+\frac{v_{\delta,S}}{2}\right),1\right)d\mathcal{H}^{n-1}\nonumber\\
&\int_{G}\phi_{K^s}\left(\nabla\left(b_{E}-\frac{v}{2}\right),1\right)+ \phi_{K^s}\left(-\nabla\left(b_{E}+\frac{v}{2}\right),1\right)d\mathcal{H}^{n-1}\nonumber\\
&=2\int_{G}\phi_{K^s}\left(-\nabla\left(\frac{v}{2}\right),1\right)d\mathcal{H}^{n-1}.
\end{align}
Furthermore, substituting (\ref{eq:3.88 page38 Filippo}) into (\ref{eq:3.85a page 38 Filippo}),(\ref{eq:3.85b page 38 Filippo}) and (\ref{eq:3.85c page 38 Filippo}), and using (\ref{eq:3.83 page 38 Filippo}) applied to (\ref{eq:3.86a page 38 Filippo}) and (\ref{eq:3.86b page 38 Filippo}), we find that
\begin{align}\label{eq:3.89 page 38 Filippo}
P_{K^s}(E;\{\delta<v< S  \}^{(1)}\times \R)\leq P_{K^s}(F[v];\{\delta<v< S  \}^{(1)}\times \R),
\end{align}
where, actually, equality holds thanks to  (\ref{eq:anisotropic Steiner inequality}). Recalling that by \cite[4.5.9(3)]{Fed69} we have that $\mathcal{H}^{n-2}\left(\{v\Upp =\infty  \}\right)=0$, thanks to (\ref{eq: 2.9 Filippo page 15}) it follows that 
\begin{align}\label{eq:3.90 page 38 Filippo}
\bigcup_{M\in I}\{v<M  \}^{(1)}= \{v\Upp <\infty  \}=_{\mathcal{H}^{n-2}}\R^{n-1}.
\end{align}
By (\ref{eq: 2.9 Filippo page 15}) if we  consider the sequences $\delta_h\in I$ and $S_h \in I$ such that $\delta_h \rightarrow 0$ and $S_h\rightarrow 0$ as $h\rightarrow \infty$ we get
\begin{align*}
\{v\Upp>0\}=\bigcup_{h\in\mathbb{N}}\{\delta_h<v\Upp<S_h\}^{(1)}.
\end{align*}
So, by the above relation together with  (\ref{eq:3.89 page 38 Filippo}),  and (\ref{eq:3.90 page 38 Filippo}) we get that  
\begin{align*}
P_{K^s}(E;\{ v\Low>0 \}\times \R)\leq P_{K^s}(F[v];\{ v\Low>0 \}\times \R).
\end{align*}
By Proposition \ref{prop:3.7 page 34 Filippo} $P_{K^s}(E;\{ v\Low=0 \}\times \R)=P_{K^s}(F[v];\{ v\Low=0 \}\times \R)$ and thus $P_{K^s}(E)=P_{K^s}(F[v])$. This concludes  the proof.
\end{proof}
\section{Rigidity of the Steiner's inequality for the anisotropic perimeter}\label{section Rigidity}
In this final section we will prove the main results about (\ref{rigidity anisotropic steiner}). Let us start the section with the proof of Theorem \ref{thm: rigidity}.

\begin{proof}[(Proof of Theorem \ref{thm: rigidity})]
By Theorem \ref{thm:1.2 Filippo} we have to prove that conditions (\ref{eq: 1.16 Filippo})-(\ref{eq: 1.18Dc Filippo}) hold true. We divide the proof in few steps.\\
\textbf{Step 1} In this step we prove that (\ref{eq: 1.16 Filippo}) holds true. Since $E\in\mathcal{M}_{K^s}(v)$, by Theorem \ref{thm:2.2 pag 117} we have that condition (\ref{eq:1.16 pag 7 Filippo}) holds true, namely for $\mathcal{H}^{n-1}$-a.e. $x\in \{v>0\}$ there exists $z(x)\in\partial K^s$ s.t.
\begin{align*}
\left(-\frac{1}{2}\nabla v(x)+t \nabla b_E(x),1  \right)  \in C^*_{K^s}(z(x))\quad \forall\, t\in [-1,1]. 
\end{align*}
By condition \textbf{R1} we have that for $\mathcal{H}^{n-1}$-a.e. $x\in\{v>0\}$ there exists $z(x)\in\partial K^s$ s.t. $\forall\,t\in[-1,1]$ there exists $\lambda=\lambda(t,x)\in [0,1]$ such that
$$(t \nabla b_E(x),0)=\lambda\left(-\frac{1}{2}\nabla v(x),1  \right).$$
that implies $\nabla b_E=0$ for $\mathcal{H}^{n-1}$-a.e. $x\in\{v>0\}$, that implies $\nabla b_E=0$ for $\mathcal{H}^{n-1}$-a.e. $x\in\R^{n-1}$.
\\
\textbf{Step 2} In this step we prove that (\ref{eq: 1.18Dc Filippo}) holds true. Again, since $E\in\mathcal{M}_{K^s}(v)$ we know that condition (\ref{eq:1.18 pag 7 Filippo}) holds true, namely we know that
for $|D^cv|$-a.e. $x\in\{v\Low>0\}$ there exists $z(x)\in \partial K$ s.t.
\begin{align}\label{eq: Step 2 thm 8.1}
(h(x)+tg(x),0)\in C^*_{K^s}(z(x)),\quad \forall\, t\in[-1,1].
\end{align}
%where
%\begin{align*}
%g(x)&= \frac{dD^c (\tau_Mb_{\delta})}{d\left|(D^c v/2,0) \right|_{K^s}}(x),\\
%h(x)&= \frac{-dD^c v/2 }{d\left|(D^c v/2,0) \right|_{K^s}}(x).
%\end{align*}
So, by condition \textbf{R2} we know that for $|D^cv|$-a.e. $x\in\{v\Low>0\}$ there exists $\lambda=\lambda(x)\in[-1,1]$ such that $g(x)=\lambda h(x)$. By definition of $g(x)$ and $h(x)$, for every Borel set $G\subset\R^{n-1}$, every $M>0$, and $\mathcal{H}^1$-a.e. $\delta >0$ we have
\begin{align*}
D^c(\tau_M(b_{\delta}))(G)&=\int_{G\cap \{ |b_{\delta}|<M \}^{(1)}\cap \{ v>\delta \}^{(1)}}g(x)d\left|\left(\frac{1}{2}D^cv,0  \right)\right|_{K^s}(x)\\&=\int_{G\cap \{ |b_{\delta}|<M \}^{(1)}\cap \{ v>\delta \}^{(1)}}\lambda(x)h(x)d\left|\left(\frac{1}{2}D^cv,0  \right)\right|_{K^s}(x)\\&=\int_{G\cap \{ |b_{\delta}|<M \}^{(1)}\cap \{ v>\delta \}^{(1)}}-\frac{1}{2}\lambda(x)dD^cv(x).
\end{align*}
Since $-\frac{1}{2}\lambda(x)\in [-1/2,1/2]$ for $|D^cv|$-a.e. $x\in \{v\Low>0  \}$, we conclude the proof of step 2.\\
\textbf{Step 3} In this step we prove that (\ref{eq: 1.19 Filippo}) and (\ref{eq: 1.20 Filippo}) holds true. By step 2 we have that (\ref{eq: 1.18Dc Filippo}) holds true. By taking the total variation in (\ref{eq: 1.18Dc Filippo}) we find that $2|D^c(\tau_M(b_\delta))|(G)\leq |D^cv|(G)$ for every bounded Borel set $G\subset \R^{n-1}$. By passing to the limit for $M\rightarrow +\infty$ (in $J_{\delta}$) and then $\delta\rightarrow 0$ (in $I$) we prove (\ref{eq: 1.19 Filippo}).  As observed in \cite[Remark 1.10]{CCPMSteiner}, note that (\ref{eq: 1.20 Filippo}) is a consequence of (\ref{eq: 1.12 Filippo pag 7}), taking into account (\ref{eq: 1.16 Filippo}), (\ref{eq: 1.18Dc Filippo}) and (\ref{eq: 1.19 Filippo}). This concludes the proof.
\end{proof}
\noindent
The following result provides a geometrical characterization of the validity of \textbf{R1} and \textbf{R2}. In the following, given any set $G\subset\R^n$ we denote by $\overline{G}$ its topological closure. Having in mind the definitions of exposed and extreme points (see Definition \ref{def: exposed point} and \ref{def: extreme point} respectively), we can now prove the following proposition, that will be an important intermediate result in order to prove Proposition \ref{prop: R1,R2}.

\begin{proposition}\label{prop: R1 e R2 sono estremi}
Let $v$ be as in (\ref{due tilde}) and let $K\subset\R^n$ be as in (\ref{HP per K}). For $\h^{n-1}$-a.e. $x\in \{v>0\}$ let us call $\nu(x)=\left(-\frac{1}{2}\nabla v(x),1  \right)$. Then,
\begin{align}\label{R1 extreme}
\textbf{R1} \textit{ holds true }\quad \Longleftrightarrow\quad &\frac{\nu(x)}{\phi_{K^s}\left( \nu(x) \right)} \textit{ is an extreme point of $\overline{(K^s)^*}$} \\&\textit{for $\h^{n-1}$-a.e. $x\in\{v>0\}$}.\nonumber
\end{align}
\begin{align}\label{R2 extreme}
\textbf{R2}\textit{ holds true}\quad \Longleftrightarrow \quad &\frac{(h(x),0)}{\phi_{K^s}\left(( h(x),0) \right)} \textit{ is an extreme point of $\overline{(K^s)^*}$}\\ &\textit{ for $|D^cv|$-a.e. $x\in\{ v\Low>0\}$,}\nonumber
\end{align} 
where $h$ has been defined in (\ref{h}).
\end{proposition}
\begin{proof}
Let us prove that (\ref{R1 extreme}) holds true, then statement (\ref{R2 extreme}) follows using an identical argument.\\
\textbf{Step 1 } Let us assume that \textbf{R1} holds true and suppose by contradiction that there exists $G\subset \{v>0\}$ such that $\h^{n-1}(G)>0$ and $\nu(x)/\phi_{K^s}(\nu(x))$ is not an extreme point for every $x\in G$. Note that by construction, $\nu(x)/\phi_{K^s}(\nu(x))\in \partial (K^s)^*$. So, if $\nu(x)/\phi_{K^s}(\nu(x))$ is not an extreme point, there exist $y(x),z(x)\in \partial (K^s)^*$ with $y(x)\neq z(x)$ and $\lambda(x)\in (0,1 )$ such that
$$
\frac{\nu(x)}{\phi_{K^s}(\nu(x))}= (1-\lambda(x))z(x)+\lambda(x)(y(x)).
$$ 
By convexity, of $(K^s)^*$, we have that
\begin{align*}
(1-\lambda)z(x)+\lambda y(x)\in \partial (K^s)^*\quad \forall\,\lambda\in[0,1],
\end{align*} 
and thus there exist $\bar{x}\in \partial(K^s)^*$, and $\omega\in\mathbb{S}^{n-1}$ such that 
\begin{align*}
(1-\lambda)y(x)+\lambda z(x) &\in \partial(K^s)^* \cap H_{\bar{x},\omega}\quad \forall\,\lambda\in [0,1],\\
(K^s)^* &\subset H^-_{\bar{x},\omega}.
\end{align*}
Thanks to the above relations, and by definition of $\phi^*_{K^s}$ we have that
\begin{align*}
\left((1-\lambda)y(x)+\lambda z(x)  \right)\cdot \omega = \phi_{K^s}^*(\omega)\quad \forall\, \lambda\in [0,1].
\end{align*}
Recalling (\ref{altra caratterizzazione del sub diff}) we get that 
\begin{align*}
(1-\lambda)y(x)+\lambda z(x) \in \partial\phi^*_{K^s}(\omega)\quad \forall\, \lambda\in [0,1],
\end{align*}
and thus, since $\omega\in Z_{K^s}(\nu(x)/\phi_{K^s}(\nu(x)))$, by Lemma \ref{lem: propedeutico lemma R1,R2} this implies that 
\begin{align*}
(1-\lambda)z(x)+\lambda y(x)\in \partial \phi^*_{K^s}(z)\quad \forall\,\lambda\in[0,1],\,\forall\, z\in\mathcal{Z}_{K^s}\left(\frac{\nu(x)}{\phi_{K^s}(\nu(x))}  \right).
\end{align*}
In particular, this implies that 
\begin{align}\label{eq: prop extreme R1}
(1-\lambda)\phi_{K^s}(\nu(x))z(x)+\lambda \phi_{K^s}(\nu(x))y(x)\in C^*_{K^s}(z)\quad \forall\,\lambda\in[0,1],\,\forall\, z\in\mathcal{Z}_{K^s}\left(\frac{\nu(x)}{\phi_{K^s}(\nu(x))}\right),
\end{align}
where recall that $\mathcal{Z}_{K^s}\left(\nu(x)/\phi_{K^s}(\nu(x))\right)=\mathcal{Z}_{K^s}\left(\nu(x)  \right)$.
Let us consider $z\in \mathcal{Z}_{K^s}(\nu(x))$. Applying the above formula with $\lambda(x)\in (0,1)$ we obtain that $\nu(x)$ belongs to the \emph{interior} of $C^*_{K^s}(z)$, that is there exists a radius $r>0$ such that $B(\nu(x),r)\subset C^*_{K^s}(z)$. Let us take $w\in B(\nu(x),r)$ such that $w\neq t\nu(x)$ for every $t\in \R$, and let us denote $\bar{w}= w-\nu(x)$. Then, 
\begin{align}
&\bar{w}\neq t \nu(x)\quad \forall\,t\in \R,\label{eq: contraddetto R1 0}\\
&\nu(x)+ \bar{w}\in C^*_{K^s}(z),\label{eq: contraddetto R1 1}\\
&\nu(x)- \bar{w}\in C^*_{K^s}(z).\label{eq: contraddetto R1 2}
\end{align}
Relation (\ref{eq: contraddetto R1 0}) is true since $w\neq t \nu(x)$ for every $t \in \R$. From the choice of $w\in B(\nu(x),r)$ we get that $\nu(x)+\bar{w}=w\in B(\nu(x,r))\subset C^*_{K^s}(z)$. On the other hand, $\nu(x)-\bar{w}=2\nu(x)-w$. In order to prove that $2\nu(x)-w\in B(\nu(x),r)$ let us check if $|2\nu(x)-w-v|< r$. So, $|2\nu(x)-w-v|=|\nu(x)-w|= |\bar{w}|< r$ since $w\in B(\nu(x,r))$. Thus, since (\ref{eq: prop extreme R1}) holds true for $\mathcal{H}^{n-1}$-a.e. $x\in G$ and $\h^{n-1}(G)>0$, and having in mind (\ref{eq: contraddetto R1 0}), (\ref{eq: contraddetto R1 1}), and (\ref{eq: contraddetto R1 2}) we reached a contradiction with \textbf{R1}.\\
\textbf{Step 2 } Let us now assume that $\nu(x)/\phi_{K^s}\left( \nu(x) \right)$ is an extreme point of $\overline{(K^s)^*}$ for $\h^{n-1}$-a.e. $x\in\{ v>0 \}$, and suppose by contradiction that  \textbf{R1} is not verified, namely that there exists $y\in\R^n$, and $G\subset \{v>0\}$ with $\h^{n-1}(G)>0$ such that, for every $x\in G$ there exists $z\in\mathcal{Z}_{K^s}(\nu(x))$ such that,
\begin{align*}
\nu(x)\pm y \in C^*_{K^s}(z)\quad \textnormal{but}\quad y\neq\lambda \nu(x),\quad \textit{for every $\lambda\in [-1,1]$}.
\end{align*} 
In particular, by convexity, 
\begin{align*}
(1-\lambda) \left( \nu(x)+y \right)+\lambda\left( \nu(x)-y \right)\in  C^*_{K^s}(z),\quad \forall\, \lambda\in [0,1].
\end{align*}
But this implies that the projection of this segment over $\partial\phi^*_{K^s}(z)$ contains in its relative interior the point $\nu(x)/\phi_{K^s}(\nu(x))$, namely there exists $\lambda(x)\in (0,1)$ such that 
\begin{align}\label{eq: prop extreme 2}
\frac{\nu(x)}{\phi_{K^s}(\nu(x))}=(1-\lambda(x))\frac{\left( \nu(x)+y \right)}{\phi_{K^s}\left( \nu(x)+y \right)}+\lambda(x)\frac{\left( \nu(x)-y \right)}{\phi_{K^s}\left( \nu(x)-y \right)}.
\end{align}
Since (\ref{eq: prop extreme 2}) holds true for $\h^{n-1}$-a.e. $x\in G$ and $\h^{n-1}(G)>0$ we contradicted our assumptions. This concludes the proof.
\end{proof}
\noindent
As mentioned above, Proposition \ref{prop: R1 e R2 sono estremi} give a characterization of conditions \textbf{R1} and \textbf{R2} in terms of the geometric properties of the dual Wulff shape $(K^s)^*$ we are considering. Before the proof of Proposition \ref{prop: R1,R2}, we need the following lemma.

\begin{lemma}\label{lem: V_{K^s}}
Let $K\subset\R^n$ be as in (\ref{HP per K}), and consider $y\in\R^n$. Then, $y/\phi_{K^s}(y)$ is an extreme point of $\overline{({K^s})^*}$ if and only if $y/|y|\in \overline{\mathbb{V}_{K^s}}$, where $\mathbb{V}_{K^s}$ is the set defined in (\ref{V_{K^s}}).
\end{lemma}
\begin{proof}
\textbf{Step 1} We first prove the result for the exposed points of $\overline{(K^s)^*}$, namely we prove that $y/\phi_{K^s}(y)$ is an exposed point of $\overline{({K^s})^*}$ if and only if $y/|y|\in\mathbb{V}_{K^s}$. This first part is the direct consequence of Lemma \ref{lem: corollary 25.1.3 Rockafellar} using $g=\phi_{K^s}^*$ and observing that $\partial\phi_{K^s}^*(x)=\{\nu^{K^s}(x)/\phi_{K^s}(\nu^{K^s}(x))\}$ for every $x\in\partial^*K^s$.\\
\textbf{Step 2} We now conclude the proof. Let $y\in\R^n$ be such that $y/\phi_{K^s}(y)$ is an extreme point of $\overline{(K^s)^*}$, by Remark \ref{exposed dense in extreme}, it implies that there exists a sequence $(\omega_h)_{h\in \mathbb{N}}$ of exposed points of $\overline{(K^s)^*}$ such that $\lim_{h\to \infty}\omega_h=y/\phi_{K^s}(y)$. Observe that by definition, $\omega_h\in \partial (K^s)^*$, and so $\phi_{K^s}(\omega_h)=1$ for all $h\in\mathbb{N}$. Thanks to the first step, every $\omega_h$ is such that $\eta_h:=\omega_h/|\omega_h|\in \mathbb{V}_{K^s}$, and so $\phi_{K^s}(\eta_h)=1/|\omega_h|$. Moreover, the fact that $\omega_h$ is a converging sequence implies that there exists $\eta\in\mathbb{S}^{n-1}$ such that $\lim_{h\to \infty}\eta_h=\eta$. Thus, $\eta\in \overline{\mathbb{V}_{K^s}}$, and $y/\phi_{K^s}(y)=\eta/\phi_{K^s}(\eta)$. In particular, since $|\eta|=1$, we have that $\eta= y/|y|\in \overline{\mathbb{V}_{K^s}}$. The reverse implication follows by similar argument.
\end{proof}

\begin{corollary}\label{cor: new R1, R2 equivalence}
Let $v$ be as in (\ref{due tilde}) and let $K\subset\R^n$ be as in (\ref{HP per K}). Then,

\begin{align}\label{R1 extreme new}
\textbf{R1} \textit{ holds true }\quad \Longleftrightarrow\quad &\exists\, S_1\subset\{v\Low >0 \}\text{ such that } \mathcal{H}^{n-1}(S_1)=0,\text{ and }\nonumber\\
&\nu^{F[v]}\left(x,\frac{1}{2}v(x)\right)\in \overline{\mathbb{V}_{K^s}}\quad \forall\, x\in\{v\Low>0 \}\setminus S_1.
\end{align}
\begin{align}\label{R2 extreme new}
\textbf{R2} \textit{ holds true }\quad \Longleftrightarrow\quad &\exists\, S_2\subset\{v\Low >0 \}\text{ such that } |D^cv|(S_2)=0,\text{ and }\nonumber\\
&\nu^{F[v]}\left(x,\frac{1}{2}v(x)\right)\in \overline{\mathbb{V}_{K^s}}\quad \forall\, x\in\{v\Low>0 \}\setminus S_2.
\end{align}
\end{corollary}

\begin{proof}
We prove (\ref{R1 extreme new}), then (\ref{R2 extreme new}) follows by similar argument. Thanks to (\ref{R1 extreme}) we have just to prove that 
\begin{align}
&\frac{\nu(x)}{\phi_{K^s}\left( \nu(x) \right)} \textit{ is an extreme point of $\overline{(K^s)^*}$}\quad \Longleftrightarrow &&\exists\, S_1\subset\{v\Low >0 \}\text{ s.t. } \mathcal{H}^{n-1}(S_1)=0,\,\text{and}\nonumber\\&\textit{for $\h^{n-1}$-a.e. $x\in\{v>0\}$}  &&\nu^{F[v]}\left(x,\frac{1}{2}v(x)\right)\in \overline{\mathbb{V}_{K^s}}\quad \forall\, x\in\{v\Low>0 \}\setminus S_1,
\end{align}
where $\nu(x)=\left(-\frac{1}{2}\nabla v(x),1  \right)$ for $\h^{n-1}$-a.e. $x\in \{v>0\}$. By Lemma \ref{lem: V_{K^s}} we have that $\nu(x)/\phi_{K^s}(\nu(x))$ is an extreme point if and only if $\nu(x)/|\nu(x)|\in \overline{\mathbb{V}_{K^s}}$ for $\mathcal{H}^{n-1}$-a.e. $x\in\{v>0\}$, that is equivalent to say that there exists $S_1\subset \{ v\Low>0\}$ s.t. $\mathcal{H}^{n-1}(S_1)=0$, and $\nu(x)/|\nu(x)|\in \overline{\mathbb{V}_{K^s}}$ for every $x\in \{v\Low>0\}\setminus S_1$. By Theorem \ref{thm:1.6 pag 57 note}, with $u=v/2$, we deduce that $\nu(x)/|\nu(x)|= \nu^{F[v]}\left(x,\frac{1}{2}v(x)\right)$ for $\mathcal{H}^{n-1}$-a.e. $x\in \{v\Low>0\}$, and thus we conclude.
\end{proof}

We are ready now to prove Proposition \ref{prop: R1,R2}.

\begin{proof}[Proof of Proposition \ref{prop: R1,R2}]
\textbf{Step 1} Suppose that \textbf{R1} and \textbf{R2} hold true. Then, by Corollary \ref{cor: new R1, R2 equivalence} there exist $S_1\subset \{v\Low>0 \}$, $S_2\subset\{v\Low>0 \}$ such that $\mathcal{H}^{n-1}(S_1)=|D^cv|(S_2)=0$, and 
$$\nu^{F[v]}\left(z,\frac{1}{2}v(z) \right)\in \overline{\mathbb{V}_{K_s}}\quad \forall\, x\in\{v\Low>0\}\setminus S_1 \cup \{v\Low>0\}\setminus S_2.$$
By De Morgan's laws, calling $S:= S_1\cap S_2$, the above relation is equivalent to
$$\nu^{F[v]}\left(z,\frac{1}{2}v(z) \right)\in \overline{\mathbb{V}_{K_s}}\quad \forall\, x\in\{v\Low>0\}\setminus S.$$ Since $S\subseteq S_i$, for $i=1,2$, we deduce  that $\mathcal{H}^{n-1}(S)=|D^cv|(S)=0$. This concludes the first step.\\
\textbf{Step 2} Suppose that $ii)$ of Proposition \ref{prop: R1,R2} holds true, namely that $\exists\, S\subset\{ v\Low>0 \}$ such that $\mathcal{H}^{n-1}(S)=|D^cv|(S)=0$, and 
\begin{align*}
\nu^{F[v]}\left(z,\frac{1}{2}v(z)\right) \in \overline{\mathbb{V}_{K^s}}\quad \forall\,z\in \{v\Low>0\}\setminus S.
\end{align*}
Thanks to Corollary \ref{cor: new R1, R2 equivalence} with $S_1=S_2=S$ we deduce that \textbf{R1} and \textbf{R2} hold true. This concludes the proof.
\end{proof}

\begin{proof}[Proof of Remark \ref{rem: post Prop 1.12 2}]
If $K^s$ is polyhedral, then $\mathbb{V}_{K^s}$ coincides with the set of the outer unit normals to the facets of $K^s$. Since $K^s$ has a finite number of facets, we conclude that $\mathbb{V}_{K^s}$ is closed. In case $K^s$ has $C^1$ boundary, we have to notice that thanks to \cite[Corollary 3, 
Theorem 1]{HanNi}), every point in $\partial (K^s)^*$ is an exposed point, so by Lemma \ref{lem: V_{K^s}} we have that $\mathbb{V}_{K^s}$ coincides with $\mathbb{S}^{n-1}$, which is closed.
\end{proof}

\begin{proof}[Proof of Corollary \ref{cor: smooth K^s}]
Thanks to Remark \ref{rem: post Prop 1.12 2} we know that if $K^s$ has $C^1$ boundary, then $\mathbb{V}_{K^s}$ coincides with $\mathbb{S}^{n-1}$. Therefore condition (\ref{eq lem R1,R2}), namely there exists $S\subset\{v\Low>0 \}$ s.t. $\mathcal{H}^{n-1}(S)=|D^cv|(S)=0$ and $\nu^{F[v]}\left(z,\frac{1}{2}v(z) \right)\in \overline{\mathbb{V}_{K^s}}=\mathbb{S}^{n-1}$ for every $z\in \{v\Low>0\}$, is always verified. This concludes the proof.  
\end{proof}

\section*{Acknowledgements}\noindent
I would like to thank my PhD advisor Filippo Cagnetti for the help he was always willing to give me. Lastly, but not the least, I would like to thank my PhD thesis examiners Miroslav Chleb\'{i}k and Nicola Fusco for their valuable comments and suggestions. This work was partially supported by the DFG project FR 4083/1-1 and by the Deutsche
Forschungsgemeinschaft (DFG, German Research Foundation) under Germany’s Excellence Strategy EXC 2044-390685587, Mathematics M\"unster: Dynamics-Geometry-Structure.

\def\cprime{$'$}

\end{document}